\documentclass{amsart}
\usepackage{graphicx}
\usepackage{pslatex}
\usepackage{amsmath,amsfonts}
\usepackage{rotating}
\usepackage{amssymb}
\usepackage{verbatim}
\usepackage{rotating}
\usepackage{color}
\usepackage{xcolor}
\usepackage{setspace}
\usepackage{mathrsfs}
\usepackage{hyperref}

\newtheorem{theorem}{Theorem}

\newtheorem{lemma}{Lemma}[section]
\newtheorem{proposition}{Proposition}[section]

\theoremstyle{definition}
\newtheorem{definition}{Definition}[section]

\theoremstyle{remark}
\newtheorem{remark}{Remark}[section]

\numberwithin{equation}{section}

\def\R{\mathbb R}
\def\C{\mathbb C}
\def\Z{\mathbb Z}
\def\N{\mathbb N}

\def\supp{\text{supp}}
\def\dist{\text{dist}}
\def\Arg{\text{Arg}}
\def\({\left(}
\def\){\right)}
\def\[{\left[}
\def\]{\right]}
\def\<{\left<}
\def\>{\right>}

\def\less{\lesssim}

\def\ER{\color{black}}

\begin{document}

\title{Holomorphic Extension on Product Lipschitz Surfaces in Two Complex Variables}

\author{Jarod Hart}
\address{Department of Mathematics\\ Wayne State University\\ Detroit, MI 48202}
\email{ jarod.hart@wayne.edu}


\author{Alessandro Monguzzi}
\address{Dipartimento di Matematica\\ 
Universit\`{a} Statale di Milano\\ Via C. Saldini 
50, 20133 Milan\\
Italy}
\email{alessandro.monguzzi@unimi.it}

\thanks{The second author partially supported by the grant PRIN 2010-11 {\em Real and Complex Manifolds: Geometry, Topology and Harmonic Analysis} }


\date{\today}

\dedicatory{ }


\begin{abstract}
In this work we prove a new $L^p$ holomorphic extension result for functions defined on product Lipschitz surfaces with small Lipschitz constants in two complex variables.  We define biparameter and partial Cauchy integral operators that play the role of boundary values for holomorphic functions on product Lipschitz domain.  In the spirit of the application of David-Journ\'e-Semmes and Christ's $Tb$ theorem to the Cauchy integral operator, we prove a biparameter $Tb$ theorem and apply it to prove $L^p$ space bounds for the biparameter Cauchy integral operator.  We also prove some new biparameter Littlewood-Paley-Stein estimates and use them to prove the biparameter $Tb$ theorem.
\end{abstract}

\maketitle

\begin{spacing}{1}

\section{Introduction}

In this work, we solve a holomorphic extension problem for certain product surfaces in $\C^2$ and prove some results in harmonic analysis pertaining to biparameter singular integral operators and Littlewood-Paley-Stein theory.  To motivate our results, we start with a brief history of holomorphic extension and boundary values of holomorphic functions results related to our problem.

The first situation we describe is one on the upper half plane $\mathbb H=\{x+it:x\in\R,\;t>0\}$ in $\C$.  Given a function 
$f\in L^p(\R)$ for $1<p<\infty$, one can extend $f$ to a holomorphic function 
\begin{align*}
F(x+it)=\frac{1}{2\pi i}\int_\R\frac{f(y)}{y-(x+it)}dy;\text{ for }x\in\R,\;t\neq0.
\end{align*}
This function $F$ is a holomorphic extension of $f$ in the the sense that $F$ is holomorphic on $\C\backslash\R$ and $f(x)=f_+(x)-f_-(x)$ for $x\in\R$, where
\begin{align*}
&f_+(x)=\lim_{t\rightarrow0^+}F(x+it)\hspace{.25cm}\text{ and }\hspace{.25cm}f_-(x)=\lim_{t\rightarrow0^+}F(x-it).
\end{align*}
These limits hold almost everywhere in $\R$ and in $L^p(\R)$.  Sometimes this sort of holomorphic extension result is known as a Hilbert-Riemann type problem.  It also follows that $f_{\pm}=\frac{1}{2}(\pm\,I+iH)f$ where $I$ is the identity operator and $H$ is the Hilbert transform
\begin{align*}
Hf(x)=\lim_{t\rightarrow0^+}\frac{1}{\pi}\int_\R\frac{x-y}{(x-y)^2+t^2}f(y)dy.
\end{align*}
There is a rich history involving the Hilbert transform and boundary behavior of holomorphic functions, which is intrinsically related to the study of Hardy spaces.  The $L^p(\R)$ extension results mentioned here were solved by the combined work of many people in the early 1900's, including classical works of Hilbert and Riesz, among others.

The next situation we discuss is a Lipschitz perturbed upper half space of the form $\mathbb H_\Gamma=\{\gamma(x)+it:x\in\R,\;t>0\}$ where $\gamma:\R\rightarrow\C$ is a Lipschitz 
graph.  Problems related to holomorphic functions on $\mathbb H_\Gamma$ can often be solved using the corresponding solution on $\mathbb H$ and the Riemann mapping theorem, but that is not the case in general with the $L^p$ boundary behavior of holomorphic functions on $\mathbb H_\Gamma$.  The holomorphic extension result corresponding to the one in the last paragraph is the following:  given a function $g\in L^p(\Gamma)$ for $1<p<\infty$, one can extend $g$ to a holomorphic function 
\begin{align*}
G(z+it)=\frac{1}{2\pi i}\int_\Gamma\frac{g(\xi)}{\xi-(z+it)}d\xi;\text{ for }z\in\Gamma,\;t\neq0,
\end{align*}
which is a holomorphic extension of $g$ in the the sense that $G$ is holomorphic on $\C\backslash\Gamma$ and $g(z)=g_+(z)-g_-(z)$ for $z\in\Gamma$, where
\begin{align*}
&g_+(z)=\lim_{t\rightarrow0^+}G(z+it)\hspace{.25cm}\text{ and }\hspace{.25cm}g_-(x)=\lim_{t\rightarrow0^+}G(z-it)
\end{align*}
and these limits exist in pointwise almost everywhere on $\Gamma$ and in $L^p(\Gamma)$.  The boundary values of $G$ can be realized in this setting as well by $g_{\pm}(z)=\frac{1}{2}(\pm \,I+iC_\Gamma)g(z)$, where $C_\Gamma$ is the Cauchy integral transform
\begin{align*}
C_\Gamma g(z)=\lim_{t\rightarrow0^+}\frac{1}{\pi}\int_\Gamma\frac{z-\xi}{(z-\xi)^2+t^2}g(\xi)d\xi.
\end{align*}
Progressing from the extension problem on $\mathbb H$ to the one on $\mathbb H_\Gamma$ was not an easy feat.  It took more than 40 years from the proof of $L^p$ bounds for the Hilbert transform to prove the $L^p$ bounds for the Cauchy integral transform along Lipschitz curves with small constants, which was due to Calder\'on \cite{Ca}.  The proof for a general Lipschitz constant appeared some years later in works of Coifman, McIntosh and Meyer \cite{MR672839, MR758451}.  Later, new proofs and generalizations appeared in the work of David--Journ\`{e}--Semmes \cite{DJS}, Jones \cite{Jo}, and Chist \cite{Chr}, among others. 
 
These results were extended to upper half spaces of type $\R^{n+1}_+=\R^n\times(0,\infty)$ in place of $\mathbb H$ by Stein in terms of systems of conjugate harmonic functions, see e.g. \cite{St1}.  In this situation, the role of the Hilbert transform is replaced by the Riesz transforms $R_j$ on $\R^n$, and convergence results hold in 
$L^p(\R^n)$ for $1<p<\infty$ and appropriate functions $f:\R^n\rightarrow\C$ with its harmonic conjugates $R_jf(x)$ for $j=1,...,n$.  The $n$-dimensional Lipschitz upper half spaces were also addressed in a series of papers, Fabes-Kenig-Neri \cite{FKN}, Jerison-Kenig \cite{JK}, and Kenig-Pipher \cite{KP}.  They solved problems related to harmonic functions on upper half domains of the form $\R^{n+1}_{L+}=\R^n\times\{L(x)+t:x\in\R^n,\;t>0\}$, among others, where $L:\R^n\rightarrow\R$ is a Lipschitz 
function.  In \cite{FKN,JK,KP}, double layer potentials replace the Riesz transforms in Stein's work, and their associated Hardy spaces are defined.

Another setting where this type of problem has been solved is on the product upper half plane $\mathbb H\times\mathbb H$ in $\C^2$.  The corresponding Hilbert-Riemann property for the product upper half plane is stated as follows:  given a function $f\in L^p(\R^2)$ for $1<p<\infty$, one can extend $f$ to a holomorphic function 
\begin{align*}
F(x+it)=\frac{1}{(2\pi i)^2}\int_{\R^2}\frac{f(y)}{(y_1-(x_1+it_1))(y_2-(x_2+it_2))}dy;\text{ for }x=(x_1,x_2)\in\R^2,\;t=(t_1,t_2)
\end{align*}
with $t_1,t_2\neq0$.  This function $F$ is a holomorphic extension of $f$ in the the sense that $F$ is holomorphic on $(\C\backslash\R)\times(\C\backslash\R)$ and $f(x)=f_{++}(x)-f_{+-}(x)-f_{-+}(x)+f_{--}(x)$ for $x\in\R^2$, where
\begin{align*}
&f_{++}(x)=\lim_{t_1,t_2\rightarrow0^+}F(x_1+it_1,x_2+it_2),& && &f_{+-}(x)=\lim_{t_1,t_2\rightarrow0^+}F(x_1+it_1,x_2-it_2),&\\
&f_{-+}(x)=\lim_{t_1,t_2\rightarrow0^+}F(x_1-it_1,x_2+it_2),& &\text{ and }& &f_{--}(x)=\lim_{t_1,t_2\rightarrow0^+}F(x_1-it_1,x_2-it_2).
\end{align*}
These limits hold almost everywhere in $\R^2$ and in $L^p(\R^2)$.  In this situation, it follows that $f_{\pm,\pm}=\frac{1}{4}(\pm\, I+iH_1)(\pm \,I+iH_2)f(x)$ where $H_1f$ and $H_2f$ are the Hilbert transforms applied to the first and second variable of $f$ respectively.  These operators $H_1$, $H_2$, and $H_1H_2$ are sometimes called the partial and biparameter Hilbert transforms, which are bounded on $L^p(\R^2)$, see e.g. \cite{rF2,rFS1}.   These boundedness results are related to the biparameter Hardy space theory that is addressed in \cite{MM,GS,Gu,CF,rF2,rFS1,rF3,rF4}, among many others.  Many of these articles work on the polydisk instead of products of upper half planes, but working in these two settings is essentially equivalent; look, for example, in \cite{GS}. 

In this work, we address a holomorphic extension result similar to the ones above for product Lipschitz upper half spaces, which is stated as follows. Given an appropriate Lipschitz boundary surface $\Gamma=\Gamma_1\times\Gamma_2\subset\C^2$ and a function $g:\Gamma\rightarrow\C$, there is a function $G$ that is holomorphic on $(\C\backslash\Gamma_1)\times(\C\backslash\Gamma_2)$ satisfying
\begin{align}
g(z)=g_{++}(z)-g_{+-}(z)-g_{-+}(z)+g_{--}(z),\label{holomorphiclimit}
\end{align}
for $z=(z_1,z_2)\in\Gamma$, where
\begin{align}
&g_{++}(z)=\lim_{t_1,t_2\rightarrow0^+}G(z_1+it_1,z_2+it_2),& && &g_{+-}(z)=\lim_{t_1,t_2\rightarrow0^+}G(z_1+it_1,z_2-it_2),\label{limits}\\
&g_{-+}(z)=\lim_{t_1,t_2\rightarrow0^+}G(z_1-it_1,z_2+it_2),& &\text{and}& &g_{--}(z)=\lim_{t_1,t_2\rightarrow0^+}G(z_1-it_1,z_2-it_2).\notag
\end{align}
For now we leave the sense in which \eqref{holomorphiclimit} holds, the sense that the limits in \eqref{limits} hold, and the conditions on $\Gamma$ unspecified, but these things will be defined later in this section.

Before we state our holomorphic extension result, we will set a few definitions.  We say that $G(\omega_1,\omega_2)$ is holomorphic at $(\omega_1,\omega_2)\in\C^2$ if $G$ has an absolutely convergent power series representation on a neighborhood of $(\omega_1,\omega_2)$.  We will call the Lipschitz surfaces that we work with product Lipschitz surfaces with small Lipschitz constants, and they are defined as follows.  Let $L_1,L_2:\R\rightarrow\R$ be Lipschitz functions.  Define $\gamma_1(x_1)=x_1+iL_1(x_1)$, $\gamma_2(x_2)=x_2+iL_2(x_2)$, and $\gamma(x)=(\gamma_1(x_1),\gamma_2(x_2))\in\C^2$ for $x=(x_1,x_2)\in\R^2$.  Then we call $\Gamma=\Gamma_1\times\Gamma_2=\gamma_1(\R)\times\gamma_2(\R)$ a product Lipschitz surface in $\C^2$.  We say that $\Gamma$ is a product Lipschitz surface with small Lipschitz constants if the Lipschitz constants $\lambda_1$ and $\lambda_2$ of $L_1$ and $L_2$ respectively are both smaller than $1$.  The upper half space associated to 
$\Gamma$ is defined $\mathbb H_{\Gamma_1}\times\mathbb H_{\Gamma_2}$, where $\mathbb H_{\Gamma_j}=\{\gamma_j(x_j)+it_j:x_j\in\R,\;t_j>0\}$.  We also define $L^p(\Gamma)$ for a product Lipschitz surface $\Gamma$ as follows.  Given a product Lipschitz surface $\Gamma=\gamma_1(\R)\times\gamma_2(\R)$, let $L^p(\Gamma)$ be the collection of measurable functions $g:\Gamma\rightarrow\C$ such that
\begin{align*}
||g||_{L^p(\Gamma)}^p=\int_{\R^2}|g(\gamma(x))|^p|\gamma_1'(x_1)\gamma_2'(x_2)|dx_1\,dx_2<\infty.
\end{align*}
Now we state our holomorphic extension result.

\begin{theorem}\label{t:ext}
Let $\Gamma$ be a product Lipschitz surface with small Lipschitz constants in $\C^2$ defined by $\gamma=(\gamma_1,\gamma_2):\R^2\rightarrow\C^2$.  Assume that 
\begin{align*}
&\lim_{|x_1|\rightarrow\infty}\frac{\gamma_1(x_1)}{x_1}=c_1&  &\text{and}& &\lim_{|x_2|\rightarrow\infty}\frac{\gamma_2(x_2)}{x_2}=c_2&
\end{align*}
for some $c_1,c_2\in\C$.  If $g\in L^p(\Gamma)$ for some 
$1<p<\infty$, then there exists a function $G:(\C\backslash\Gamma_1)\times(\C\backslash\Gamma_2)\rightarrow\C$ that is a 
holomorphic extension of $g$, where \eqref{holomorphiclimit} holds almost everywhere on $\Gamma$ and the limits in \eqref{limits} hold in $L^p(\Gamma)$ and pointwise almost everywhere on $\Gamma$.
\end{theorem}


In addition to the problems mentioned above, some other boundary value problems related to Theorem \ref{t:ext} can be found in the work of Bochner \cite{Bo}, Weinstock \cite{W}, Stein 
\cite{St2,St3}, Jacewicz \cite{Ja}, and Krantz \cite{Kr1,Kr2}.  These works prove a number results about the behavior of holomorphic functions on domains with smooth boundaries in $\C^n$, but the point of view taken in \cite{Bo,W,St2,St3,Ja,Kr1,Kr2} is different than the one taken in this work.  They start with a holomorphic function $G$ defined on a domain $D$ and make conclusions about the $G$ near or on the boundary $\partial D$.  Whereas we are given a boundary $\Gamma$ with initial data $g$ and construct a holomorphic function $G$
on the domain $\mathbb H_{\Gamma_1}\times\mathbb H_{\Gamma_2}$ whose behavior at the boundary is determined by $g$.  The meaning of boundary behavior for us is described in 
\eqref{holomorphiclimit} and \eqref{limits}.

We take this ``extension from the boundary'' point of view because we want this work to emphasize the boundedness of boundary value singular integral operators that take the place of the partial and biparameter Hilbert transforms from the extension problems above; we call these operators the biparameter and partial Cauchy integral transforms, and they will be defined later in this section.  

It is natural to eventually define Hardy spaces of holomorphic functions associated to our product upper half space in the same way that Hardy spaces are defined on $\mathbb 
H$, $\mathbb H_\Gamma$, $\R^{n+1}_+$, $\R^{n+1}_{L+}$, and $\mathbb H\times\mathbb H$.  These Hardy spaces are related to the holomorphic extension problems briefly described in the beginning of the Introduction.  It is also natural to expect that every holomorphic function in these new Hardy spaces would be realized as one of our extensions from the boundary $\Gamma$.  However, we do not want to deal with the extra technicalities involved with developing these spaces in this work.  Instead we focus on the holomorphic extension problem for $\Gamma$ as stated in Theorem \ref{t:ext}.

The situation in Theorem \ref{t:ext} is more general than holomorphic extension results from \cite{Bo,W,St2,St3,Ja,Kr1,Kr2} in terms of the regularity required for the boundary.  In all of these works, the domain $D$ is assumed to have smooth boundary, at least $C^2$.  Whereas Theorem \ref{t:ext} can be viewed as a boundary result for holomorphic functions on $\mathbb H_{\Gamma_1}\times\mathbb H_{\Gamma_2}$, which requires only Lipschitz type smoothness for the boundary $\Gamma$.  

To prove Theorem \ref{t:ext}, we take an approach related to the ones in \cite{MM,Chan,rF1,GS,St4,CF}, which are more geometric in nature and uses the boundedness of biparameter and partial Hilbert transforms.  In place of the Hilbert transforms, we define biparameter and partial Cauchy integral transforms for $z=(z_1,z_2)\in\Gamma$ and appropriate $g:\Gamma\rightarrow\C$,
\begin{align*}
&\mathcal C_\Gamma g(z)=\lim_{t_1,t_2\rightarrow0^+}\mathcal C_tg(z);& &\mathcal C_tg(z)=\frac{1}{(2\pi i)^2}\int_{\Gamma}\frac{z_1-\xi_1}{(z_1-\xi_1)^2+t_1^2}\frac{z_2-\xi_2}{(z_2-\xi_2)^2+t_2^2}g(\xi)d\xi,\\
&\mathcal C_\Gamma^{p1} g(z)=\lim_{t_1,t_2\rightarrow0^+}\mathcal C_t^{p1}g(z);& &\mathcal C_t^{p1}g(z)=\frac{1}{(2\pi i)^2}\int_{\Gamma}\frac{z_1-\xi_1}{(z_1-\xi_1)^2+t_1^2}\frac{t_2}{(z_2-\xi_2)^2+t_2^2}g(\xi)d\xi,\\
&\mathcal C_\Gamma^{p2} g(z)=\lim_{t_1,t_2\rightarrow0^+}\mathcal C_t^{p2}g(z);& &\mathcal C_t^{p2}g(z)=\frac{1}{(2\pi i)^2}\int_{\Gamma}\frac{t_1}{(z_1-\xi_1)^2+t_1^2}\frac{z_2-\xi_2}{(z_2-\xi_2)^2+t_2^2}g(\xi)d\xi.
\end{align*}
The limits defining $\mathcal C_\Gamma$, $\mathcal C_\Gamma^{p1}$, and $\mathcal C_\Gamma^{p2}$ are taken in the following pointwise sense:  given $c\in\C$ and $c_t\in\C$ for $t=(t_1,t_2)\in(0,\infty)^2$, we say $c_t\rightarrow c$ as $t_1,t_2\rightarrow0^+$ if for all $\epsilon>0$, there exists $\delta>0$ such that $0<t_1,t_2<\delta$ implies $|c_t-c|<\epsilon$.  We also define convergence in normed spaces as $t_1,t_2\rightarrow0^+$:  given a normed function space $X$, $F\in X$, and $F_t\in X$ for $t=(t_1,t_2)\in(0,\infty)^2$, we say $F_t\rightarrow F$ as $t_1,t_2\rightarrow0^+$ if $||F_t-F||_X\rightarrow0$ as $t_1,t_2\rightarrow0^+$.  The operators $\mathcal C_\Gamma g$, $\mathcal C_\Gamma^{p1} g$, and $\mathcal C_\Gamma^{p2} g$ are defined initially as pointwise limits on test functions, and we will prove later that these limits hold in $L^p(\Gamma)$ as well for $1<p<\infty$ and appropriate $g$.  These convergence results will be proved in Sections 5 and 6.  A crucial part of the proof of these convergence results is the $L^p(\Gamma)$ boundedness of $\mathcal C_\Gamma$, $\mathcal C_\Gamma^{p1}$, and $\mathcal C_\Gamma^{p2}$, which we state now in Theorem \ref{t:Cgammabounds}.  

\begin{theorem}\label{t:Cgammabounds}
Let $\Gamma$ be a product Lipschitz surface with small Lipschitz constant in $\C^2$ defined by $\gamma=(\gamma_1,\gamma_2):\R^2\rightarrow\C^2$.  Assume that 
\begin{align*}
&\lim_{|x_1|\rightarrow\infty}\frac{\gamma_1(x_1)}{x_1}=c_1&  &\text{and}& &\lim_{|x_2|\rightarrow\infty}\frac{\gamma_2(x_2)}{x_2}=c_2&
\end{align*}
for some $c_1,c_2\in\C$.  Then the operators $\mathcal C_\Gamma$, $\mathcal C_\Gamma^{p1}$, and $\mathcal C_\Gamma^{p2}$ can be continuously extended to bounded operators on $L^p(\Gamma)$ and for $g\in L^p(\Gamma)$
\begin{align*}
&\lim_{t_1,t_2\rightarrow0^+}\mathcal C_tg=\mathcal C_\Gamma g,& &\lim_{t_1,t_2\rightarrow0^+}\mathcal C_t^{p1}g=\mathcal C_\Gamma^{p1} g,& &\text{and}& &\lim_{t_1,t_2\rightarrow0^+}\mathcal C_t^{p2}g=\mathcal C_\Gamma^{p2} g
\end{align*}
in $L^p(\Gamma)$ when $1<p<\infty$ and pointwise almost everywhere on $\Gamma$.
\end{theorem}


  We take a moment now to discuss why Theorem \ref{t:Cgammabounds} cannot be proved with techniques currently available in the literature; in particular why the analysis of holomorphic functions related to $\mathbb H\times\mathbb H$ and the Hilbert transforms is not applicable to our problem.  Much of the machinery used in the analysis of holomorphic functions related to $\mathbb H\times\mathbb H$ and the Hilbert transforms is not available when we move to the setting of $\Gamma=\Gamma_1\times\Gamma_2$ and the Cauchy integral transforms.  If one defines the biparameter Hilbert transform $H^{bp}$ as $C_\Gamma$ is defined above (but with $\gamma_j(x_j)=x_j$ for $j=1,2$), then with the aid of the Fourier transform it is easy to show that $H^{bp}=H_1H_2$.  Since $H^{bp}$ can be realized as this composition of $H_1$ and $H_2$ in this way, the $L^p(\R^2)$ bounds for $H^{bp}$ trivially follow from those of $H_1$ and $H_2$.  Furthermore, the fact that $H^{bp}=H_1H_2$ says that the two dimensional limit $t_1,t_2\rightarrow0^+$ defining $H^{bp}$ can actually be realized as iterated one dimensional limits $t_1\rightarrow0^+$ and $t_2\rightarrow0^+$.  There is no such formula to write $C_\Gamma$ as a composition $C_\Gamma^{p1}$ and $C_\Gamma^{p2}$ that we know of since the Fourier transform is not a viable tool in this setting; that is, it is not known in general if the two dimensional limit defining $C_\Gamma$ can be realized as an iterated one dimensional limit.  This precludes, at least with the tools currently available, any relatively simple proof of $L^p(\Gamma)$ bounds for $C_\Gamma$, and hence motivates the development of the harmonic analysis theory in this article.

To obtain the pointwise convergence result stated in Theorem \ref{t:Cgammabounds}, we need a little more that then boundedness of the partial and biparameter Cauchy integral transforms. For $z=(z_1,z_2)\in\Gamma$ and appropriate functions $g:\Gamma\to\C$, we define the maximal biparameter Cauchy integral transform
\begin{equation}\label{max : cauchy}
  \mathcal{C}_\Gamma^* g(z)=\sup_{t_1,t_2>0}|\mathcal{C}_tg(z)|,
\end{equation}
where $t=(t_1,t_2)$.  Then, we prove the following boundedness result. 

\begin{theorem}\label{t:MaxCgammabounds}
Let $\Gamma$ be as in Theorem \ref{t:Cgammabounds}. The maximal operator $\mathcal{C}_\Gamma^*$ extends to a bounded operator $\mathcal{C}_\Gamma^*:L^p(\Gamma)\to L^p(\Gamma)$ for $1<p<\infty$. Moreover,  for all $g$ in $L^p(\Gamma)$, $\mathcal{C}_{t}g$ converges to $\mathcal{C}_\Gamma g$ almost everywhere on $\Gamma$.
\end{theorem}

We prove Theorem \ref{t:Cgammabounds} using the approach that David-Journ\'e-Semmes used to apply their $Tb$ theorem to prove $L^p$ bounds for Cauchy integral transform in \cite{DJS}.  For this, we prove the following reduced biparameter $Tb$ theorem.

\begin{theorem}\label{t:Tbtheorem}
Let $b_1,\tilde b_1\in L^\infty(\R^{n_1})$ and $b_2,\tilde b_2\in L^\infty(\R^{n_2})$ be para-accretive functions, and define $b(x)=b_1(x_1)b_2(x_2)$ and $\tilde b(x)=\tilde b_1(x_1)\tilde b_2(x_2)$ for $x=(x_1,x_2)\in\R^{n_1+n_2}$.  Also let $T$ be a biparameter operator of Calder\'on-Zygmund type associated to $b$ and $\tilde b$.  If $T$ satisfies the weak boundedness property, mixed weak boundedness properties, and the $Tb=T^*\tilde b=0$ conditions, then $T$ can be continuously extended to a bounded linear operator on $L^p(\R^n)$ for $1<p<\infty$.
\end{theorem}


There have been a number of results for biparameter singular integral operators of Calder\'on-Zygmund type, going back to R. Fefferman, Stein, and Journ\'e, among others.  There were different versions of $T1$ theorems proved in R. Fefferman-Stein \cite{rFS1}, Journ\'e \cite{J}, Pott-Villaroya \cite{PV}, Han-Lin-Lee \cite{HLL}, Ou \cite{O}, and Hart-Lu-Torres \cite{HLT}.  In fact, the recent articles \cite{HLL,O} include biparameter $Tb$ theorems as well.  The formulation of Theorem \ref{t:Tbtheorem} is different than the ones in \cite{HLL,O}, and even the definitions of biparameter Calder\'on-Zygmund operators are different.  In Section 4, we define biparameter singular integral operators relying only on continuity in test function spaces, a full kernel representation, and testing conditions on normalized bumps, whereas in \cite{HLL,O} the singular integral operators addressed are required to have full and partial kernel representations as well as some a priori partial $L^2$ bounds.  In our formulation, we do not use partial kernel representations or partial $L^2$ boundedness hypotheses.  Instead, we introduce a mixed weak boundedness, which is a testing condition similar to the full weak boundedness property used in many of the aforementioned works.    The formulation of Theorem \ref{t:Tbtheorem} in this work is a natural extension of the single parameter theory, and the sufficient conditions seem to be easy to verify, as will be demonstrated in Section 5.
Unfortunately, Theorem \ref{t:Tbtheorem} is still not a full characterization of $L^p$ bounds for biparameter Calder\'on-Zygmund operators since difficulties of working with product $BMO$ persist, but this reduced $Tb=T^*b=0$ Theorem \ref{t:Tbtheorem} is sufficient to prove the the boundedness results in Theorem \ref{t:Cgammabounds} and hence the holomorphic extension result of in Theorem \ref{t:ext}.  The formulation of the biparameter singular integral operators in this work is essentially the same as the one by Hart-Lu-Torres in \cite{HLT}, but we repeat the constructions to fit the accretive function setting in Theorem \ref{t:Tbtheorem}.

Even though we will only apply Theorem \ref{t:Tbtheorem} when $n_1=n_2=1$, we prove it for general dimensions $n_1,n_2\in\N$.  Our strategy to prove Theorem \ref{t:Tbtheorem} is to decompose the operator $T$,
\begin{align*}
\<Tf,g\>=\sum_{\vec k\in\Z^2}\<\Theta_{\vec k}f,g\>,
\end{align*}
where $\Theta_{\vec k}$ are smooth truncations of $T$.  These 
truncations $\Theta_{\vec k}$ are biparameter Littlewood-Paley-Stein operators, which have been studied extensively in the single parameter setting, see e.g. \cite{DJ,DJS,Se,Ha}.  There are a few results for biparameter Littlewood-Paley-Stein operators due to R. Fefferman, Stein, and Journ\'e \cite{rF2,rFS1,rF3,J}, among others.  All of these results are for operators of convolution type.  We prove estimates for the square function associated to a larger class of operators including non-convolution operators, which we call biparameter Littlewood-Paley-Stein operators.  In particular, we prove bounds for square function operators associated to biparameter Littlewood-Paley-Stein operators, defined by
\begin{align}
Sf(x)^2=\sum_{\vec k\in\Z}|\Theta_{\vec k}f(x)|^2\label{sqfunctiondefn}
\end{align}
for $x\in\R^n$ and appropriate $f:\R^n\rightarrow\C$.

\begin{theorem}\label{t:squarefunction}
Let $b_1\in L^\infty(\R^{n_1})$ and $b_2\in L^\infty(\R^{n_2})$ be para-accretive functions, and define $b(x)=b_1(x_1)b_2(x_2)$ for $x=(x_1,x_2)\in\R^{n_1+n_2}$.  Also let $\Theta_{\vec k}$ for $\vec k\in\Z^2$ be a collection of biparameter Littlewood-Paley-Stein operators with kernels $\theta_{\vec k}$.  If 
\begin{align*}
\int_{\R^{n_1}}\theta_{\vec k}(x,y)b_1(y_1)dy_1=\int_{\R^{n_2}}\theta_{\vec k}(x,y)b_2(y_2)dy_2=0
\end{align*}
for all $\vec k\in\Z^2$ and $x,y\in\R^n$, then $||Sf||_{L^p(\R^n)}\less||f||_{L^p(\R^n)}$ for all $f\in L^p(\R^n)$ when $1<p<\infty$.  Note that $S$ is the square function operator defined in \eqref{sqfunctiondefn}
\end{theorem}

In fact, we will prove Theorem \ref{t:squarefunction} for a slightly larger class of operators than the biparameter Littlewood-Paley-Stein operators.  These classes of operators will be defined in the coming sections, and it will be specified how they can be generalized to a slightly larger class by weakening the regularity properties of $\theta_{\vec k}$.

The formulations and proofs of Theorems \ref{t:Tbtheorem} and \ref{t:squarefunction} were introduced by Hart-Lu-Torres \cite{HLT} in a slightly different setting, where $b=\tilde b=1$.  In Sections 3 and 4, we reproduce the proofs from \cite{HLT}, and address the additional technical difficulties that arise when accretive functions $b$ and $\tilde b$ are used in place of $1$.

This article is organized in the following way.  In Section 2, we prove the holomorphic extension result in Theorem \ref{t:ext} assuming that Theorem \ref{t:Cgammabounds} holds.  In Section 3, we develop some biparameter Littlewood-Paley-Stein theory and prove Theorem \ref{t:squarefunction}.  In Section 4, we prove the biparameter $Tb$ Theorem \ref{t:Tbtheorem} using results from Section 3.  Finally in Section 5, we prove part of Theorem \ref{t:Cgammabounds} by applying Theorem \ref{t:Tbtheorem} to a parameterized version of $\mathcal C_\Gamma$, \ER and we prove Theorem \ref{t:MaxCgammabounds}.  In Section 6 we prove the rest of Theorem \ref{t:Tbtheorem} by applying the one parameter $Tb$ theorem from \cite{DJS} to parameterized versions of $\mathcal C_\Gamma^{p1}$ and $\mathcal C_\Gamma^{p2}$.

\section{Holomorphic Extension from Product Lipschitz Domains}

Fix Lipschitz functions $L_1,L_2:\R\rightarrow\R$ with Lipschitz constants $\lambda_1<1$ and $\lambda_2<1$.  Define $\gamma_1(x_1)=x_1+iL_1(x_1)$, $\gamma_2(x_2)=x_2+iL_2(x_2)$, and $\gamma(x)=(\gamma_1(x_1),\gamma_2(x_2))$ for $x=(x_1,x_2)\in\R^2$.  Then $\Gamma=\Gamma_1\times\Gamma_2$ is a product Lipschitz surface with small Lipschitz constants in $\C^2$, where $\Gamma_1=\gamma_1(\R)$ and $\Gamma_2=\gamma_2(\R)$.  It follows that
\begin{align*}
0<1-\lambda_j^2&\leq \frac{(x_j-y_j)^2-(L_j(x_j)-L_j(y_j))^2}{(x_j-y_j)^2}=\frac{|Re\[(\gamma_j(x_j)-\gamma_j(y_j))^2\]|}{(x_j-y_j)^2}\leq2.
\end{align*}
Throughout this work, we will use the fact that $Re\[(\gamma_j(x_j)-\gamma_j(y_j))^2\]$ and $(x_j-y_j)^2$ are comparable with constants only depending on the Lipschitz constants of $\gamma$, not on $x_j$ and $y_j$.  We also remark that the norms of $g$ and $g\circ\gamma$ are comparable in the following sense:  for any $g\in L^p(\Gamma)$,
\begin{align}
||g\circ\gamma||_{L^p(\R^2)}^p&\leq ||(\gamma_1')^{-1}||_{L^\infty(\R)}||(\gamma_2')^{-1}||_{L^\infty(\R)}||g||_{L^p(\Gamma)}^p\leq||g||_{L^p(\Gamma)}^p\notag\\
&\leq||\gamma_1'||_{L^\infty(\R)}||\gamma_2'||_{L^\infty(\R)}||g\circ\gamma||_{L^p(\R^2)}^p\leq2||g\circ\gamma||_{L^p(\R^2)}^p.\label{equivnorm}
\end{align}
Note that since $Re[\gamma_j'(x_j)]=1$ for all $x_j\in\R$, we have $|\gamma_j'(x_j)|\geq Re[\gamma_j'(x_j)]=1$ for all $x_j\in\R$.  Now given a function $g:\Gamma\rightarrow\C$, we define for $\omega=(\omega_{t_1},\omega_{t_2})=(z_1+it_1,z_2+it_2)\in(\C\backslash\Gamma_1)\times(\C\backslash\Gamma_2)$ where $(z_1,z_2)\in\Gamma$ and $t_1,t_2\neq0$,
\begin{align}
G(\omega_{t_1},\omega_{t_2})=\frac{1}{(2\pi i)^2}\int_{\Gamma}\frac{g(\xi)d\xi}{(\xi_1-\omega_{t_1})(\xi_2-\omega_{t_2})}.\label{Gdefn}
\end{align}
It follows that 
\begin{align*}
&G(\omega_{t_1},\omega_{t_2})=\frac{1}{4}\int_{\Gamma}\(p_{t_1}(z_1-\xi_1)p_{t_2}(z_2-\xi_2)-q_{t_1}(z_1-\xi_1)q_{t_2}(z_2-\xi_2)\phantom{\int}\right.\\
&\left.\phantom{\int}\hspace{2.9cm}+iq_{t_1}(z_1-\xi_1)p_
{t_2}(z_2-\xi_2)+ip_{t_1}
(z_1-\xi_1)q_{t_2}(z_2-\xi_2)\)g(\xi)d\xi,
\end{align*}
where
\begin{align*}
&p_{t_j}(\omega_j)=\frac{1}{\pi}\frac{t_j}{\omega_j^2+t_j^2}& &\text{and}& &q_{t_j}(\omega_j)=\frac{1}{\pi}\frac{\omega_j}{\omega_j^2+t_j^2}& &\text{for }\omega_j\in\C.&
\end{align*}
Also define for $t=(t_1,t_2)\in(0,\infty)^2$, $g_1:\Gamma_1\rightarrow\C$, $g_2:\Gamma_2\rightarrow\C$, $g:\Gamma\rightarrow\C$, and $z=(z_1,z_2)\in\Gamma$, the operators
\begin{align*}
&P_{t_1}g_1(z_1)=\int_{\Gamma_1}p_{t_1}(z_1-\xi_1)g_1(\xi_1)d\xi_1,& &\hspace{-1.2cm}P_{t_2}g_2(z_2)=\int_{\Gamma_2}p_{t_2}(z_2-\xi_2)g_2(\xi_2)d\xi_2,&\\
&\text{and }\hspace{.25cm}P_tg(z)=\int_{\Gamma}p_{t_1}(z_1-\xi_1)p_{t_2}(z_2-\xi_2)g(\xi)d\xi.
\end{align*}
We use the indices of $P_{t_1}$, $P_{t_2}$, and $P_t$ to identify the operators.  Note that $P_tg=P_{t_1}P_{t_2}g$ for $g:\Gamma\rightarrow\C$, where we use the notation
\begin{align*}
&P_{t_1}g(z)=\int_{\Gamma_1}p_{t_1}
(z_1-\xi_1)g(\xi_1 , z_2)d\xi_1& &\text{and}& 
&P_{t_2}g(z)=\int_{\Gamma_2}p_{t_2}(z_2-\xi_2)g(z_1
, \xi_2)d\xi_2&
\end{align*}
This is an abuse of notation, but it  is clear in context which operator is being used.  We start with a lemma about the convergence of the operators $P_{t_1}g$, $P_{t_2}g$, and $P_tg$ for $g\in L^p(\Gamma)$.

\begin{lemma}\label{l:Poissonconverge}
Let $\Gamma$ be a product Lipschitz surface with small Lipschitz constants in $\C^2$ and $g\in L^p(\Gamma)$ for some $1<p<\infty$.  Then
\begin{align*}
&\lim_{t_1\rightarrow0^+}P_{t_1}g=g,& &\lim_{t_2\rightarrow0^+}P_{t_2}g=g,& &\text{and}& &\lim_{t_1,t_2\rightarrow0^+}P_tg=g,&
\end{align*}
where each limit holds in the topology of $L^p(\Gamma)$ and pointwise almost everywhere on $\Gamma$.
\end{lemma}

\begin{proof}
We first verify that $P_{t_j}1=1$ for each $j=1,2$.  Let $R>0$ and 
\begin{align*}
E_R=\{z_j\in\Gamma_j:|z_j|\leq R\}\cup \{z_j\in\C:|z_j|=R,\;Im(z_j)>L_j(Re(z_j))\}.
\end{align*}
$E_R$ is a closed, and for $R$ sufficiently large, it defines the boundary of an open, simply connected region $U_R=\{z_j\in\C:|z_j|<R,\;Im(z_j)>L_j(Re(z_j))\}$.  For $z_j\in\Gamma_j$,$t_j>0$, and $R$ sufficiently large, it follows that $z_j+it_j\in U_R$ and $z_j-it_j\notin U_R$.  Then
\begin{align*}
\frac{t_j}{\xi_j-(z_j-it_j)}
\end{align*}
is holomorphic in $\xi_j$ on $U_R$ for such $z_j$, $t_j$, and $R$.  Using the decay of $p_{t_j}$ and a residue theorem, it follows that
\begin{align*}
\int_{\Gamma_j}p_{t_j}(z_j-\xi_j)d\xi_j&=\lim_{R\rightarrow\infty}\frac{1}{\pi}\int_{E_R}\frac{t_j}{(\xi_j-(z_j+it_j))(\xi_j-(z_j-it_j))}d\xi_j\\
&=\lim_{R\rightarrow\infty}\frac{1}{\pi}\frac{2\pi i\,t_j}{(z_j+it_j)-(z_j-it_j)}=1.
\end{align*}
Consider the following parameterized versions of $P_t$, $P_{t_1}$, and $P_{t_2}$:  for $f:\R^2\rightarrow\C$ and $x\in\R^2$
\begin{align*}
&\widetilde P_{t_1}f(x)=\int_{\R}p_{t_1}(\gamma_1(x_1)-\gamma_1(y_1))\gamma_1'(y_1)f(y_1,x_2)dy_1,\\
&\widetilde P_{t_2}f(x)=\int_{\R}p_{t_2}(\gamma_2(x_2)-\gamma_2(y_2))\gamma_2'(y_2)f(x_1,y_2)dy_2,\text{ and }\\
&\widetilde P_tf(x)=\widetilde P_{t_1}\widetilde P_{t_2}f(x)=\int_{\R^2}p_{t_1}(\gamma_1(x_1)-\gamma_1(y_1))p_{t_2}(\gamma_2(x_2)-\gamma_1(y_2))\gamma_1'(y_1)\gamma_2'(y_2)f(y)dy.
\end{align*}
The kernels of $\widetilde P_{t_1}$, $\widetilde P_{t_2}$, and $\widetilde P_t$ are 
\begin{align*}
&\widetilde p_{t_1}(x_1,y_1)=p_{t_1}(\gamma_1(x_1)-\gamma_1(y_1))\gamma_1'(y_1),& &\widetilde p_{t_2}(x_2,y_2)=p_{t_2}(\gamma_2(x_2)-\gamma_2(y_2))\gamma_2'(y_2),\\
&\text{and }\;\;\;\widetilde p_t(x,y)=\widetilde p_{t_1}(x_1,y_1)\widetilde p_{t_2}(x_2,y_2),& &\hspace{-.75cm}\text{ respectively.}
\end{align*}
Note that $\widetilde P_{t_j}1(x_j)=P_{t_j}1(\gamma_j(x_j))=1$ for all $x_j\in\R$.  Also since the Lipschitz constant of $L_1$ and $L_2$ are small, it follows that
\begin{align*}
|\widetilde 
p_{t_j}(x_j,y_j)|&=\frac{1}{\pi}\left|\frac{t_j|\gamma'_j(y_j)|}{t_j^2+(\gamma_j(x_j)-\gamma_j(y_j))^2}\right| \leq\frac { t_j }{t_j^2+(1-\lambda_j^2)(x_j-y_j)^2}\less\frac{t_j^{ -1}}{(1+t_j^{-1}|x_j-y_j|)^2}.
\end{align*}
Then $\{\widetilde p_{t_j}:t_j>0\}$ forms an approximation to identity on $\R$ for each $j=1,2$.  Fix $g\in L^p(\Gamma)$ for some $1<p<\infty$.  It follows that $g\circ\gamma\in L^p(\R^2)$, and hence that $g\circ\gamma(\cdot,x_2)\in L^p(\R)$ for almost every $x_2\in\R$.  Now fix $x_2\in\R$ outside of an appropriate exceptional set, so that $||g\circ\gamma(\cdot,x_2)||_{L^p(\R)}<\infty$.  It follows that $g\circ\gamma(\cdot,x_2)\in L^p(\R)$ and hence that
\begin{align*}
&\lim_{t_1\rightarrow0^+}||\widetilde P_{t_1}(g\circ\gamma)(\cdot,x_2)-g\circ\gamma(\cdot,x_2)||_{L^p(\R)}=0.
\end{align*}
By dominated convergence, it also follows that
\begin{align*}
&\lim_{t_1\rightarrow0^+}||\widetilde P_{t_1}(g\circ\gamma)-g\circ\gamma||_{L^p(\R^2)}^p=\int_\R\lim_{t_1\rightarrow0^+}||\widetilde P_{t_1}(g\circ\gamma)(\cdot,x_2)-g\circ\gamma(\cdot,x_2)||_{L^p(\R)}^pdx_2=0.
\end{align*}
Therefore $\widetilde P_{t_1}(g\circ\gamma)\rightarrow g\circ\gamma$ in $L^p(\R^2)$, and in light of \eqref{equivnorm} it easily follows that $P_{t_1}g\rightarrow g$ in $L^p(\Gamma)$.  By symmetry, it follows that $P_{t_2}g\rightarrow g$ in $L^p(\Gamma)$ as well.  Now for $g\in L^p(\Gamma)$, we verify that $P_tg\rightarrow g$ in $L^p(\Gamma)$ as $t_1,t_2\rightarrow0^+$ for $1<p<\infty$, as defined in the introduction.  First, define $\mathcal M_1$ to be the Hardy-Littlewood maximal function acting on the first variable of a function $f:\R^2\rightarrow\C$, i.e.
\begin{align*}
\mathcal M_1f(x)=\sup_{I\ni x_1}\frac{1}{|I|}\int_I|f(y_1,x_2)|dy_1,
\end{align*}
where the supremum is taken over all intervals $I\subset\R$ that contain $x_1$.  It is not hard to verify that $\mathcal M_1$ is bounded on $L^p(\R^2)$ for $1<p\leq\infty$ and that $|P_{t_1}h(\gamma(x))|\less\mathcal M_1(h\circ\gamma)(x)$ uniformly in $t_1>0$ for any $h\in L^p(\Gamma)$.  The $L^p(\Gamma)$ convergence of $P_tg$ follows:
\begin{align*}
\lim_{t_1,t_2\rightarrow0}||P_tg-g||_{L^p(\Gamma)}&\leq\lim_{t_1,t_2\rightarrow0}||P_{t_1}(P_{t_2}g-g)||_{L^p(\Gamma)}+||P_{t_1}g-g||_{L^p(\Gamma)}\\
&\less\lim_{t_1,t_2\rightarrow0}||\mathcal M_1(\widetilde P_{t_2}(g\circ\gamma)-g\circ\gamma)||_{L^p(\R^2)}+||P_{t_1}g-g||_{L^p(\Gamma)}\\
&\less\lim_{t_2\rightarrow0}||\widetilde P_{t_2}(g\circ\gamma)-g\circ\gamma||_{L^p(\R^2)}+\lim_{t_1\rightarrow0}||P_{t_1}g-g||_{L^p(\Gamma)}=0.
\end{align*}
In the last line, we use that $\widetilde P_{t_2}(g\circ\gamma)\rightarrow g\circ\gamma$ in $L^p(\R^2)$ and that $P_{t_1}(g\circ\gamma)\rightarrow g\circ\gamma$ in $L^p(\R^2)$.  This completes the proof of the $L^p(\Gamma)$ convergence properties in Lemma \ref{l:Poissonconverge}.  Now we prove the pointwise convergence results.  For $g\in L^p(\Gamma)$, it follows that $g\circ\gamma(\cdot,x_2)\in L^p(\R)$ for almost every $x_2\in\R$.  For a fixed $x_2\in\R$ outside of an appropriate measure zero set, by the Lebesgue differentiation theorem it follows that
\begin{align*}
\lim_{t_1\rightarrow0^+}\widetilde P_{t_1}(g\circ\gamma)(x_1,x_2)=g(\gamma(x_1,x_2))
\end{align*}
for almost every $x_1\in\R$.  Hence $\widetilde P_{t_1}(g\circ\gamma)\rightarrow g\circ\gamma$ as $t_1\rightarrow0^+$ pointwise almost everywhere in $\R^2$ and hence that $P_{t_1}g\rightarrow g$ as $t_1\rightarrow0^+$ pointwise almost everywhere in $\Gamma$.  By symmetry, $\widetilde P_{t_2}(g\circ\gamma)\rightarrow g\circ\gamma$ as $t_2\rightarrow0^+$ pointwise almost everywhere in $\R^2$ and hence that $P_{t_2}g\rightarrow g$ as $t_2\rightarrow0^+$ pointwise almost everywhere in $\Gamma$.

Now we verify the pointwise convergence for $P_tg$ on $\Gamma$.  Fix $x\in\R^2$ such that $\widetilde P_{t_1}(g\circ\gamma)(x)\rightarrow g\circ\gamma(x)$ as $t_1\rightarrow0^+$ and $||g\circ\gamma(\cdot,x_2)||_{L^p(\R)}<\infty$, which is true for almost every $x\in\R^2$.  Now we bound
\begin{align}
|\widetilde P_t(g\circ\gamma)(x)-g\circ\gamma(x)|&\leq|\widetilde P_{t_1}(\widetilde P_{t_2}(g\circ\gamma)-(g\circ\gamma))(x)|+|\widetilde P_{t_1}(g\circ\gamma)(x)-(g\circ\gamma)(x)|\notag\\
&\less\int_{\R}p_{t_1}
(\gamma_1(x_1)-\gamma_1(y_1))|\widetilde 
P_{t_2}(g\circ\gamma)(y_1,x_2)-(g\circ\gamma)(y_1,
x_2)|dy_1\label{twoterms}\\
&\hspace{5cm}+|\widetilde P_{t_1}(g\circ\gamma)(x)-(g\circ\gamma)(x)|.\notag
\end{align}
We verify that the first term of \eqref{twoterms} tends to zero as $t_1,t_2\rightarrow0^+$:  let $\epsilon>0$.  Since $\widetilde P_{t_2}(g\circ\gamma)(y_1,x_2)\rightarrow (g\circ\gamma)(y_1,x_2)$ pointwise as $t_2\rightarrow0^+$ for almost every $y_1\in\R$, there exists $\delta>0$ such that $0<t_2<\delta$ implies $|\widetilde P_{t_2}(g\circ\gamma)(y_1,x_2)-g\circ\gamma(y_1,x_2)|<\epsilon$ for almost every $y_1\in\R$ such that $|x_1-y_1|\leq1$ (recall we have fixed $x_1$ and $x_2$).  The selection of $\delta$ does not depend on $y_1$ as long as it is within the compact set defined by $|x_1-y_1|\leq1$.  Now we take $0<t_1,t_2<\min(\delta,\epsilon)/(1+||g\circ\gamma(\cdot,x_2)||_{L^p(\R)})$, which is possible since $x\in\R^2$ was selected so that $||g\circ\gamma(\cdot,x_2)||_{L^p(\R)}$ is finite.  Then
\begin{align*}
&\int_{\R}p_{t_1}(\gamma_1(x_1)-\gamma_1(y_1))|\widetilde P_{t_2}(g\circ\gamma)(y_1,x_2)-g\circ\gamma(y_1,x_2)|dy_1\\
&\hspace{.25cm}\less\epsilon \int_{|x_1-y_1|\leq 
1}p_{t_1}(\gamma_1(x_1)-\gamma_1(y_1))dy_1\\
&\hspace{4cm}+\int_{|x_1-y_1|>1} \frac{t_1(|\widetilde P_{t_2}(g\circ\gamma)(y_1,x_2)|+|g\circ\gamma(y_1,x_2)|)}{(\gamma_1(x_1)-\gamma_1(y_1))^2+t_1^2}dy_1
\end{align*}
\begin{align*}
&\hspace{.25cm}\less\epsilon +t_1\int_{|x_1-y_1|>1} \frac{(|\widetilde P_{t_2}(g\circ\gamma)(y_1,x_2)|+|g\circ\gamma(y_1,x_2)|)}{(x_1-y_1)^2}dy_1\\
&\hspace{.25cm}\less\epsilon +t_1\(||\widetilde P_{t_2}(g\circ\gamma)(\cdot,x_2)||_{L^p(\R)}+||g\circ\gamma(\cdot,x_2)||_{L^p(\R)}\)\(\int_{|x_1-y_1|>1} \frac{dy_1}{(x_1-y_1)^{2p'}}\)^\frac{1}{p'}\\
&\hspace{.25cm}\less\epsilon +t_1||g\circ\gamma(\cdot,x_2)||_{L^p(\R)}\less\epsilon.
\end{align*}
It follows that the first term of \eqref{twoterms} tends to zero as $t_1,t_2\rightarrow0^+$ for almost every $x\in\R^2$.  The second term in 
\eqref{twoterms} also tends to zero as $t_1,t_2\rightarrow0^+$ since $x$ was chosen so that $\widetilde P_{t_1}f(x)\rightarrow f(x)$ as $t_1\rightarrow0^+$.  Again using \eqref{equivnorm}, it easily follow that $P_tg\rightarrow g$ as $t_1,t_2\rightarrow0^+$ pointwise almost everywhere on $\Gamma$.
\end{proof}

Now we prove Theorem \ref{t:ext} assuming Theorem \ref{t:Cgammabounds}; we will prove Theorem \ref{t:Cgammabounds} in Section 5.

\begin{proof}
Let $1<p<\infty$, $g\in L^p(\Gamma)$, and define $G$ as in \eqref{Gdefn}.  Note that $p_{-t_j}(z_j-\xi_j)=-p_{t_j}(z_j-\xi_j)$ and $q_{-t_j}(z_j-\xi_j)=q_{t_j}(z_j-\xi_j)$ for $t_j\neq0$, $z_j\in\Gamma_j$, and $j=1,2$.  Then it follows that for $(z_1,z_2)\in\Gamma$ and $t_1,t_2>0$, we have
\begin{align*}
&G(z_1+it_1,z_2+it_2)=\frac{1}{4}\(P_tg(z)-\mathcal C_tg(z)+i\mathcal C_t^{p1}g(z)+i\mathcal C_t^{p2}g(z)\),\\
&G(z_1+it_1,z_2-it_2)=\frac{1}{4}\(-P_tg(z)-\mathcal C_tg(z)-i\mathcal C_t^{p1}g(z)+i\mathcal C_t^{p2}g(z)\),\\
&G(z_1-it_1,z_2+it_2)=\frac{1}{4}\(-P_tg(z)-\mathcal C_tg(z)+i\mathcal C_t^{p1}g(z)-i\mathcal C_t^{p2}g(z)\),\\
&G(z_1-it_1,z_2-it_2)=\frac{1}{4}\(P_tg(z)-\mathcal C_tg(z)-i\mathcal C_t^{p1}g(z)-i\mathcal C_t^{p2}g(z)\).
\end{align*}
By Theorem \ref{t:Cgammabounds}, it follows that $\mathcal C_\Gamma g,\mathcal C_\Gamma^{p1} g,\mathcal C_\Gamma^{p2} g\in L^p(\Gamma)$ and $\mathcal C_t g\rightarrow\mathcal C_\Gamma g$, $\mathcal C_t^{p1} g\rightarrow\mathcal C_\Gamma^{p1} g$, and $\mathcal C_t^{p2} g\rightarrow\mathcal C_\Gamma^{p2} g$ as $t_1,t_2\rightarrow0^+$ in $L^p(\Gamma)$ and pointwise almost everywhere on $\Gamma$.  Then for $z=(z_1,z_2)\in\Gamma$
\begin{align*}
&g_{++}(z)=\frac{1}{4}\(g(z)-\mathcal C_\Gamma g(z)+i\mathcal C_\Gamma^{p1}g(z)+i\mathcal C_\Gamma^{p2}g(z)\),\\
&g_{+-}(z)=\frac{1}{4}\(-g(z)-\mathcal C_\Gamma g(z)-i\mathcal C_\Gamma ^{p1}g(z)+i\mathcal C_\Gamma ^{p2}g(z)\),\\
&g_{-+}(z)=\frac{1}{4}\(-g(z)-\mathcal C_\Gamma g(z)+i\mathcal C_\Gamma ^{p1}g(z)-i\mathcal C_\Gamma ^{p2}g(z)\),\text{ and}\\
&g_{--}(z)=\frac{1}{4}\(g(z)-\mathcal C_\Gamma g(z)-i\mathcal C_\Gamma ^{p1}g(z)-i\mathcal C_\Gamma ^{p2}g(z)\).
\end{align*}
Then it also follows that \eqref{holomorphiclimit} holds, i. e. $g=g_{++}-g_{+-}-g_{-+}+g_{--}$, as $L^p(\Gamma)$ functions and almost everywhere in $\Gamma$.  It is also not hard to verify that $G(\omega_1,\omega_2)$ is holomorphic for $(\omega_1,\omega_2)\in(\C\backslash\Gamma_1)\times(\C\backslash\Gamma_2)$: for $\zeta=(\zeta_1,\zeta_2)\in(\C\backslash\Gamma_1)\times(\C\backslash\Gamma_2)$, we have the following power series representation
\begin{align*}
G(\omega_1,\omega_2)=\frac{1}{(2\pi i)^2}\sum_{k_1,k_2=0}^\infty\(\int_{\Gamma}\frac{g(\xi)d\xi}{(\xi_1-\zeta_1)^{k_1+1}(\xi_2-\zeta_2)^{k_2+1}}\)(\omega_1-\zeta_1)^{k_1}(\omega_2-\zeta_2)^{k_2},
\end{align*}
when $|\omega_1-\zeta_1|<\dist(\zeta_1,\Gamma_1)/2$ and $|\omega_2-\zeta_2|<\dist(\zeta_2,\Gamma_2)/2$.  Therefore $G$ is a holomorphic extension of $g$.
\end{proof}

\section{Littlewood-Paley Square Function Theory}

In this section, we develop some biparameter Littlewood-Paley-Stein theory.  We work in arbitrary dimension $\R^n$, where $n=n_1+n_2$.  We start by fixing some notation and defining biparameter Littlewood-Paley-Stein operators and square function.  For $k_j\in\Z$, $N_j>0$, and $x_j\in\R$
\begin{align*}
\Phi_{k_j}^{N_j}(x_j)=\frac{2^{n_jk_j}}{(1+2^{k_j}|x_j|)^{N_j}}
\end{align*}
for $j=1,2$.  Again we will use the subscripts of $k_j$, $N_j$, and $x_j$ to distinguish between functions on $\R^{n_1}$ and $\R^{n_2}$.  A collection of functions 
$\theta_{\vec k}:\R^{2n}\rightarrow\C$ for $\vec 
k\in\Z^2$ is a collection of biparameter 
Littlewood-Paley-Stein kernels if for all 
$x_1,y_1,x_1',y_1'\in\R^{n_1}$ and 
$x_2,y_2,x_2',y_2'\in\R^{n_2}$
\begin{align}
&|\theta_{\vec k}(x,y)|\less\Phi_{k_1}^{N_1+\gamma}(x_1-y_1)\Phi_{k_2}^{N_2+\gamma}(x_2-y_2)\label{size}\\
&|\theta_{\vec k}(x,y)-\theta_{\vec k}(x_1',x_2,y)|\less (2^{k_1}|x_1-x_1'|)^\gamma\notag\\
&\hspace{4cm}\times\(\Phi_{k_1}^{N_1+\gamma}(x_1-y_1)+\Phi_{k_1}^{N_1+\gamma}(x_1'-y_1)\)\Phi_{k_2}^{N_2}(x_2-y_2)\label{regx1}\\
&|\theta_{\vec k}(x,y)-\theta_{\vec k}(x_1,x_2',y)|\less (2^{k_2}|x_2-x_2'|)^{\gamma}\notag\\
&\hspace{4cm}\times\Phi_{k_1}^{N_1}(x_1-y_1)\(\Phi_{k_2}^{N_2+\gamma}(x_2-y_2)+\Phi_{k_2}^{N_2+\gamma}(x_2'-y_2)\)\label{regx2}\\
&|\theta_{\vec k}(x,y)-\theta_{\vec k}(x,y_1',y_2)|\less (2^{k_1}|y_1-y_1'|)^{\gamma}\notag\\
&\hspace{4cm}\times\(\Phi_{k_1}^{N_1+\gamma}(x_1-y_1)+\Phi_{k_1}^{N_1+\gamma}(x_1-y_1')\)\Phi_{k_2}^{N_2}(x_2-y_2)\label{regy1}
\end{align}
\begin{align}
&|\theta_{\vec k}(x,y)-\theta_{\vec k}(x,y_1,y_2')|\less (2^{k_2}|y_2-y_2'|)^{\gamma}\notag\\
&\hspace{4cm}\times\Phi_{k_1}^{N_1}(x_1-y_1)\(\Phi_{k_2}^{N_2+\gamma}(x_2-y_2)+\Phi_{k_2}^{N_2+\gamma}(x_2-y_2')\)\label{regy2}
\end{align}
for some $N_1>n_1$, $N_2>n_2$, and $0<\gamma\leq1$.  We say that a collection of operators $\Theta_{\vec k}$ for $\vec k\in\Z^2$ is a collection of biparameter Littlewood-Paley-Stein operators if 
\begin{align}
&\Theta_{\vec k}f(x)=\int_{\R^n}\theta_{\vec 
k}(x,y)f(y)dy.\label{theta}
\end{align}
for some collection of biparameter Littlewood-Paley-Stein kernels $\theta_{\vec k}$ satisfying \eqref{size}-\eqref{regy2}.

\begin{remark}\label{r:kernelequiv}
Properties \eqref{size}-\eqref{regy2} hold if and 
only if $\theta_{\vec k}$ satisfies the alternate 
condition set:
\begin{align*}
&|\theta_{\vec k}(x,y)|\less\Phi_{k_1}^{N_1'}(x_1-y_1)\Phi_{k_2}^{N_2'}(x_2-y_2),\\
&|\theta_{\vec k}(x,y)-\theta_{\vec k}(x_1',x_2,y)|\less 2^{n_1k_1}\,2^{n_2k_2}(2^{k_1}|x_1-x_1'|)^{\gamma'},\\
&|\theta_{\vec k}(x,y)-\theta_{\vec k}(x_1,x_2',y)|\less 2^{n_1k_1}\,2^{n_2k_2}(2^{k_2}|x_2-x_2'|)^{\gamma'},\\
&|\theta_{\vec k}(x,y)-\theta_{\vec k}(x,y_1',y_2)|\less 2^{n_1k_1}\,2^{n_2k_2}(2^{k_1}|y_1-y_1'|)^{\gamma'},\\
&|\theta_{\vec k}(x,y)-\theta_{\vec k}(x,y_1,y_2')|\less 2^{n_1k_1}\,2^{n_2k_2}(2^{k_2}|y_2-y_2'|)^{\gamma'}
\end{align*}
for some $N_1'>n_1$, $N_2'>n_2$, and $0<\gamma'\leq1$.
\end{remark}

\begin{proof}
It is obvious that \eqref{size}-\eqref{regy2} imply the above condition set since $\Phi_{k_j}^{N_j}(x_j)\leq2^{k_jn_j}$.  Assume there exist $N_1'>n_1$, $N_2'>n_2$, and $0<\gamma'\leq1$ such that the alternate condition set holds and choose $\eta\in(0,1)$ small enough so that $N_1=(1-\eta)N_1'-\eta\gamma'>n_1$ and $N_2=(1-\eta)N_2'-\eta\gamma'>n_2$, which is possible since $N_1'>n_1$ and $N_2'>n_2$.  Also define $\gamma=\eta\gamma'$, and it follows that
\begin{align*}
|\theta_{\vec k}(x,y)-\theta_{\vec k}(x_1',x_2,y)|&\less \(2^{k_1n_1}\,2^{k_2n_2}(2^{k_1}|x_1-x_1'|)^{\gamma'}\)^\eta\\
&\hspace{1.25cm}\times\(\Phi_{k_1}^{N_1'}(x_1-y_1)+\Phi_{k_1}^{N_1'}(x_1'-y_1)\)^{1-\eta}\Phi_{k_2}^{N_2'}(x_2-y_2)^{1-\eta}\\
&\hspace{-.65cm}\less(2^{k_1}|x_1-x_1'|)^{\gamma}\(\Phi_{k_1}^{N_1+\gamma}(x_1-y_1)+\Phi_{k_1}^{N_1+\gamma}(x_1'-y_1)\)\Phi_{k_2}^{N_2+\gamma}(x_2-y_2).
\end{align*}
The other conditions follow by symmetry, and hence the condition sets are equivalent.
\end{proof}

\ER We use the definition of para-accretive given by Han in \cite{Ha}. 

\begin{definition}
A function $b\in L^\infty(\R^n)$ is para-accretive 
if $b^{-1}\in L^\infty(\R^n)$ and
there exists a $c_0>0$ such that for all cubes 
$Q\subset\R^n$ there exists a cube $R\subset Q$ 
such that
\begin{align*}
\frac{1}{|Q|}\left|\int_Rb(x)dx\right|\geq c_0.
\end{align*}
\end{definition}

Let $\varphi\in C_0^\infty(\R^n)$ be non-negative with integral $1$ and $\supp(\varphi)\subset B(0,1/8)$.  Define for $x\in\R^n$, $k\in\Z$, and $f:\R^n\rightarrow\C$, $P_kf(x)=\varphi_k*f(x)$ where $\varphi_k(x)=2^{kn}\varphi(2^kx)$ and
\begin{align}
&S_k^bf(x)=P_kM_{(P_kb)^{-1}}P_kf(x)& &\text{and}& &D_k^bf(x)=S_{k+1}^bf(x)-S_k^bf(x).&\label{accapptoID}
\end{align}
Here $M_b$ is the pointwise multiplication operator defined by $M_bf(x)=b(x)f(x)$.  These operators were introduced by David-Journ\'e-Semmes in 
\cite{DJS}, and in that work it was proved that $|P_kb(x)|\geq Cc_0$ where the constant $C>0$ depends only on the dimension $n$.  It also follows that
\begin{align}
&\lim_{k\rightarrow\infty}S_{k}^bM_bf=f& &\text{and}& &\lim_{k\rightarrow\infty}S_{-k}^bM_bf=0&\label{reproducingformula}
\end{align}
in $L^p(\R^{n})$ for all $f\in L^p(\R^n)\cap L^q(\R^n)$ when $1<q< p<\infty$.  We also have the following properties for $S_k^b$ and $D_k^b$ and their kernels $s_k^b$ and $d_k^b$, see \cite{DJS} or \cite{Ha} for details:
\begin{align*}
&s_k^b(x,y)=d_k^b(x,y)=0\hspace{.5cm}\text{for }2^k|x-y|>1,\\
&|s_{k}^b(x,y)|+|d_k^b(x,y)|\less 2^{kn},\\
&|s_{k}^b(x,y)-s_{k}^b(x',y)|+|d_k^b(x,y)-d_k^b(x',y)|\less 2^{kn}(2^k|x-x'|)^\gamma,\\
&|s_{k}^b(x,y)-s_{k}^b(x,y')|+|d_k^b(x,y)-d_k^b(x,y')|\less 2^{kn}(2^k|y-y'|)^\gamma.
\end{align*}
Also let $\mathcal M_S$ be the biparameter strong maximal function
\begin{align*}
\mathcal M_Sf(x)=\sup_{Q_i\ni x_i}\frac{1}{|Q_1|\,|Q_2|}\int_{Q_1\times Q_2}|f(y_1,y_2)|dy_1\,dy_2
\end{align*}
where the supremum is taken over cubes $Q_1\subset\R^{n_1}$ and $Q_2\subset\R^{n_2}$.  It follows by standard arguments that for all $f\in L^1(\R^{n})+L^\infty(\R^n)$
\begin{align}
\sup_{k_1,k_2\in\Z}(\Phi_{k_1}^{N_1}\otimes\Phi_{k_2}^{N_2})*|f|(x)\less \mathcal M_Sf(x)
\end{align}
for any $N_1>n_1$ and $N_2>n_2$.  We now prove an almost orthogonality lemma.

\begin{lemma}\label{l:almostorth}
Assume that $\Theta_{\vec k}$ and $\Psi_{\vec k}$ are operators defined by \eqref{theta} with kernels respectively $\theta_{\vec k}$ and $\psi_{\vec k}$.  Also assume that $\theta_{\vec k}$ satisfies \eqref{size}, \eqref{regy1}, and \eqref{regy2} and that $\psi_{\vec k}$ satisfies \eqref{size}, \eqref{regx1}, and \eqref{regx2}. If there exist para-accretive functions $b_1\in L^\infty(\R^{n_1})$ and $b_2\in L^\infty(\R^{n_2})$ such that
\begin{align*}
&\int_{\R^{n_j}}\theta_{\vec k}(x,y)b_j(y_j)dy_j=\int_{\R^{n_j}}\psi_{\vec k}(x,y)b_j(x_j)dx_j=0
\end{align*}
for $j=1,2$ all $x\in\R^n$ and $k_1,k_2\in\Z$, then for all $\vec k=(k_1,k_2),\vec j=(j_1,j_2)\in\Z^2$
\begin{align*}
|\Theta_{\vec k}M_b\Psi_{\vec j}f(x)|\less 2^{-\epsilon|j_1-k_1|}2^{-\epsilon|j_2-k_2|} \mathcal M_Sf(x)
\end{align*}
for some $\epsilon>0$, where $b(x)=b_1(x_1)b_2(x_2)$ for $x=(x_1,x_2)\in\R^n$.
\end{lemma}

\begin{proof}
Using the cancellation of $\psi_{\vec j}$ and conditions \eqref{size} and \eqref{regy1}, it follows that
\begin{align*}
&\left|\int_{\R^n}\theta_{\vec k}(x,u)b(u)\psi_{\vec j}(u,y)du\right|\less\int_{\R^n}|\theta_{\vec k}(x,u)-\theta_{\vec k}(x,y_1,u_2)|\;|\psi_{\vec j}(u,y)|du\\
&\hspace{.5cm}\less\int_{\R^n}(2^{k_1}|u_1-y_1|)^{\gamma}\(\Phi_{k_1}^{N_1+\gamma}(x_1-u_1)+\Phi_{k_1}^{N_1+\gamma}(x_1-y_1)\)\Phi_{k_2}^{N_2+\gamma}(x_2-u_2)\Phi_{j_1}^{N_1+\gamma}(u_1-y_1)\Phi_{j_2}^{N_2+\gamma}(u_2-y_2)du\\
&\hspace{.5cm}=2^{\gamma(k_1- j_1)}\int_{\R^n}(2^{j_1}|u_1-y_1|)^{\gamma}\Phi_{j_1}^{N_1+\gamma}(u_1-y_1)\(\Phi_{k_1}^{N_1+\gamma}(x_1-u_1)+\Phi_{k_1}^{N_1+\gamma}(x_1-y_1)\)\\
&\hspace{4.5cm}\times\Phi_{k_2}^{N_2+\gamma}(x_2-u_2)\Phi_{j_2}^{N_2+\gamma}(u_2-y_2)du\\
&\hspace{.5cm}\leq2^{\gamma(k_1- j_1)}\int_{\R^n}\Phi_{j_1}^{N_1}(u_1-y_1)\(\Phi_{k_1}^{N_1+\gamma}(x_1-u_1)+\Phi_{k_1}^{N_1+\gamma}(x_1-y_1)\)du_1\int_{\R^n}\Phi_{k_2}^{N_2+\gamma}(x_2-u_2)\Phi_{j_2}^{N_2+\gamma}(u_2-y_2)du_2\\
&\hspace{.5cm}\less2^{\gamma(k_1- j_1)}\(\Phi_{k_1}^{N_1}(x_1-y_1)+\Phi_{j_1}^{N_1}(x_1-y_1)\)\(\Phi_{k_2}^{N_2}(x_2-y_2)+\Phi_{j_2}^{N_2}(x_2-y_2)\).
\end{align*}
By similar computations using the cancellation of $\theta_{\vec k}$, we have
\begin{align*}
&\left|\int_{\R^n}\theta_{\vec k}(x,u)b(u)\psi_{\vec j}(u,y)du\right|\\
&\hspace{1cm}\less2^{-\gamma(j_1-k_1)}\(\Phi_{k_1}^{N_1}(x_1-y_1)+\Phi_{j_1}^{N_1}(x_1-y_1)\)\(\Phi_{k_2}^{N_2}(x_2-y_2)+\Phi_{j_2}^{N_2}(x_2-y_2)\).
\end{align*}
Then it follows that
\begin{align*}
|\Theta_{\vec k}M_b\Psi_{\vec j}f(x)|\less2^{-\gamma|j_1-k_1|}\mathcal M_Sf(x).
\end{align*}
Our assumptions are symmetric in $k_1,j_1$ and $k_2,j_2$, so it follows that 
\begin{align*}
|\Theta_{\vec k}M_b\Psi_{\vec j}f(x)|\less2^{-\gamma|j_2-k_2|}\mathcal M_Sf(x).
\end{align*}
Then taking the geometric mean of these two estimates, we have
\begin{align*}
|\Theta_{\vec k}M_b\Psi_{\vec j}f(x)|&\less2^{-\gamma|j_1-k_1|/2}2^{-\gamma|j_2-k_2|/2}\mathcal M_Sf(x).
\end{align*}
This completes the proof.
\end{proof}
Given a para-accretive function $b$, let $S_k^b$ and $D_k^b=S_{k+1}^b-S_k^b$ be the operators from \eqref{accapptoID}.  Theorem 2.3 in \cite{Ha} says that there exist operators $\widetilde D_k^b$ for $k\in\Z$ such that
\begin{align}
\sum_{k\in\Z}\widetilde D_k^bM_bD_k^bM_bf=f\label{Hanformula}
\end{align}
in $L^p(\R^n)$ for any function $f:\R^n\rightarrow\C$ such that $|f(x)|\less\Phi_0^N(x)$ for some $N>n$, $|f(x)-f(y)|\less|x-y|^\gamma$ for some $\gamma>0$, and $bf$ has mean zero.  Furthermore, $\widetilde D_k^d$ is given by integration against its kernel $\widetilde d_k^b:\R^{2n}\rightarrow\C$,
\begin{align*}
\widetilde D_k^bf(x)=\int_{\R^n}\widetilde d_k^b(x,y)f(y)dy,
\end{align*}
and $\widetilde d_k^b$ satisfies
\begin{align*}
&|\widetilde d_k^b(x,y)|\less\Phi_k^{N+\gamma}(x-y),\\
&|\widetilde d_k^b(x,y)-\widetilde d_k^b(x',y)|\less(2^k|x-x'|)^\gamma\(\Phi_k^{N+\gamma}(x-y)+\Phi_k^{N+\gamma}(x'-y)\),\\
&\int_{\R^n}\widetilde d_k^b(x,y)b(y)dy=\int_{\R^n}\widetilde d_k^b(x,y)b(x)dx=0
\end{align*}
for some $N>n$ and $0<\gamma\leq1$.

\begin{lemma}\label{l:vvbound}
Let $b_1\in L^\infty(\R^{n_1})$ and $b_2\in L^\infty(\R^{n_2})$ be para-accretive functions and $D_{k_1}^{b_1}$ and $D_{k_2}^{b_2}$ be the operators defined above.  Also define $D_{\vec k}=D_{k_1}^{b_1} D_{k_2}^{b_2}$ for $\vec k\in\Z^2$.  Then 
\begin{align*}
\left|\left|\(\sum_{\vec k\in\Z^2}|D_{\vec k}f|^2\)^\frac{1}{2}\right|\right|_{L^p(\R^n)}\less||f||_{L^p(\R^n)}
\end{align*}
for $1<p<\infty$ and $f\in L^p(\R^n)$.  
\end{lemma}

This proof is essentially the same as the one due to R. Fefferman and Stein in Theorem 2 of \cite{rFS1}.  We reproduce the argument to demonstrate that there are no problems that arise by introducing para-accretive perturbations.

\begin{proof}
We start by viewing the operator $\{D_{k_1}^{b_1}\}$ defined initially from $L^2(\R^{n_1},\ell^2(\Z))$ into $L^2(\R^{n_1},\ell^2(\Z^2))$ in the following way:  for $\{F_{k_2}\}\in L^2(\R^{n_1},\ell^2(\Z))$, define
\begin{align*}
\{D_{k_1}^{b_1}\}(\{F_{k_2}\})(x_1)=\{D_{k_1}^{b_1}F_{k_2}(x_1)\}_{k_1,k_2\in\Z};\hspace{.25cm} \text{ for }x_1\in\R^{n_1}.
\end{align*}
Let $\{F_{k_2}\}\in L^2(\R^{n_1},\ell^2(\Z))$. For each $k_2\in\Z$, we use the square function bound for $D_{k_1}^{b_1}$ from \cite{DJS}, and it follows that
\begin{align*}
\int_{\R^{n_1}}\sum_{k_1\in\Z}|D_{k_1}^{b_1}F_{k_2}(x_1)|^2dx_1\less\int_{\R^{n_1}}|F_{k_2}(x_1)|^2dx_1.
\end{align*}
Then it follows that
\begin{align*}
||\{D_{k_1}^{b_1}\}(\{F_{k_2}\})||_{L^2(\R^{n_1},\ell^2(\Z^2))}^2&=\sum_{k_2\in\Z}\(\int_{\R^{n_1}}\sum_{k_1\in\Z}|D_{k_1}^{b_1}F_{k_2}(x_1)|^2dx_1\)\\
&\less\sum_{k_2\in\Z}\(\int_{\R^{n_1}}|F_{k_2}(x_1)|^2dx_1\)=||\{F_{k_2}\}||_{L^2(\R^n,\ell^2(\Z))}.
\end{align*}
That is, $\{D_{k_1}^{b_1}\}$ is bounded from $L^2(\R^{n_1},\ell^2(\Z))$ into $L^2(\R^{n_1},\ell^2(\Z^2))$.  Now the kernel of $\{D_{k_1}^{b_1}\}$ is given by $\{d_{k_1}^{b_1}(x_1,y_1)\}\in\mathcal L(\ell^2(\Z),\ell^2(\Z^2))$ for all $x_1,y_1\in\R^{n_1}$, where $\mathcal L(X,Y)$ for Banach spaces $X$ and $Y$ denotes the collection of all linear operators from $X$ into $Y$.  For fixed $x_1,y_1\in\R^{n_1}$, the kernel $\{d_{k_1}^{b_1}(x_1,y_1)\}$ is realized as a linear operator by the scalar multiplication:  $\{a_{k_2}\}\mapsto\{d_{k_1}^{b_1}(x_1,y_1)a_{k_2}\}_{(k_1,k_2)\in\Z^2}$.  Furthermore for $x_1\neq y_1$
\begin{align*}
||\{d_{k_1}^{b_1}(x_1,y_1)\}||_{\mathcal L(\ell^2(\Z),\ell^2(\Z^2))}&=\sup_{||\{a_{k_2}\}||_{\ell^2(\Z)}=1}||\{d_{k_1}^{b_1}(x_1,y_1)a_{k_2}\}||_{\ell^2(\Z^2)}\\
&=\sup_{||\{a_{k_2}\}||_{\ell^2(\Z)}=1}||\{d_{k_1}^{b_1}(x_1,y_1)\}||_{\ell^2(\Z)}||\{a_{k_2}\}||_{\ell^2(\Z)}\\
&=||\{d_{k_1}^{b_1}(x_1,y_1)\}||_{\ell^2(\Z)}\less\frac{1}{|x_1-y_1|^{n_1}}.
\end{align*}
The last inequality is a well-known vector-valued Calder\'on-Zygmund kernel result, see e.g. Coifman-Meyer \cite{CM1}.  It also follows that
\begin{align*}
&||\{d_{k_1}^{b_1}(x_1,y_1)\}-\{d_{k_1}^{b_1}(x_1',y_1)\}||_{\mathcal L(\ell^2(\Z),\ell^2(\Z^2))}\less\frac{|x_1-x_1'|^\gamma}{|x_1-y_1|^{n_1+\gamma}};\text{ for }|x_1-x_1'|<|x_1-y_1|/2,\\
&||\{d_{k_1}^{b_1}(x_1,y_1)\}-\{d_{k_1}^{b_1}(x_1,y_1')\}||_{\mathcal L(\ell^2(\Z),\ell^2(\Z^2))}\less\frac{|y_1-y_1'|^\gamma}{|x_1-y_1|^{n_1+\gamma}};\text{ for }|y_1-y_1'|<|x_1-y_1|/2.
\end{align*}
Then $\{D_{k_1}^{b_1}\}$ is bounded from $L^p(\R^{n_1},\ell^2(\Z))$ into $L^p(\R^{n_1},\ell^2(\Z^2))$ for $1<p<\infty$ by the vector-valued Calder\'on-Zygmund theory developed by Benedek-Calder\'on-Panzone in \cite{BCP} and by Rubio de Francia-Ruiz-Torrea in \cite{RdFRT}.  Alternatively, see Theorem 4.6.1 in Grafakos \cite{G} for a statement of the result applied here.  Now we fix $f\in L^p(\R^n)$ and define for $x_2\in\R^{n_2}$ and $k_2\in\Z$, 
\begin{align*}
F_{k_2}^{x_2}(x_1)=D_{k_2}^{b_2}f(x)=\int_{\R^{n_2}}d_{k_2}^{b_2}(x_2,y_2)f(x_1,y_2)dy_2.
\end{align*}
For almost every $x_2\in\R^{n_2}$, we have $\{F_{k_2}^{x_2}\}\in L^p(\R^{n_1},\ell^2(\Z))$ and hence
\begin{align}
\int_{\R^{n_1}}\(\sum_{\vec k\in\Z^2}|D_{\vec k}f(x)|^2\)^\frac{p}{2}dx_1&=\int_{\R^{n_1}}\(\sum_{\vec k\in\Z^2}|D_{k_1}^{b_1}F_{k_2}^{x_2}(x_1)|^2\)^\frac{p}{2}dx_1\notag\\
&=||\{D_{k_1}^{b_1}\}(\{F_{k_2}^{x_2}\})||_{L^p(\R^{n_1},\ell^2(\Z^2))}\notag\\
&\less||\{F_{k_2}^{x_2}\}||_{L^p(\R^{n_1},\ell^2(\Z))}=\int_{\R^{n_1}}\(\sum_{k_2\in\Z}|D_{k_2}^{b_2}f(x)|^2\)^\frac{p}{2}dx_1.\label{vvbound}
\end{align}
Now integrate both sides of \eqref{vvbound} in $x_2$, and using the square function bound for $D_{k_2}^{b_2}$, it follows that
\begin{align*}
\int_{\R^n}\(\sum_{\vec k\in\Z^2}|D_{\vec k}f(x)|^2\)^\frac{p}{2}dx&\less\int_{\R^{n_1}}\[\int_{\R^{n_2}}\(\sum_{k_2\in\Z}|D_{k_2}^{b_2}f(x)|^2\)^\frac{p}{2}dx_2\]dx_1\\
&\less\int_{\R^{n_1}}\[\int_{\R^{n_2}}|f(x)|^pdx_2\]dx_1=||f||_{L^p(\R^n)}^p.
\end{align*}
This completes the proof.
\end{proof}

We now prove Theorem \ref{t:squarefunction}, but first we specify precisely which assumptions on 
$\theta_{\vec k}$ are needed.  One need not assume that $\Theta_{\vec k}$ for $\vec k\in\Z^2$ is a 
collection of biparameter Littlewood-Paley-Stein operators as initially stated in Theorem 
\ref{t:squarefunction}.  Instead, we only need to assume that $\theta_{\vec k}$ satisfies 
\eqref{size}, \eqref{regy1}, and \eqref{regy2}.  In short, we can remove the assumption that 
$\theta_{\vec k}$ satisfies conditions \eqref{regx1} and \eqref{regx2} from Theorem \ref{t:squarefunction}.  In particular, this means that the square function associated to $\widetilde 
D_{\vec k}^*$ is bounded as well:  let $\widetilde D_{k_1}^{b_1}$ and $\widetilde D_{k_2}^{b_2}$ be 
the operators constructed in Theorem 2.3 from \cite{Ha}.  Define $\widetilde D_{\vec k}=\widetilde D_{k_1}^{b_1}\widetilde D_{k_2}^{b_2}$ for $\vec k\in\Z^2$, and it follows 
that
\begin{align*}
\left|\left|\(\sum_{\vec k\in\Z^2}|\widetilde D_{\vec k}^*f|^2\)^\frac{1}{2}\right|\right|_{L^p(\R^n)}\less||f||_{L^p}
\end{align*}
for all $f\in L^p(\R^n)$ when $1<p<\infty$.  Before we prove Theorem \ref{t:squarefunction}, we prove a lemma analogous to the result in Theorem 2.3 from \cite{Ha}.

\begin{lemma}\label{l:Econverge}
Let $b_1\in L^\infty(\R^{n_1})$ and $b_2\in L^\infty(\R^{n_2})$ be para-accretive functions and $b(x)=b_1(x_1)b_2(x_2)$ for $x=(x_1,x_2)\in\R^n$.   For $j=1,2$ let $D_{k_j}^{b_i}$ be as in \eqref{accapptoID} and $\widetilde D_{k_j}^{b_i}$ be as in \eqref{Hanformula} from Theorem 2.3 in \cite{Ha}.  Define $E_{k_j}^{b_j}=\widetilde D_{k_j}M_{b_j}D_{k_j}^{b_j}$ for $k_j\in \Z$ and $j=1,2$.  For any differentiable compactly supported function $f:\R^n\rightarrow\C$ such that
\begin{align*}
\int_{\R^{n_1}}f(x)b(x)dx_1=\int_{\R^{n_2}}f(x)b(x)dx_2=0
\end{align*}
for $x=(x_1,x_2)\in\R^n$, we have the following convergence
\begin{align*}
\lim_{T\rightarrow\infty}\left|\left|\sum_{|j_1|<T,|j_2|<N_T}E_{\vec j}M_bf-f\right|\right|_{L^p(\R^n)}=0
\end{align*}
for some sequence $N_T\geq T$.
\end{lemma}

\begin{proof}
Let $f:\R^n\rightarrow\C$ be differentiable and compactly supported such that
\begin{align*}
\int_{\R^{n_1}}f(x)b(x)dx_1=\int_{\R^{n_2}}f(x)b(x)dx_2=0.
\end{align*}
For each $x_2\in\R^{n_2}$, $f(\cdot,x_2)$ is differentiable, compactly supported, and $b_1\cdot f(\cdot,x_2)$ has mean zero.  Then by Theorem 2.3 in \cite{Ha}, for every $x_2\in\R^{n_2}$
\begin{align*}
&\lim_{T\rightarrow\infty}\left|\left|\sum_{|j_1|<T}E_{j_1}M_{b_1}f(\cdot,x_2)-f(\cdot,x_2)\right|\right|_{L^p(\R^{n_1})}=0
\end{align*}
Since $f$ is compactly supported and the above quantity is bounded uniformly in $T$, it follows by dominated convergence that
\begin{align}
&\lim_{T\rightarrow\infty}\left|\left|\sum_{|j_1|<T}E_{j_1}M_{b_1}f-f\right|\right|_{L^p(\R^n)}^p=\int_{\R^{n_2}}\lim_{T\rightarrow\infty}\left|\left|\sum_{|j_1|<T}E_{j_1}M_{b_1}f(\cdot,x_2)-f(\cdot,x_2)\right|\right|_{L^p(\R^{n_1})}^pdx_2=0.\label{limit1}
\end{align}
Define for each $T>0$
\begin{align*}
F_T^{x_1}(x_2)=\sum_{|j_1|<T}E_{j_1}M_{b_1}f(x_1,x_2).
\end{align*}
It follows that 
\begin{align*}
|F_T^{x_1}(x_2)|\leq\sum_{|j_1|<T}|E_{j_1}M_{b_1}f(x_1,x_2)|\leq2T\mathcal M_1f(x)\leq2T\sup_{x_1\in\R^{n_1}}|f(x_1,x_2)|.
\end{align*}
Therefore $F_T^{x_1}:\R^{n_2}\rightarrow\C$ is bounded (depending on $T$) and compactly supported.  Furthermore
\begin{align*}
|F_T^{x_1}(x_2)-F_T^{x_1}(y_2)|&\leq\sum_{|j_1|<T}|E_{j_1}M_{b_1}f(x_1,x_2)-f(x_1,y_2)|\\
&\leq\sum_{|j_1|<T}\int_{\R^{n_2}}|\widetilde d_{j_1}^{b_1}(x_2,u_2)-\widetilde d_{j_1}^{b_1}(y_2,u_2)||M_{b_1}D^{b_1}_{j_1}M_{b_1}f(x_1,u_2)|du_2\\
&\less\sum_{|j_1|<T}\int_{\R^{n_2}} (2^{j_1}|x_2-y_2|)^\gamma|D^{b_1}_{j_1}M_{b_1}f(x_1,u_2)|du_2\\
&\less2^{T}|x_2-y_2|^\gamma\sum_{|j_1|<T}||D^{b_1}_{j_1}M_{b_1}f(x_1,\cdot)||_{L^1(\R^{n_2})}\\
&\leq2^{T}|x_2-y_2|^\gamma\sum_{|j_1|<T}||f(x_1,\cdot)||_{L^1(\R^{n_2})}\leq T2^{T+1}||f(x_1,\cdot)||_{L^1(\R^{n_2})}|x_2-y_2|^\gamma.
\end{align*}
Finally, we have that
\begin{align*}
\int_{\R^{n_2}}F_T^{x_1}b_2(x_2)dx_2&=\sum_{|j_1|<T}E_{j_1}M_{b_1}\int_{\R^{n_2}}f(x_1,x_2)b_2(x_2)dx_2=0.
\end{align*}
Then by Theorem 2.3 from \cite{Ha}, it follow that
\begin{align*}
\lim_{N\rightarrow\infty}\left|\left|\sum_{|j_2|<N}E_{j_2}M_{b_2}F_T^{x_1}-F_T^{x_1}\right|\right|_{L^p(\R^{n_2})}=0.
\end{align*}
Then by dominated convergence
\begin{align}
&\lim_{N\rightarrow\infty}\left|\left|\sum_{|j_1|<T,|j_2|<N}E_{\vec j}M_bf-\sum_{|j_1|<T}E_{j_1}M_{b_1}f\right|\right|_{L^p(\R^{n_2})}^p\notag\\
&\hspace{3cm}=\int_{\R^{n_1}}\lim_{N\rightarrow\infty}\left|\left|\sum_{|j_2|<N}E_{j_2}M_{b_2}F_T^{x_1}-F_T^{x_1}\right|\right|_{L^p(\R^{n_2})}^pdx_1=0.\label{limit2}
\end{align}
For each $T>0$, using \eqref{limit2} there exists $N_T>T$ such that
\begin{align*}
&\left|\left|\sum_{|j_1|<T,|j_2|<N_T}E_{\vec j}M_bf-\sum_{|j_1|<T}E_{j_1}M_{b_1}f\right|\right|_{L^p(\R^{n_2})}<\frac{1}{T}.
\end{align*}
This defines the sequence $N_T$, and so now we verify the conclusion of Lemma \ref{l:Econverge}.  Let $\epsilon>0$.  Fix $M>\frac{2}{\epsilon}$ large enough so that for $T>M$
\begin{align*}
&\left|\left|\sum_{|j_1|<T}E_{j_1}M_{b_1}f-f\right|\right|_{L^p(\R^n)}<\frac{\epsilon}{2}.
\end{align*}
Then
\begin{align*}
\left|\left|\sum_{|j_1|<T,|j_2|<N_T}E_{\vec j}M_bf-f\right|\right|_{L^p(\R^n)}&=\left|\left|\sum_{|j_1|<T,|j_2|<N_T}E_{\vec j}M_bf-\sum_{|j_1|<T}E_{j_1}M_{b_1}f\right|\right|_{L^p(\R^n)}+\left|\left|\sum_{|j_1|<T}E_{j_1}M_{b_1}f-f\right|\right|_{L^p(\R^n)}\\
&\hspace{0cm}<\frac{1}{T}+\frac{\epsilon}{2}<\epsilon.
\end{align*}
This completes the proof.
\end{proof}

Now we prove Theorem \ref{t:squarefunction}.

\begin{proof}
Let $b(x)=b_1(x_1)b_2(x_2)$ for $x=(x_1,x_2)\in\R^n$, and $f,g_{\vec k}$ be differentiable, compactly supported such that 
\begin{align*}
\int_{\R^{n_1}}f(x)b(x)dx_1=\int_{\R^{n_2}}f(x)b(x)dx_2=0
\end{align*}
and
\begin{align*}
\left|\left|\(\sum_{\vec k\in\Z^2}|g_{\vec k}|^2\)^\frac{1}{2}\right|\right|_{L^{p'}(\R^n)}\leq1.
\end{align*}
Let  $R>1$, and define
\begin{align*}
\Lambda_R(f)=\sum_{|k_1|,|k_2|<R}\left|\int_{\R^n}\Theta_{\vec k}M_bf(x)g_{\vec k}(x)dx\right|,
\end{align*}
which satisfies 
\begin{align}
0\leq\Lambda_R(f)\less\int_{\R^n}\mathcal M_Sf(x)\sum_{|k_1|,|k_2|<R}|g_{\vec k}(x)|dx\less R ||f||_{L^p}.\label{LambdaRbound2}
\end{align}
Let $S_{k_j}^{b_j}$, $D_{k_j}^{b_j}=S_{k_j+1}^{b_j}-S_{k_j}^{b_j}$, $\widetilde D_{k_j}^{b_j}$, and $D_{\vec k}=D_{k_1}^{b_1}D_{k_2}^{b_2}$ be the operators defined in \eqref{accapptoID}.  Also define $E_{k_j}^{b_j}=\widetilde D_{k_j}^{b_j}M_{b_j}D_{k_j}^{b_j}$ and $E_{\vec k}=E_{k_1}^{b_1} E_{k_2}^{b_2}$, where $\widetilde D_{k_j}^{b_j}$ are the operators from \eqref{Hanformula} that were constructed in Theorem 2.3 of \cite{Ha}.  Let $f:\R^n\rightarrow\C$ be continuous, compactly supported such that
\begin{align*}
\int_{\R^{n_1}}f(x)b_1(x_1)dx_1=\int_{\R^{n_2}}f(x)b_2(x_2)dx_2=0
\end{align*}
for all $x=(x_1,x_2)\in\R^n$.  For $T>1$ it follows that
\begin{align*}
\Lambda_R(f)&\leq\sum_{|k_1|,|k_2|<R}\left|\int_{\R^n}\[\Theta_{\vec k}M_b-\Theta_{\vec k}M_b\(\sum_{|j_1|<T,|j_2|<N_T}E_{\vec j}M_b\)\]f(x)g_{\vec k}(x)dx\right|\\
&\hspace{2.32cm}+\sum_{|k_1|,|k_2|<R}\left|\sum_{|j_1|<T,|j_2|<N_T}\int_{\R^n}\Theta_{\vec k}M_bE_{\vec j}M_bf(x)g_{\vec k}(x)dx\right|=I_T+II_T.
\end{align*}
where $N_T$ are chosen as in Lemma \ref{l:Econverge}.  We first estimate $I_T$ using \eqref{LambdaRbound2}:
\begin{align*}
I_T&=\sum_{|k_1|,|k_2|<R}\left|\int_{\R^n}\[\Theta_{\vec k}M_b\(f(x)-\sum_{|j_1|<T,|j_2|<N_T}E_{\vec j}M_bf(x)\)\]g_{\vec k}(x)dx\right|\\
&\leq\Lambda_R\(f-\sum_{|j_1|<T,|j_2|<N_T}E_{\vec j}M_bf\)\less R\left|\left|f-\sum_{|j_1|<T,|j_2|<N_T}E_{\vec j}M_bf\right|\right|_{L^p},
\end{align*}
which tends to $0$ as $T\rightarrow\infty$ by Lemma \ref{l:Econverge}.  Now we estimate $II_T$ by putting the absolute value inside and summing more terms,
\begin{align*}
II_T&\leq\sum_{\vec k,\vec j\in\Z^2}\int_{\R^n}|\Theta_{\vec k}M_bE_{\vec j}M_bf(x)g_{\vec k}(x)|dx,
\end{align*}
So we now estimate $II_T$.  By Lemma \ref{l:almostorth}, there exists $\epsilon>0$ such that
\begin{align*}
|\Theta_{\vec k}M_bE_{\vec j}f(x)|\less 2^{-\epsilon|k_1-j_1|}2^{-\epsilon|k_2-j_2|}\mathcal M_SD_{\vec j}M_bf(x).
\end{align*}
Then it follows that
\begin{align*}
\Lambda_R(f)&\leq\int_{\R^n}\sum_{\vec j,\vec k\in\Z^2}|\Theta_{\vec k}M_bE_{\vec j}M_bf(x)g_{\vec k}(x)|dx\\
&\hspace{0cm}\less\int_{\R^n}\sum_{\vec j,\vec k\in\Z^2}2^{-\frac{\epsilon}{2}(|k_1-j_1|+|k_2-j_2|)}\mathcal M_S\(D_{\vec j}M_bf\)(x)|g_{\vec k}(x)|dx\\
&\hspace{0cm}\leq\left|\left|\(\sum_{\vec j,\vec k\in\Z^2}2^{-\frac{\epsilon}{2}(|k_1-j_1|+|k_2-j_2|)}\[\mathcal M_S\(D_{\vec j}M_bf\)\]^2\)^\frac{1}{2}\right|\right|_{L^p(\R^n)}\left|\left|\(\sum_{\vec j,\vec k\in\Z^2}2^{-\frac{\epsilon}{2}(|k_1-j_1|+|k_2-j_2|)}|g_{\vec k}|^2\)^\frac{1}{2}\right|\right|_{L^{p'}(\R^n)}\\
&\hspace{0cm}\less\left|\left|\(\sum_{\vec j\in\Z^2}\[\mathcal M_S\(D_{\vec j}M_bf\)\]^2\)^\frac{1}{2}\right|\right|_{L^p(\R^n)}\left|\left|\(\sum_{\vec k\in\Z^2}|g_{\vec k}|^2\)^\frac{1}{2}\right|\right|_{L^{p'}(\R^n)}\hspace{0cm}\less\left|\left|\(\sum_{\vec j\in\Z^2}|D_{\vec j}M_bf|^2\)^\frac{1}{2}\right|\right|_{L^p(\R^n)}\less||f||_{L^p(\R^n)}.
\end{align*}
In the last two lines we use the Fefferman-Stein strong maximal function bound from \cite{rFS1} twice and the multiparameter Littlewood-Paley bound from Lemma \ref{l:vvbound}.  The estimate for general functions $f\in L^p(\R^n)$ follows by density.
\end{proof}

Next we prove a sort of dual pairing bound for biparameter Littlewood-Paley-Stein operators.  This is the estimate that we use to bound the truncations of singular integral operators in the next section.

\begin{proposition}\label{p:dualpairbound}
Let $\Theta_{\vec k}$ be a collection of biparameter Littlewood-Paley-Stein operators with kernels $\theta_{\vec k}$ for $\vec k\in\Z^2$ and $b_1,\tilde b_1\in L^\infty(\R^{n_1})$ and $b_2,\tilde b_2\in L^\infty(\R^{n_2})$ be para-accretive functions.  If
\begin{align*}
\int_{\R^{n_j}}\theta_{\vec k}(x,y)b_j(y_j)dy_j=\int_{\R^{n_j}}\theta_{\vec k}(x,y)\tilde b_j(x_j)dx_j=0
\end{align*}
for $j=1,2$, then for all $f\in L^p(\R^n)$ and $g\in L^{p'}(\R^n)$
\begin{align*}
\sum_{k_1,k_2\in\Z}\left|\int_{\R^2}\Theta_{\vec k}M_bf(x)\tilde b(x)g(x)dx\right|\less||f||_{L^p(\R^n)}||g||_{L^{p'}(\R^n)},
\end{align*}
where $b(x)=b_1(x_1)b_2(x_2)$ and $\tilde b(x)=\tilde b_1(x_1)\tilde b_2(x_2)$ for $x=(x_1,x_2)\in\R^n$.
\end{proposition}

\begin{proof}
Let $f,g$ be differentiable, compactly supported functions such that 
\begin{align*}
\int_{\R^{n_1}}f(x)b(x)dx_1=\int_{\R^{n_2}}f(x)b(x)dx_2=\int_{\R^{n_1}}g(x)\tilde b(x)dx_1=\int_{\R^{n_2}}g(x)\tilde b(x)dx_2=0.
\end{align*}
Define for $R>1$ 
\begin{align*}
\Lambda_R(f,g)=\sum_{|k_1|,|k_2|<R}\left|\int_{\R^n}\Theta_{\vec k}M_bf(x)\tilde b(x)g(x)dx\right|,
\end{align*}
which satisfies 
\begin{align}
0\leq\Lambda_R(f,g)\less\sum_{|k_1|,|k_2|<R}||\mathcal M_Sf||_{L^p(\R^n)}||g||_{L^{p'}(\R^n)}\less R^2||f||_{L^p}||g||_{L^{p'}}.\label{LambdaRbound}
\end{align}
Let $S_{k_j}^{b_j}$, $D_{k_j}^{b_j}=S_{k_j+1}^{b_j}-S_{k_j}^{b_j}$, $\widetilde D_{k_j}^{b_j}$, $D_{\vec k}^b=D_{k_1}^{b_1}D_{k_2}^{b_2}$, and $\widetilde D_{\vec k}^b=\widetilde D_{k_1}^{b_1}\widetilde D_{k_2}^{b_2}$ be the operators defined in \eqref{accapptoID}.  Also define $E_{k_j}^{b_j}=\widetilde D_{k_j}^{b_j}M_{b_j}D_{k_j}^{b_j}$ and $E_{\vec k}^b=E_{k_1}^{b_1} E_{k_2}^{b_2}$, where $\widetilde D_{k_j}^{b_j}$ are the operators constructed in Theorem 2.3 in \cite{Ha}.  We also construct the corresponding operators with $b_j$ replaced by $\tilde b_j$.  Then for $f,g\in C_0^\delta(\R^n)$ for some $0<\delta\leq1$ where $bf$ and $\tilde bg$ have mean zero in both $x_1$ and $x_2$, it follows that
\begin{align*}
\Lambda_R&(f,g)\leq \limsup_{T\rightarrow\infty}\; I_T+II_T+III_T,
\end{align*}
where
\begin{align*}
&I_T=\sum_{|k_1|,|k_2|<R}\left|\int_{\R^n}\[\Theta_{\vec k}M_b-\Theta_{\vec k}M_b\(\sum_{|j_1|<T,|j_2|<N_T}E_{\vec j}^bM_b\)\]f(x)M_{\tilde b}g(x)dx\right|,\\
&II_T=\sum_{|k_1|,|k_2|<R}\left|\int_{\R^n}\[\Theta_{\vec k}M_b\(\sum_{|j_1|<T,|j_2|<N_T}E_{\vec j}^bM_b\)\right.\right.\\
&\hspace{2cm}\left.\left.-\(\sum_{|m_1|<T,|m_2|<M_T}E_{\vec m}^{\tilde b}M_{\tilde b}\)\Theta_{\vec k}M_b\(\sum_{|j_1|<T,|j_2|<N_T}E_{\vec j}^bM_b\)\]f(x)M_{\tilde b}g(x)dx\right|,\\
&III_T=\sum_{|k_1|,|k_2|<R}\left|\sum_{|j_1|<T,|j_2|<N_T,|m_1|<T,|m_2|<M_T}\int_{\R^n}E_{\vec m}^{\tilde b}M_{\tilde b}\Theta_{\vec k}M_bE_{\vec j}^bM_bf(x)M_{\tilde b}g(x)dx\right|,
\end{align*}
where $N_T$ and $M_T$ are chosen as in Lemma \ref{l:Econverge} for $f$ and $g$ respectively.  We first estimate $I_T$ using \eqref{LambdaRbound} and Lemma \ref{l:Econverge}:
\begin{align*}
I_T&=\sum_{|k_1|,|k_2|<R}\left|\int_{\R^n}\[\Theta_{\vec k}M_b\(f(x)-\sum_{|j_1|<T,|j_2|<N_T}E_{\vec j}^bM_bf(x)\)\]M_{\tilde b}g(x)dx\right|\\
&\leq\Lambda_R\(f-\sum_{|j_1|<T,|j_2|<N_T}E_{\vec j}^bM_bf,g\)\less R\left|\left|f-\sum_{|j_1|<T,|j_2|<N_T}E_{\vec j}^bM_bf\right|\right|_{L^p(\R^n)}||g||_{L^{p'}(\R^n)},
\end{align*}
which tends to $0$ as $T\rightarrow\infty$.  Now we estimate $II_T$ again using \eqref{LambdaRbound} and Lemma \ref{l:Econverge},
\begin{align*}
II_T&=\sum_{|k_1|,|k_2|<R}\left|\int_{\R^n}\[{\bf I}-\sum_{|m_1|<T,|m_2|<M_T}E_{\vec m}^{\tilde b}M_{\tilde b}\]\Theta_{\vec k}M_b \(\sum_{|j_1|<T,|j_2|<N_T}E_{\vec k}^bM_b \)f(x)M_{\tilde b} g(x)dx\right|\\
&=\Lambda_R\(\sum_{|j_1|<T,|j_2|<N_T}E_{\vec j}^bM_bf,g-\sum_{|m_1|<T,|m_2|<M_T}E_{\vec m}^{\tilde b}M_{\tilde b}g\)\\
&\less R\left|\left|\sum_{|j_1|<T,|j_2|<N_T}E_{\vec j}^bM_bf\right|\right|_{L^p(\R^n)}\left|\left|g-\sum_{|m_1|<T,|m_2|<M_T}E_{\vec m}^{\tilde b}M_{\tilde b}g\right|\right|_{L^{p'}(\R^n)}\\
&\less R||f||_{L^p(\R^n)}\left|\left|g-\sum_{|m_1|<T,|m_2|<M_T}E_{\vec m}^{\tilde b}M_{\tilde b}g\right|\right|_{L^{p'}(\R^n)},
\end{align*}
where ${\bf I}$ is the identity operator.  This term also tends to $0$ as $T\rightarrow\infty$ by Lemma \ref{l:Econverge}.  So we are left with the third term, to estimate $\Lambda_R$
\begin{align}
\Lambda_R(f,g)&\leq\limsup_{T\rightarrow\infty}\sum_{|k_1|,|k_2|<R}\left|\sum_{|j_1|<T,|j_2|<N_T,|m_1|<T,|m_2|<M_T}\int_{\R^n}E_{\vec m}^{\tilde b}M_{\tilde b}\Theta_{\vec k}M_bE_{\vec j}^bM_bf(x)M_{\tilde b}g(x)dx\right|\notag\\
&\leq\sum_{\vec k,\vec j,\vec m\in\Z^2}\left|\int_{\R^n}M_bD_{\vec m}^{\tilde b}M_{\tilde b}\Theta_{\vec k}M_bE_{\vec j}^bM_bf(x)(\widetilde D_{\vec m}^{\tilde b})^*M_{\tilde b}g(x)dx\right|.\label{III}
\end{align}
So we now estimate \eqref{III}.  By Lemma \ref{l:almostorth}, there exists $\epsilon>0$ such that
\begin{align*}
&|D_{\vec m}^{\tilde b}M_{\tilde b}\Theta_{\vec k}M_bE_{\vec j}^bf(x)|\less 2^{-\epsilon|m_1-k_1|}2^{-\epsilon|m_2-k_2|}\mathcal M_S^2D_{\vec j}^bf(x),\hspace{.25cm}\text{and}\\
&|D_{\vec m}^{\tilde b}M_{\tilde b}\Theta_{\vec k}M_bE_{\vec j}^bf(x)|\less\mathcal M_S(\Theta_{\vec k}M_bE_{\vec j}^bf)(x)\less 2^{-\epsilon|k_1-j_1|}2^{-\epsilon|k_2-j_2|}\mathcal M_S^2D_{\vec j}^bf(x).
\end{align*}
Therefore we also have
\begin{align}
|D_{\vec m}^{\tilde b}M_{\tilde b}\Theta_{\vec k}M_bE_{\vec j}^bf(x)|\less2^{-\frac{\epsilon}{2}(|m_1-k_1|+|m_2-k_2|+|k_1-j_1|+|k_2-j_2|}\mathcal M_S^2D_{\vec j}^bf(x).\label{IIIAObound}
\end{align}
Using \eqref{IIIAObound} we have
\begin{align*}
&\int_{\R^n}\sum_{\vec j,\vec k,\vec m\in\Z^2}|M_{\tilde b}D_{\vec m}^{\tilde b}M_{\tilde b}\Theta_{\vec k}M_bE_{\vec j}^bM_bf(x)(\widetilde D_{\vec m}^{\tilde b})^*M_{\tilde b}g(x)|dx\less\int_{\R^n}\sum_{\vec j,\vec k,\vec m\in\Z^2}2^{-\frac{\epsilon}{2}(|m_1-k_1|+|m_2-k_2|+|k_1-j_1|+|k_2-j_2|)}\\
&\hspace{10cm}\times\mathcal M_S^2\(D_{\vec j}^bM_bf\)(x)(\widetilde D_{\vec m}^{\tilde b})^*M_{\tilde b}g(x)|dx\\
&\hspace{.5cm}\leq\left|\left|\(\sum_{\vec j,\vec k,\vec m\in\Z^2}2^{-\frac{\epsilon}{2}(|m_1-k_1|+|m_2-k_2|+|k_1-j_1|+|k_2-j_2|)}\[\mathcal M_S^2\(D_{\vec j}^bM_bf\)\]^2\)^\frac{1}{2}\right|\right|_{L^p(\R^n)}\\
&\hspace{2cm}\times\left|\left|\(\sum_{\vec j,\vec k,\vec m\in\Z^2}2^{-\frac{\epsilon}{2}(|m_1-k_1|+|m_2-k_2|+|k_1-j_1|+|k_2-j_2|)}|(\widetilde D_{\vec m}^{\tilde b})^*M_{\tilde b}g|^2\)^\frac{1}{2}\right|\right|_{L^{p'}(\R^n)}\\
&\hspace{.5cm}\less\left|\left|\(\sum_{\vec j\in\Z^2}\[\mathcal M_S^2\(D_{\vec j}^bM_bf\)\]^2\)^\frac{1}{2}\right|\right|_{L^p(\R^n)}\left|\left|\(\sum_{\vec m\in\Z^2}|(\widetilde D_{\vec m}^{\tilde b})^*M_{\tilde b}g|^2\)^\frac{1}{2}\right|\right|_{L^{p'}(\R^n)}\\
&\hspace{.5cm}\less\left|\left|\(\sum_{\vec j\in\Z^2}|D_{\vec j}^bM_bf|^2\)^\frac{1}{2}\right|\right|_{L^p(\R^n)}||g||_{L^{p'}(\R^n)}\less||f||_{L^p(\R^n)}||g||_{L^{p'}(\R^n)}.
\end{align*}
In the last two lines we use the Fefferman-Stein maximal function bound from \cite{rFS1} twice and the biparameter Littlewood-Paley-Stein bound proved  in Theorem \ref{t:squarefunction}.  Recall that the square function associated to $(\widetilde D_{\vec m}^{\tilde b})^*$ is bounded on $L^p(\R^n)$ for $1<p<\infty$.  The estimate for general functions $f\in L^p(\R^n)$ and $g\in L^{p'}(\R^n)$ follows by density.
\end{proof}

\section{A Biparameter Tb Theorem}

We define the class of test functions that will be used to define biparameter singular integral operators.  Define $C_0^{0,\delta}(\R^n)$ to be the collection of all $\delta$-H\"older continuous, compactly supported functions $f:\R^n\rightarrow\C$ with norm
\begin{align*}
||f||_\delta=\sup_{x\neq y}\frac{|f(x)-f(y)|}{|x-y|^\delta}<\infty.
\end{align*}
Since $C_0^{0,\delta}(\R^n)$ is made up of compactly supported functions, it follows that $||\cdot||_\delta$ is a norm, and we endow $C_0^{0,\delta}(\R^n)$ the topology generated by the norm $||\cdot||_\delta$.  Given a function $b\in L^\infty(\R^n)$ such that $b^{-1}\in L^\infty(\R^n)$, let $bC_0^{0,\delta}(\R^n)$ be the collection of functions $bf$ such that $f\in C_0^{0,\delta}(\R^n)$.  We define $||bf||_{b,\delta}=||f||_{\delta}$ for $bf\in bC_0^{0,\delta}(\R^n)$, and endow $bC_0^{0,\delta}(\R^n)$ the topology generated by the norm $||\cdot||_{b,\delta}$.  Finally, given a function space $X$, we define $X'$ to be the continuous dual of $X$ with the weak$^*$ topology.  In our situation, we will primarily use this definition for $X=bC_0^{0,\delta}(\R^n)$.

\begin{definition}
We say that $K$ a standard biparameter kernel on $\R^n=\R^{n_1}\times\R^{n_2}$ if
\begin{align}
&|K(x,y)|\less\frac{1}{|x_1-y_1|^{n_1}\,|x_2-y_2|^{n_2}}\hspace{.25cm}\text{ for }|x_1-y_1|,|x_2-y_2|\neq0\\
&|K(x,y)-K(x_1',x_2,y)-K(x_1,x_2',y)+K(x_1',x_2',y)|\less\frac{|x_1-x_1'|^{\gamma}|x_2-x_2'|^{\gamma}}{|x_1-y_1|^{n_1+\gamma}|x_2-y_2|^{n_2+\gamma}}\label{regx}\\
&\hspace{2.9cm}\text{ whenever }|x_1-x_1'|<|x_1-y_1|/2\text{ and }|x_2-x_2'|<|x_2-y_2|/2\notag,
\end{align}
\begin{align}
&|K(x,y)-K(x,y_1',y_2)-K(x,y_1,y_2')+K(x,y_1',y_2')|\less\frac{|y_1-y_1'|^{\gamma}|y_2-y_2'|^{\gamma}}{|x_1-y_1|^{n_1+\gamma}|x_2-y_2|^{n_2+\gamma}}\label{regy}\\
&\hspace{2.9cm}\text{ whenever }|y_1-y_1'|<|x_1-y_1|/2\text{ and }|y_2-y_2'|<|x_2-y_2|/2.\notag,\\
&|K(x,y)-K(x,y_1',y_2)-K(x_1,x_2',y)+K(x_1,x_2',y_1',y_2)|\less\frac{|y_1-y_1'|^{\gamma}|x_2-x_2'|^{\gamma}}{|x_1-y_1|^{n_1+\gamma}|x_2-y_2|^{n_2+\gamma}}\label{mreg1}\\
&\hspace{2.9cm}\text{ whenever }|y_1-y_1'|<|x_1-y_1|/2\text{ and }|x_2-x_2'|<|x_2-y_2|/2.\notag,\\
&|K(x,y)-K(x,y_1,y_2')-K(x_1',x_2,y)+K(x_1',x_2,y_1,y_2')|\less\frac{|x_1-x_1'|^{\gamma}|y_2-y_2'|^{\gamma}}{|x_1-y_1|^{n_1+\gamma}|x_2-y_2|^{n_2+\gamma}}\label{mreg1}\\
&\hspace{2.9cm}\text{ whenever }|x_1-x_1'|<|x_1-y_1|/2\text{ and }|y_2-y_2'|<|x_2-y_2|/2.\notag
\end{align}
Let $b_1,\tilde b_1\in L^\infty(\R^{n_1})$ and $b_2,\tilde b_2\in L^\infty(\R^{n_2})$ be para-accretive functions and define $b(x)=b_1(x_1)b_2(x_2)$ and $\tilde b(x)=\tilde b_1(x_1)\tilde b_2(x_2)$ for $x=(x_1,x_2)\in\R^n$.  A linear operator $T$ that is continuous from $bC_0^{0,\delta}(\R^n)$ into $(\tilde bC_0^{0,\delta}(\R^n))'$ for some $0<\delta\leq1$ is a biparameter singular integral operator of Calder\'on-Zygmund type associated to $b,\tilde b$ if
\begin{align*}
\<M_{\tilde b}TM_bf,g\>=\int_{\R^{2n}}K(x,y)f(y)g(x)\tilde b(x)b(y)dx\,dy
\end{align*}
is an absolutely convergent integral whenever $f,g\in C_0^{0,\delta}(\R^n)$ and
\begin{align*}
\bigcup_{x_1,y_1\in\R^{n_1}}\supp(f(y_1,\cdot))\cap\supp(g(x_1,\cdot))=\bigcup_{x_2,y_2\in\R^{n_2}}\supp(f(\cdot,y_2))\cap\supp(g(\cdot,x_2))=\emptyset.
\end{align*}

\end{definition}

\begin{definition}\label{d:WBP}
A function $\phi\in C_0^\infty(\R^n)$ is a normalized bump of order $m\in\N$ if $\supp(\phi)\subset B(0,1)\subset\R^n$ and for all $\alpha\in\N_0^n$ with $|\alpha|\leq m$
\begin{align*}
||\partial^\alpha\phi||_{L^\infty(\R^n)}\leq1.
\end{align*}
Let $T$ be a biparameter singular integral operator of Calder\'on-Zygmund type associated to $b(x)=b_1(x_1)b_2(x_2)$ and $\tilde b(x)=\tilde b_1(x_1)\tilde b_2(x_2)$ for $x=(x_1,x_2)\in\R^n$, where $b_1,\tilde b_1\in L^\infty(\R^{n_1})$ and $b_2,\tilde b_2\in L^\infty(\R^{n_2})$ are para-accretive functions.  We say $T$ satisfies the biparameter weak boundedness property if there exists $m\in\N$ such that the following holds:   let $\varphi_j,\psi_j\in C_0^\infty(\R^{n_j})$ be normalized bumps of order $m$.  Let $x=(x_1,x_2)\in\R^n$ and $R_1,R_2>0$.  Assume that either $b_1\varphi_1^{x_1,R_1}$ or $\tilde b_1\psi_1^{x_1,R_1} $ has mean zero and that either $b_2\varphi_2^{x_2,R_2}$ or $\tilde b_2\psi_2^{x_2,R_2} $ has mean zero.  Then 
\begin{align}
&\left|\<M_{\tilde b}TM_b(\varphi_1^{x_1,R_1}\otimes \varphi_2^{x_2,R_2}),\psi_1^{x_1,R_1}\otimes \psi_2^{x_2,R_2}\>\right|\less R_1^{n_1}R_2^{n_2},\label{WBP1}
\end{align}
where $\phi^{x_j,R_j}(u_j)=\varphi\(\frac{u_j-x_j}{R_j}\)$.
\end{definition}

\begin{definition}\label{d:mixedWBP}
Let $T$ be a biparameter singular integral operator of Calder\'on-Zygmund type associated to  $b(x)=b_1(x_1)b_2(x_2)$ and $\tilde b(x)=\tilde b_1(x_1)\tilde b_2(x_2)$ for $x=(x_1,x_2)\in\R^n$, where $b_1,\tilde b_1\in L^\infty(\R^{n_1})$ and $b_2,\tilde b_2\in L^\infty(\R^{n_2})$ are para-accretive functions.  We say $T$ satisfies the mixed biparameter weak boundedness property if there exists $m\in\N$ and $0<\gamma\leq1$ such that the following two conditions hold: (1)  Let be $R_1,R_2>0$, $x_1,y_1\in\R^{n_1}$ with $|x_1-y_1|>4R_1$, and $x_2\in\R^{n_2}$ and let $\varphi_j,\psi_j\in C_0^\infty(\R^{n_j})$ be normalized bumps of order $m$.  Then
\begin{align}
&\left|\<M_{\tilde b}TM_b(\varphi_1^{y_1,R_1}\otimes \varphi_2^{x_2,R_2}),\psi_1^{x_1,R_1}\otimes \psi_2^{x_2,R_2}\>\right|\less \frac{R_1^{n_1}R_2^{n_2}}{(R_1^{-1}|x_1-y_1|)^{n_1}}.\label{WBP2}
\end{align}
Further assume that either $b_1\varphi_1^{y_1,R_1}$ or $\tilde b_1\psi_1^{x_1,R_1} $ has mean zero and that either $b_2\varphi_2^{x_2,R_2}$ or $\tilde b_2\psi_2^{x_2,R_2} $ has mean zero.  Then
\begin{align}
&\left|\<M_{\tilde b}TM_b(\varphi_1^{y_1,R_1}\otimes \varphi_2^{x_2,R_2}),\psi_1^{x_1,R_1}\otimes \psi_2^{x_2,R_2}\>\right|\less \frac{R_1^{n_1}R_2^{n_2}}{(R_1^{-1}|x_1-y_1|)^{n_1+\gamma}}.\label{WBP3}
\end{align}
(2)  Let be $R_1,R_2>0$, $x_2,y_2\in\R^{n_1}$ with $|x_2-y_2|>4R_2$, and $x_2\in\R^{n_2}$ and let $\varphi_j,\psi_j\in C_0^\infty(\R^{n_j})$ be normalized bumps of order $m$.  Then
\begin{align}
&\left|\<M_{\tilde b}TM_b(\varphi_1^{x_1,R_1}\otimes \varphi_2^{y_2,R_2}),\psi_1^{x_1,R_1}\otimes \psi_2^{x_2,R_2}\>\right|\less \frac{R_1^{n_1}R_2^{n_2}}{(R_2^{-1}|x_2-y_2|)^{n_2}}.\label{WBP4}
\end{align}
Further assume that either $b_1\varphi_1^{x_1,R_1}$ or $\tilde b_1\psi_1^{x_1,R_1} $ has mean zero and that either $b_2\varphi_2^{y_2,R_2}$ or $\tilde b_2\psi_2^{x_2,R_2} $ has mean zero.    Then, 
\begin{align}
&\left|\<M_{\tilde b}TM_b(\varphi_1^{x_1,R_1}\otimes \varphi_2^{y_2,R_2}),\psi_1^{x_1,R_1}\otimes \psi_2^{x_2,R_2}\>\right|\less \frac{R_1^{n_1}R_2^{n_2}}{(R_2^{-1}|x_2-y_2|)^{n_2+\gamma}}.\label{WBP5}
\end{align}

\end{definition}

\begin{lemma}\label{l:kernelconditions}
Suppose $b_1,\tilde b_1\in L^\infty(\R^{n_1})$ and $b_2,\tilde b_2\in L^\infty(\R^{n_2})$ are para-accretive functions, and define $b(x)=b_1(x_1)b_2(x_2)$ and $\tilde b(x)=\tilde b_1(x_1)\tilde b_2(x_2)$ for $x=(x_1,x_2)\in\R^n$.  Let $T$ be a biparameter singular integral operator of Calder\'on-Zygmund type associated to $b$ and $\tilde b$ with standard biparameter kernel $K$.  Also assume that $M_{\tilde b}TM_{b}$ satisfies the biparameter weak boundedness and the mixed weak boundedness properties.  Define $\Theta_{\vec k}$ for $\vec k\in\Z^2$ by integration against its kernel $\theta_{\vec k}$, as in \eqref{theta}, where
\begin{align}
\theta_{\vec k}(x,y)=\<M_{\tilde b}TM_b(s_{k_1}^{b_1}(\cdot,y_1)\otimes s_{k_2}^{b_2}(\cdot,y_2)),d_{k_1}^{\tilde b_1}(x_1,\cdot)\otimes d_{k_2}^{\tilde b_2}(x_2,\cdot)\>.
\end{align}
Then $\Theta_{\vec k}$ for $\vec k\in\Z^2$ is a collection of Littlewood-Paley-Stein operators and
\begin{align*}
\int_{\R^{n_1}}\theta_{\vec k}(x,y)\tilde b_1(x_1)dx_1=\int_{\R^{n_2}}\theta_{\vec k}(x,y)\tilde b_2(x_2)dx_2=0.
\end{align*}
\end{lemma}

\begin{proof}
Fix $x,y\in\R^n$ such that $|x_1-y_1|\leq2^{-k_1+2}$ and $|x_2-y_2|\leq2^{-k_2+2}$.  Then using \eqref{WBP1} 
\begin{align*}
&|\theta_{\vec k}(x,y)|\\
&=2^{2k_1n_1}2^{2k_2n_2}\left|\<M_{\tilde b}TM_b\(\phi_1^{\frac{x_1+y_1}{2},2^{-k_1+2}}\otimes\phi_2^{\frac{x_2+y_2}{2},2^{-k_2+2}}\),\phi_3^{\frac{x_1+y_1}{2},2^{-k_1+2}}\otimes\phi_4^{\frac{x_2+y_2}{2},2^{-k_2+2}}\>\right|\\
&\less 2^{k_1n_1}2^{k_2n_2}\less\Phi_{k_1}^{n_1+\gamma}(x_1-y_1)\Phi_{k_2}^{n_2+\gamma}(x_2-y_2).
\end{align*}
where $\phi_1,\phi_2,\phi_3,\phi_4$ are normalized bumps of order $m$ (up to a constant multiple independent of $x$, $y$, and $\vec k$) of the form
\begin{align*}
&\phi_1(u_1)=2^{-k_1n_1}s_{k_1}^{b_1}\(2^{-k_1+2}u_1+\frac{x_1+y_1}{2},y_1\),& && &\phi_2(u_2)=2^{-k_2n_2}s_{k_2}^{b_2}\(2^{-k_2+2}u_1+\frac{x_2+y_2}{2},y_2\),&\\
&\phi_3(v_1)=2^{-k_1n_1}d_{k_1}^{\tilde b_1}\(x_1,2^{-k_1+2}v_1+\frac{x_1+y_1}{2}\),& &\text{ and}& &\phi_4(v_2)=2^{-k_2n_2}d_{k_2}^{\tilde b_2}\(x_2,2^{-k_2+2}v_2+\frac{x_2+y_2}{2}\).
\end{align*}
It is not hard to verify that $2^{k_1n_1}\phi_1^{\frac{x_1+y_1}{2},2^{-k_1+2}}(u_1)=s_{k_1}(u_1,y_1)$ for $u_1\in\R^{n_1}$ and likewise for the other three terms.  This completes the proof of \eqref{size} when both $x_1,y_1$ and $x_2,y_2$ are close.  Now fix $x,y\in\R^n$ such that $|x_1-y_1|>2^{-k_1+2}$ and $|x_2-y_2|>2^{-k_2+2}$.  It follows that 
\begin{align*}
\supp(s_{k_1}^{b_1}(\cdot,y_1))\cap\supp(d_{k_1}^{\tilde b_1}(x_1,\cdot))=\supp(s_{k_2}^{b_2}(\cdot,y_2))\cap\supp(d_{k_2}^{\tilde b_2}(x_2,\cdot))=\emptyset.
\end{align*}
Then we can use the kernel representation of $T$ to write
\begin{align*}
|\theta_{\vec s}(x,y)|&=\left|\int_{\R^{2n}}K(u,v)s_{k_1}^{b_1}(v_1,y_1)d_{k_1}^{\tilde b_1}(x_1,u_1)s_{k_2}^{b_2}(v_2,y_2)d_{k_2}^{\tilde b_2}(x_2,u_2)\tilde b(u)b(v)du\,dv\right|\\
&\less\int_{\R^{2n}}|K(u,v)-K(x_1,u_2,v_1,v_2)-K(u_1,x_2,v_1,v_2)+K(x_1,x_2,v_1,v_2)|\\
&\hspace{4.25cm}\times|s_{k_1}^{b_1}(v_1,y_1)d_{k_1}^{\tilde b_1}(x_1,u_1)s_{k_2}^{b_2}(v_2,y_2)d_{k_2}^{\tilde b_2}(x_2,u_2)|du\,dv\\
&\leq\int_{|y_i-v_i|<2^{-k_i}}\int_{|x_i-u_i|<2^{-k_i}}\frac{|x_1-u_1|^{\gamma}|x_2-u_2|^{\gamma}}{|x_1-v_1|^{n_1+\gamma}|x_2-v_2|^{n_2+\gamma}}2^{2k_1n_1}2^{2k_2n_2}du\,dv\\
&\leq\int_{|y_i-v_i|<2^{-k_i}}\int_{|x_i-u_i|<2^{-k_i}}\frac{2^{k_1(2n_1-\gamma)}2^{k_2(2n_2-\gamma)}}{(|x_1-y_1|/2+2^{-k_1})^{n_1+\gamma}(|x_2-y_2|/2+2^{-k_2})^{n_2+\gamma}}du\,dv\\
&\less\frac{2^{-\gamma k_1}2^{-\gamma k_2}}{(|x_1-y_1|+2^{-k_1})^{n_1+\gamma}(|x_2-y_2|+2^{-k_2})^{n_2+\gamma}}=\Phi_{k_1}^{n_1+\gamma}(x_1-y_1)\Phi_{k_2}^{n_2+\gamma}(x_2-y_2).
\end{align*}
Fix $x,y\in\R^n$ such that $|x_1-y_1|\leq2^{-k_1+2}$ and $|x_2-y_2|>2^{-k_2+2}$.  Then we can write
\begin{align*}
|\theta_{\vec s}(x,y)|&=\left|\<M_{\tilde b}TM_b\(s_{k_1}^{b_1}(\cdot,y_1)\otimes s_{k_2}^{b_2}(\cdot,y_2)\),d_{k_1}^{\tilde b_1}(x_1,\cdot)\otimes d_{k_2}^{\tilde b_2}(x_2,\cdot)\>\right|\\
&=2^{2k_1n_1}2^{2k_2n_2}\left|\<M_{\tilde b}TM_b\(\widetilde\phi_1^{y_1,2^{-k_1}}\otimes \phi_2^{\frac{x_2+y_2}{2},2^{-k_2+2}}\),\widetilde \phi_3^{x_1,2^{-k_1}}\otimes \phi_4^{\frac{x_2+y_2}{2},2^{-k_2+2}}\>\right|,
\end{align*}
where
\begin{align*}
&\widetilde\phi_1(u_1)=2^{-k_1n_1}s_{k_1}^{b_1}(2^{-k}u_1+y_1,y_1)& &\text{and}& &\widetilde\phi_3(v_1)=2^{-k_1n_1}d_{k_1}^{\tilde b_1}(x_1,2^{-k}v_1+x_1)&
\end{align*}
again are normalized bumps of order $m$ (up to a constant multiple independent of $x$, $y$, and $\vec k$).  Since $|x_2-y_2|>4\cdot2^{-k_2}$, we can apply \eqref{WBP5} to obtain the following estimate.
\begin{align*}
|\theta_{\vec k}(x,y)|&\less2^{2k_1n_1}2^{2k_2n_2}\(\frac{2^{-k_1n_1}2^{-k_2n_2}}{(2^{k_2}|x_2-y_2|)^{n_2+\gamma}}\)\\
&\less\frac{2^{k_1n_1}2^{k_2n_2}}{(1+2^{k_2}|x_2-y_2|)^{n_2+\gamma}}\less\Phi_{k_1}^{n_1+\gamma}(x_1-y_1)\Phi_{k_2}^{n_2+\gamma}(x_2-y_2).
\end{align*}
A similar argument using \eqref{WBP3} proves that \eqref{size} holds when $|x_1-y_1|>2^{-k_1+2}$ and $|x_2-y_2|\leq2^{-k_2+2}$.  This verifies that $\theta_{\vec k}$ satisfies condition \eqref{size} for all $x,y\in\R^n$.  Now to verify \eqref{regx1}, recall that for $W\in (C_0^\infty(\R^n))'$, $f\in C_0^\infty(\R^n)$, and $x\in\R^n$, $F(x)=\<W,f^x\>$ is a differentiable function where $\partial_{x_i}F(x)=\<W,(\partial_{x_i}f)^x\>$.  Then $\theta_{\vec k}$ is differentiable, and we can estimate
\begin{align*}
|\nabla_{x_1}\theta_{\vec k}(x,y)|^2&=\sum_{j=1}^{n_1}\left|\<M_{\tilde b}TM_b(s_{k_1}^{b_1}(\cdot,y_1)\otimes s_{k_2}^{b_2}(\cdot,y_2)),\partial_{x_{1,j}}(d_{k_1}^{\tilde b_1}(x_1,\cdot))\otimes d_{k_2}^{\tilde b_2}(x_2,\cdot)\>\right|^2\\
&\less 2^{2k_1(n_1+1)}2^{2k_2n_2},
\end{align*}
since $2^{-k_1(n_1+1)}\partial_{x_{1,j}}(d_{k_1}^{\tilde b_1}(x_1,\cdot))$ is again a normalized bump for $x_1=(x_{1,1},...,x_{1,n_1})\in\R^{n_1}$ (up to a constant multiple independent of $x$, $y$, and $\vec k$).  Therefore
\begin{align*}
|\theta_{\vec k}(x,y)-\theta_{\vec k}(x_1',x_2,y)|&\leq||\nabla_{x_1}\theta_{\vec k}(x,y)||_{L^\infty}\;|x_1-x_1'|\less 2^{k_1n_1}2^{k_2n_2}(2^{k_1}|x_1-x_1'|).
\end{align*}
This proves that $\theta_{\vec k}$ verifies \eqref{regx1} via the equivalence in Remark \ref{r:kernelequiv}.  By the same argument, it follows that $\theta_{\vec k}$ verifies \eqref{regx2}-\eqref{regy2}.  Now by the continuity of $T$ from $bC_0^\delta(\R^n)$ into $(\tilde bC_0^\delta(\R^n))'$, we have that
\begin{align*}
\int_{\R^{n_1}}\theta_{\vec k}(x,y)\tilde b_1(x_1)dx_1&=\lim_{R\rightarrow\infty}\<M_{\tilde b}TM_b(s_{k_1}^{b_1}(\cdot,y_1)\otimes s_{k_2}^{b_2}(\cdot,y_2)),\lambda_{R,k_1}\otimes d_{k_2}^{\tilde b_2}(x_2,\cdot)\>
\end{align*}
where
\begin{align*}
\lambda_{R,k_1}(u_1)=\int_{|x_1|\leq R}d_{k_1}^{\tilde b_1}(x_1,u_1)\tilde b_1(x_1)dx_1.
\end{align*}
Note that for $|u_1|>R+2^{-k_1}$, we have $|u_1-x_1|\geq|u_1|-|x_1|>2^{-k_1}$ and hence $\lambda_{R,s_1}(u_1)=0$ for such $u_1$.  Also for $|u_1|<R-2^{-k_1}$ and $x\in\supp(d_{k_1}^{\tilde b_1}(\cdot,u_1))$, it follows that $|x_1|\leq|u_1|+|u_1-x_1|<R$.  Since $D_{k_1}^{\tilde b_1}\tilde b_1=0$, $\lambda_{R,s_1}(u_1)=0$ for $|u_1|<R-2^{-k_1}$.  That is $\supp(\lambda_{R,s_1})\subset B(0,R+2^{-k_1})\backslash B(0,R-2^{-k_1})$.  Now take $R>|y_1|+2^{-k_1+1}$ so that $\lambda_{R,k_1}$ and $s_{k_1}^{b_1}(\cdot,y_1)$ have disjoint support.  Now we split into two cases:  (1) where $|x_2-y_2|\leq2^{-k_1+2}$ and (2) where $|x_2-y_2|>2^{-k_2+2}$.\\

\noindent\underline{Case 1: ($|x_2-y_2|\leq2^{-k_1+2}$) }  Here we take $R>2^{-k_1+6}+2|y_1|$.  Consider 
\begin{align*}
\mathcal B=\{B(u_1,2^{-k_1}):u_1\in\supp(\lambda_{R,k_1})\},
\end{align*}
which is an open cover of $\supp(\lambda_{R,k_1})$.  Then by Vitali's covering lemma, there exists finite collection $\{B_1,...,B_J\}\subset\mathcal B$ of disjoint balls such that $\{3B_1,...,3B_J\}$ forms an open cover of $\supp(\lambda_{R,k_1})$.  Let $c_j\in \R^{n_1}$ be the center of $B_j$ for each $j=1,..,J$.  Fix $\chi\in C_0^\infty(\R^{n_1})$ such that $\chi=1$ on $B(0,1)$ and $\supp(\chi)\subset B(0,2)$.  Let $\widetilde\chi_j(u_1)=\chi\(\frac{u_1-c_j}{3\cdot2^{-k_1}}\)$, and it follows that $\widetilde\chi_j=1$ on $3B_j$ and $\widetilde\chi_j$ is supported inside $6B_j$.  Finally define the partition of unity for $3B_1\cup\cdots\cup3B_J$,
\begin{align*}
\chi_j(u_1)=\frac{\widetilde\chi_j(u_1)}{\sum_{k=1}^J\widetilde\chi_k(u_1)}\hspace{.5cm}\text{for }j=1,...,J.
\end{align*}
Let $m\in\N_0$ be the integer specified by the weak boundedness and mixed weak boundedness properties for $M_bTM_b$.  It follows that
\begin{align*}
\eta_j(u_1)=\frac{1}{\max_{|\alpha|\leq m}||\partial^\alpha(\lambda_{R,k_1}\chi_j)||_{L^\infty} }\chi_j(2^{-k_1+3}u_1+c_j)\lambda_{R,k_1}(2^{-k_1+3}u_1+c_j)
\end{align*}
is a normalized bump of order $m$ for each $j=1,...,J$.  Note that for each $\beta\in\N_0^{n_1}$ with $|\beta|\leq |\alpha|\leq m$
\begin{align*}
|\partial^\beta \lambda_{R,k_1}(u_1)|&\leq\int_{|x_1|\leq R}|\partial_{u_1}^\beta d_{k_1}^{\tilde b_1}(x_1,u_1)\tilde b_1(x_1)|dx_1\\
&\leq2^{k_1|\beta|}\int_{\R^{n_1}}|\partial_{u_1}^\beta d_{k_1}^{\tilde b_1}(x_1,u_1)\tilde b_1(x_1)|dx_1\less2^{k_1|\beta|}.
\end{align*}
The importance here is that this estimate does not depend on $R$; it does depend on $k_1$ and $\beta$, but since we are taking a limit in $R$ for a fixed $k_1$ and $|\beta|\leq m$, this is not of consequence.  Likewise for $|\beta|\leq|\alpha|\leq m$ and $u\in\supp(\lambda_{R,k_1})\cap 3B_j$
\begin{align*}
|\partial^\beta\chi_j(u)|=\left|\partial^\beta\[\frac{\widetilde\chi\(3\frac{u_1-c_j}{2^{-k_1}}\)}{\sum_{k=1}^J\widetilde\chi_k\(3\frac{u_1-c_j}{2^{-k_1}}\)}\]\right|&=3^{|\beta|}2^{|\beta|k_1}\left|\left|\partial^\beta\[\frac{\widetilde\chi}{\sum_{k=1}^J\widetilde\chi_k}\]\right|\right|_{L^\infty(B(0,1))}\leq A_{\beta}2^{|\beta|k_1},
\end{align*}
for some constant $A_{\beta}>0$ depending only on $\beta\in\N_0^{n_1}$.  Note that we use $\widetilde\chi_j\in C_0^\infty(\R^{n_1})$ and $\sum_{k=1}^J\widetilde\chi_k\geq1$ on $\supp(\lambda_{R,k_1})\cap 3B_j$.   Again the importance here is that this estimate does not depend on $R$; it does depend on $k_1$, $\beta$, and derivatives of $\chi$, but that is not a problem.  Also define $\phi(u_1)=2^{-k_1n_1}s_{k_1}^{b_1}(2^{-k_1+3}u_1+y_1,y_1)$, and it follows that $\phi$ is a normalized bump up to a constant multiple.  We now use that
\begin{align*}
&\sum_{j=1}^J\max_{|\alpha|\leq m}||\partial^\alpha(\lambda_{R,k_1}\chi_j)||_{L^\infty}\eta_j^{c_j,2^{-k_1+3}}(u_1)=\sum_{j=1}^J\chi_j(u_1)\lambda_{R,k_1}(u_1)=\lambda_{R,k_1}(u_1),\\
&\phi^{y_1,2^{-k_1+3}}(u_1)=2^{-k_1n_1}s_{k_1}^{b_1}\(2^{-k_1+3}\frac{u_1-c_j}{2^{-k_1+3}}+y_1,y_1\)=2^{-k_1n_1}s_{k_1}^{b_1}(u_1,y_1),
\end{align*}
and since $R>2^{-k_1+6}+2|y_1|$, it follows that 
\begin{align*}
|c_j-y_1|\geq|c_j|-|y_1|\geq R-2^{-k_1}-|y_1|>2^{-k_1+6}-2^{-k_1}\geq4\cdot2^{-k_1+3}.
\end{align*}
Then we can apply \eqref{WBP2} in the following way 
\begin{align*}
&\left|\<M_{\tilde b}TM_b(s_{k_1}^{b_1}(\cdot,y_1)\otimes s_{k_2}^{b_2}(\cdot,y_2)),\lambda_{R,k_1}\otimes d_{k_2}^{\tilde b_2}(x_2,\cdot)\>\right|\\
&\hspace{1.5cm}\leq\sum_{j=1}^J\max_{|\alpha|\leq m}||\partial^\alpha(\lambda_{R,k_1}\chi_j)||_{L^\infty}\left|\<T(\phi^{y_1,2^{-k_1+3}}\otimes s_{k_2}(\cdot,y_2)),\eta_j^{c_j,2^{-k_1+3}}\otimes d_{k_2}^{\tilde b_2}(x_2,\cdot)\>\right|\\
&\hspace{1.5cm}\leq\sum_{j=1}^JA_{k_1,m}\frac{2^{k_2n_2}2^{-k_1n_1}}{(2^{k_1}|y_1-c_j|)^{n_1}}\less\sum_{j=1}^JA_{k_1,m}\frac{2^{k_2n_2}2^{-2k_1n_1}}{R^{n_1}}=A_{k_1,m}\frac{2^{k_2n_2}2^{-2k_1n_1}}{R^{n_1}}J,\hspace{.5cm}\\
&\hspace{1.75cm}\text{ where}\hspace{.5cm}A_{k_1,m}=\max_{|\beta|+|\gamma|\leq m}2^{k_1(|\beta|+|\gamma|)}A_\gamma.
\end{align*}
Now we use that $B_1,...,B_J$ is a disjoint collection of open sets to estimate $J$:
\begin{align*}
J\less 2^{-k_1n_1}\sum_{j=1}^J|B_j|=2^{-k_1n_1}\left|\bigcup_{j=1}^JB_j\right|&\leq2^{-k_1n_1}|B(0,R+2^{-k_1+3})\backslash B(0,R-2^{-k_1+3})|\\
&\less 2^{-k_1(n_1+1)}R^{n_1-1}.
\end{align*}
Note that each $B_j\subset B(0,R+2^{-k_1+3})\backslash B(0,R-2^{-k_1+3})$ since $c_j\in\supp(\lambda_{R,k_1})\subset B(0,R+2^{-k_1+3})\backslash B(0,R-2^{-k_1+3})$ and each $B_j$ has radius $2^{-k_1}$.  Therefore
\begin{align*}
&\left|\<M_{\tilde b}TM_b(s_{k_1}(\cdot,y_1)\otimes s_{k_2}(\cdot,y_2)),\lambda_{R,k_1}\otimes d_{k_2}^{\tilde b_2}(x_2,\cdot)\>\right|\\
&\hspace{2.5cm}\less A_{k_1,m}\frac{2^{-k_1(2n_1+\gamma)}2^{k_2n_2}}{R^{n_1}}2^{-k_1(n_1+1)}R^{n_1-1}= A_{k_1,m}\frac{2^{-k_1(n_1-1)}2^{k_2n_2}}{R},
\end{align*}
which tends to zero as $R\rightarrow\infty$.  This completes the proof for the first case.\\

\noindent\underline{Case 2: ($|x_2-t_2|>2^{-k_2+2}$)}  Since $\lambda_{R,k_1}$ and $s_{k_1}(\cdot,y_1)$ have disjoint support, we can use the full kernel representation for $T$ to compute
\begin{align*}
&\left|\<M_{\tilde b}TM_b(s_{k_1}^{b_1}(\cdot,y_1)\otimes s_{k_2}^{b_2}(\cdot,y_2)),\lambda_{R,k_1}\otimes d_{k_2}^{\tilde b_2}(x_1,\cdot)\>\right|\\
&\hspace{.5cm}=\left|\iint_{\R^{2n}}K(u,v)s_{k_1}^{b_1}(v_1,y_1)s_{k_2}^{b_2}(v_2,y_2)\lambda_{R,k_1}(u_1)d_{k_2}^{\tilde b_2}(x_2,u_2)\tilde b(u)b(v)du\,dv\right|\\
&\hspace{.5cm}\less\iint_{\R^{2n}}\frac{1}{|u_1-v_1|^{n_1}|u_2-v_2|^{n_2}}|s_{k_1}^{b_1}(v_1,y_1)s_{k_2}^{b_2}(v_2,y_2)\lambda_{R,k_1}(u_1)d_{k_2}^{\tilde b_2}(x_2,u_2)|du\,dv\\
&\hspace{.5cm}\less\iint_{\R^{2n}}\frac{2^{k_2n_2}}{(|u_1|-|t_1|-|t_1-v_1|)^{n_1}}|s_{k_1}^{b_1}(v_1,y_1)s_{k_2}^{b_2}(v_2,y_2)\lambda_{R,k_1}(u_1)d_{k_2}^{\tilde b_2}(x_2,u_2)|du\,dv\\
&\hspace{.5cm}\less 2^{k_2n_2}R^{-n_1}\int_{\R^{n_1}}|\lambda_{R,s_1}(u_1)|du_1\less 2^{k_2n_2}2^{-k_1}R^{-1},
\end{align*}
which again tends to zero as $R\rightarrow\infty$.  Therefore $\theta_{\vec k}$ has integral zero in $x_1$, and a similar argument proves that it has integral zero in $x_2$ as well.
\end{proof}

By symmetry, it follows that each of the following define collections of biparameter Littlewood-Paley-Stein operators:
\begin{align*}
&\theta_{\vec k}^2(x,y)=\<M_{\tilde b}TM_b(s_{k_1}^{b_1}(\cdot,y_1)\otimes d_{k_2}^{b_2}(\cdot,y_2)),d_{k_1}^{\tilde b_1}(x_1,\cdot)\otimes s_{k_2}^{\tilde b_2}(x_2,\cdot)\>,\\
&\theta_{\vec k}^3(x,y)=\<M_{\tilde b}TM_b(d_{k_1}^{b_1}(\cdot,y_1)\otimes s_{k_2}^{b_2}(\cdot,y_2)),s_{k_1}^{\tilde b_1}(x_1,\cdot)\otimes d_{k_2}^{\tilde b_2}(x_2,\cdot)\>,\;\;\text{ and}\\
&\theta_{\vec k}^4(x,y)=\<M_{\tilde b}TM_b(d_{k_1}^{b_1}(\cdot,y_1)\otimes d_{k_2}^{b_2}(\cdot,y_2)),s_{k_1}^{\tilde b_1}(x_1,\cdot)\otimes s_{k_2}^{\tilde b_2}(x_2,\cdot)\>.
\end{align*}
Furthermore, these kernels satisfy
\begin{align*}
&\int_{\R^{n_1}}\theta_{\vec k}^2(x,y)\tilde b_1(x_1)dx_1=\int_{\R^{n_2}}\theta_{\vec k}^2(x,y)b_2(y_2)dy_2=0,\\
&\int_{\R^{n_1}}\theta_{\vec k}^2(x,y) b_1(y_1)dy_1=\int_{\R^{n_2}}\theta_{\vec k}^2(x,y)\tilde b_2(x_2)dx_2=0,\;\;\text{ and}\\
&\int_{\R^{n_1}}\theta_{\vec k}^2(x,y) b_1(y_1)dy_1=\int_{\R^{n_2}}\theta_{\vec k}^2(x,y)b_2(y_2)dy_2=0.
\end{align*}

\begin{definition}
A biparameter singular integral operator satisfies the biparameter $Tb=T^*\tilde b=0$ condition if the following two conditions hold:  (1)  Let $\psi_1\in C_0^\infty(\R^{n_1})$, $\psi_2,\varphi_2\in C_0^\infty(\R^{n_2})$, and $\eta_R\in C_0^\infty(\R^{n_1})$ such that $\eta_R=1$ on $B_1(0,R)\subset\R^{n_1}$ and $\supp(\eta_R)\subset B_1(0,2R)\subset\R^{n_1}$.  If $b_1\psi_1$ has mean zero and either $b_2\varphi_2$ or $b_2\psi_2$ has mean zero, then
\begin{align}
&\<T(b_1\otimes b_2\psi_2),\tilde b_1\psi_1\otimes \tilde b_2\varphi_2\>:=\lim_{R\rightarrow\infty}\<M_{\tilde b}TM_b(\eta_R\otimes \psi_2),\psi_1\otimes \varphi_2\>=0,\label{Tb1}\\
&\<T(b_1\psi_1\otimes b_2\psi_2),\tilde b_1\otimes \tilde b_2\varphi_2\>:=\lim_{R\rightarrow\infty}\<M_{\tilde b}TM_b(\psi_1\otimes \psi_2),\eta_R\otimes \varphi_2\>=0,\label{Tb2}
\end{align}
and (2) let $\psi_2\in C_0^\infty(\R^{n_2})$, $\psi_1,\varphi_1\in C_0^\infty(\R^{n_1})$, and $\eta_R\in C_0^\infty(\R^{n_2})$ such that $\eta_R=1$ on $B_2(0,R)\subset\R^{n_1}$ and $\supp(\eta_R)\subset B_2(0,2R)\subset\R^{n_2}$.  If $b_2\psi_2$ has mean zero and either $b_1\varphi_1$ or $b_1\psi_1$ has mean zero, then
\begin{align*}
&\<T(b_1\psi_1\otimes b_2), \tilde b_1\varphi_1\otimes \tilde b_2\psi_2\>:=\lim_{R\rightarrow\infty}\<M_{\tilde b}TM_b(\psi_1\otimes\eta_R),\varphi_1\otimes \psi_2\>=0,\\
&\<T( b_1\psi_1\otimes b_2\psi_2),\tilde  b_1\varphi_1\otimes\tilde  b_2\>:=\lim_{R\rightarrow\infty}\<M_{\tilde b}TM_b( \psi_1\otimes \psi_2), \varphi_1\otimes \eta_R\>=0.
\end{align*}
\end{definition}

Next we prove Theorem \ref{t:Tbtheorem}.

\begin{proof}
Let $S_{\vec k}^{b}=S_{k_1}^{b_1}\otimes S_{k_2}^{b_2}$ and $S_{\vec k}^{\tilde b}=S_{k_1}^{\tilde b_1} S_{k_2}^{\tilde b_2}$, where $S_{k_1}^{b_1}$, $S_{k_2}^{b_2}$, $S_{k_1}^{\tilde b_1}$, and $S_{k_2}^{\tilde b_2}$ be the approximations to identity with respect to $b_1$ and $b_2$ respectively constructed in \eqref{accapptoID}.  Also define $D_{k_1}^{b_1}=S_{k_1+1}^{b_1}-S_{k_1}^{b_1}$, $D_{k_2}^{b_2}=S_{k_2+1}^{b_2}-S_{k_2}^{b_2}$, $D_{k_1}^{\tilde b_1}=S_{k_1+1}^{\tilde b_1}-S_{k_1}^{\tilde b_1}$, $D_{k_2}^{\tilde b_2}=S_{k_2+1}^{\tilde b_2}-S_{k_2}^{\tilde b_2}$, $D_{\vec k}^b=D_{k_1}^{b_1} D_{k_2}^{b_2}$, and $D_{\vec k}^{\tilde b}=D_{k_1}^{\tilde b_1} D_{k_2}^{\tilde b_2}$.  It follows that $M_{b_j}S_{k_j}^{b_j}M_{b_j}f_j\rightarrow b_jf_j$ and $M_{ b_j}S_{-k_j}^{b_j}M_{b_j}f_j\rightarrow 0$ in $b_jC_0^\delta(\R^{n_j})$ as $k_j\rightarrow\infty$ for $j=1,2$, whenever $f_j\in C_0^{0,1}(\R^{n_j})$ and 
\begin{align*}
\int_{\R^{n_j}}f_j(x_j)b_j(x_j)dx_j=0.
\end{align*}
This was proved originally in \cite{DJS}, and the proof is also available in \cite{Hart2}.  It follows that $M_{b_j}S_{k_j}^{b_j}M_{b}f\rightarrow bf$ and $M_{ b_j}S_{-k_j}^{b_j}M_{b}f_j\rightarrow 0$ in $bC_0^\delta(\R^n)$ as $k_j\rightarrow\infty$ for $j=1,2$, whenever $f\in C_0^{0,1}(\R^n)$ and 
\begin{align*}
\int_{\R^{n_1}}f(x)b(x)dx_1=\int_{\R^{n_2}}f(x)b(x)dx_2=0.
\end{align*}
Let $f,g\in C_0^{0,1}(\R^n)$ such that
\begin{align*}
\int_{\R^{n_1}}f(x)b(x)dx_1=\int_{\R^{n_2}}f(x)b(x)dx_2=\int_{\R^{n_1}}g(x)\tilde b(x)dx_1=\int_{\R^{n_2}}g(x)\tilde b(x)dx_2=0.
\end{align*}
Then by the continuity of $T$ from $bC_0^\delta(\R^n)$ into $(\tilde bC_0^\delta(\R^n))'$,
\begin{align*}
&\<M_{\tilde b}TM_bf,g\>=\lim_{N_2\rightarrow\infty}\<M_{\tilde b_2}TM_{b_2}S_{N_2}^{b_2}M_{b}f,S_{N_2}^{\tilde b}M_bg\>-\<M_{\tilde b_2}TM_{b_2}S_{-N_2}^{b_2}M_{b}f,S_{-N_2}^{\tilde b_2}M_{\tilde b}g\>\\
&\hspace{.2cm}=\sum_{k_2\in\Z}\<M_{\tilde b_2}TM_{b_2}S_{k_2+1}^{b_2}M_{b}f,D_{k_2}^{\tilde b_2}M_{\tilde b}g\>-\<M_{\tilde b_2}TM_{b_2}D_{k_2}^{b_2}M_{b}f,S_{k_2}^{\tilde b_2}M_{\tilde b}g\>\\
&\hspace{.2cm}=\sum_{k_2\in\Z}\lim_{N_1\rightarrow\infty}\<M_{\tilde b}TM_bS_{k_2+1}^{b_2}S_{N_1}^{b_1}M_bf,D_{k_2}^{\tilde b_2}S_{N_1}^{\tilde b_1}M_{\tilde b}g\>+\<M_{\tilde b}TM_bD_{k_2}^{b_2}S_{N_1}^{b_1}M_bf,S_{k_2}^{\tilde b_2}S_{N_1}^{\tilde b_1}M_{\tilde b}g\>\\
&\hspace{1cm}-\<M_{\tilde b}TM_bS_{k_2+1}^{b_2}S_{-N_1}^{b_1}M_bf,D_{k_2}^{\tilde b_2}S_{-N_1}^{\tilde b_1}M_{\tilde b}g\>-\<M_{\tilde b}TM_bD_{k_2}^{b_2}S_{-N_1}^{b_1}M_bf,S_{k_2}^{\tilde b_2}S_{-N_1}^{\tilde b_1}M_{\tilde b}g\>
\end{align*}
\begin{align*}
&\hspace{.2cm}=\sum_{k_1,k_2\in\Z}\<M_{\tilde b}TM_bS_{k_2+1}^{b_2}S_{k_1+1}^{b_1}M_bf,D_{k_2}^{\tilde b_2}D_{k_1}^{\tilde b_1}M_{\tilde b}g\>+\<M_{\tilde b}TM_bD_{k_2}^{b_2}S_{k_1+1}^{b_1}M_bf,S_{k_2}^{\tilde b_2}D_{k_1}^{\tilde b_1}M_{\tilde b}g\>\\
&\hspace{1cm}+\<M_{\tilde b}TM_bS_{k_2+1}^{b_2}D_{k_1}^{\tilde b_1}M_bf,D_{k_2}^{\tilde b_2}S_{k_1}^{\tilde b_1}M_{\tilde b}g\>+\<M_{\tilde b}TM_bD_{k_2}^{b_2}D_{k_1}^{b_1}M_bf,S_{k_2}^{\tilde b_2}S_{k_1}^{\tilde b_1}M_{\tilde b}g\>\\
&\hspace{.2cm}=\sum_{k_1,k_2\in\Z}\sum_{j=1}^4\<\Theta_{\vec k}^jM_bf,M_{\tilde b}g\>
\end{align*}
where $\Theta_j$ for $j=1,2,3,4$ are defined as follows with their respective kernels
\begin{align*}
&\Theta_{\vec k}^{1}=D_{\vec k}^{\tilde b}M_{\tilde b}TM_bS_{\vec k+1};& &\hspace{0cm}\theta_{\vec k}^1(x,y)=\<M_{\tilde b}TM_b(s_{k_1+1}^{b_1}(\cdot,y_1)\otimes s_{k_2+1}^{b_2}(\cdot,y_2)),d_{k_1}^{\tilde b_1}(x_1,\cdot)\otimes d_{k_2}^{\tilde b_2}(x_2,\cdot)\>,&\\
&\Theta_{\vec k}^{2}=D_{k_1}^{\tilde b_1}S_{k_2}^{\tilde b_2}M_{\tilde b}TM_bS_{k_1+1}^{b_1}D_{k_2}^{b_2};& &\hspace{0cm}\theta_{\vec k}^2(x,y)=\<M_{\tilde b}TM_b(s_{k_1+1}^{b_1}(\cdot,y_1)\otimes d_{k_2}^{b_2}(\cdot,y_2)),d_{k_1}^{\tilde b_1}(x_1,\cdot)\otimes s_{k_2}^{\tilde b_2}(x_2,\cdot)\>,&\\
&\Theta_{\vec k}^{3}=S_{k_1}^{\tilde b_1}D_{k_2}^{\tilde b_2}M_{\tilde b}TM_bD_{k_1}^{b_1}S_{k_2+1}^{b_2};& &\hspace{0cm}\theta_{\vec s}^3(x,y)=\<M_{\tilde b}TM_b(d_{k_1}^{b_1}(\cdot,y_1)\otimes s_{k_2+1}^{b_2}(\cdot,y_2)),s_{k_1}^{\tilde b_1}(x_1,\cdot)\otimes d_{k_2}^{\tilde b_2}(x_2,\cdot)\>,&\\
&\Theta_{\vec k}^{4}=S_{\vec k}^{\tilde b}M_{\tilde b}TM_bD_{\vec k}^b;& &\hspace{0cm}\theta_{\vec s}^4(x,y)=\<M_{\tilde b}TM_b(d_{k_1}^{b_1}(\cdot,y_1)\otimes d_{k_2}^{b_2}(\cdot,y_2)),s_{k_1}^{\tilde b_1}(x_1,\cdot)\otimes s_{k_2}^{\tilde b_2}(x_2,\cdot)\>.&
\end{align*}
By Lemma \ref{l:kernelconditions}, $\theta_{\vec s}^1$ satisfies \eqref{size}-\eqref{regy2} and 
\begin{align*}
\int_{\R^{n_1}}\theta_{\vec k}^1(x,y)b_1(x_1)dx_1=\int_{\R^{n_2}}\theta_{\vec k}^1(x,y)b_2(x_2)dx_2=0.
\end{align*}
By the biparameter $Tb=T^*b=0$ assumption on $T$, 
we also have
\begin{align*}
\int_{\R^{n_1}}\theta_{\vec k}^1(x,y)b_1(y_1)dy_1=\int_{\R^{n_2}}\theta_{\vec k}^1(x,y)b_2(y_2)dy_2=0.
\end{align*}
Then by Theorem \eqref{p:dualpairbound},
\begin{align*}
\sum_{\vec k\in\Z^2}\left|\int_{\R^n}\Theta_{\vec k}^1f(x)g(x)dx\right|\less||f||_{L^p(\R^n)}||g||_{L^{p'}(\R^n)}.
\end{align*}
The same holds for $\Theta_{\vec s}^j$ when $j=2,3,4$, and so it follows that
\begin{align*}
|\<Tf,g\>|&\leq\sum_{j=1}^4\sum_{\vec k\in\Z^2}\left|\int_{\R^n}\Theta_{\vec k}^jf(x)g(x)dx\right|\less||f||_{L^p(\R^n)}||g||_{L^{p'}(\R^n)}.
\end{align*}
Therefore by density, $T$ can be extended to a bounded operator on $L^p$ for $1<p<\infty$.
\end{proof}

\section{Proof of Bounds for $\mathcal C_\Gamma$,$\widetilde{\mathcal C}_\Gamma$ and $\mathcal{C}^*_\Gamma$ }

In this section, we use Theorem \ref{t:Tbtheorem} 
to prove bounds for $\mathcal C_\Gamma$, 
its parameterized version $\widetilde{\mathcal 
C}_\Gamma$, which we define now, and the maximal operator $\mathcal{C}^*_\Gamma$. 

For appropriate 
$f:\R^n\rightarrow\C$, define
\begin{align*}
\widetilde{\mathcal C}_\Gamma 
M_bf(x)=\lim_{t_1,t_2\rightarrow0^+}\int_{\R^2}
\frac{\gamma_1(x_1)-\gamma_1(y_1)}{
(\gamma_1(x_1)-\gamma_1(y_1))^2+t_1^2}\frac{
\gamma_2(x_2)-\gamma_2(y_2)}{
(\gamma_2(x_2)-\gamma_2(y_2))^2+t_2^2}f(y)b(y)dy,
\end{align*}
where $b(y)=\gamma_1'(y_1)\gamma_2'(y_2)$.  We call this the parameterized version of $\mathcal C_\Gamma$ since 
\begin{align*}
\widetilde{\mathcal C}_\Gamma M_bf(x)=\mathcal C_\Gamma( f\circ\gamma^{-1})(\gamma(x)),
\end{align*}
and furthermore, the $L^p(\Gamma)$ bound for $\mathcal C_\Gamma$ can be reduced to $L^p(\R^2)$ bounds for $\widetilde{\mathcal C}_\Gamma$ via \eqref{equivnorm}.  It is not hard to see that the kernel of $\widetilde{\mathcal C}_\Gamma$ is 
\begin{align*}
\frac{1}{(\gamma_1(x_1)-\gamma_1(y_1))(\gamma_1(x_2)-\gamma_1(y_2))},
\end{align*}
which is a biparameter Calder\'on-Zygmund kernel.  In the next proposition, we prove that $\widetilde{\mathcal C}_\Gamma f$ is well-defined for appropriate $f:\R^n\rightarrow\C$ and hence $\mathcal C_\Gamma g$ is also well defined for appropriate $g:\Gamma\rightarrow\C$.  Define the complex $\log$ function with the negative real branch cut, that is for $z\in\C$ we define
\begin{align*}
\log(z)=\ln(|z|)+i\Arg(z),
\end{align*}
where $\ln:(0,\infty)\rightarrow\R$ logarithm base $e$ function with positive real domain and $\Arg(z)$ is the principle argument of $z$ taking values in $(-\pi,\pi]$.  Note that for $u\in(0,\infty)$, $\ln(u)=\log(u)$; we use this notation to emphasize when the input is real versus complex.

\begin{proposition}\label{p:integralrep}
Assume $\Gamma$ satisfies the hypotheses of Theorem \ref{t:ext}.  For all $f\in C_0^\infty(\R^2)$ and $x\in\R^2$, 
\begin{align*}
\widetilde C_\Gamma(bf)(x)&=\frac{1}{4\pi^2}\int_{\R^2}\log\((\gamma_1(x_1)-\gamma_1(y_1))^2\)\log\((\gamma_2(x_2)-\gamma_2(y_2))^2\)\partial_{y_1}\partial_{y_2}f(y)dy.
\end{align*}
Also, for all $f,g\in C_0^\infty(\R^2)$, the pairing $\<\widetilde C_\Gamma(bf),bg\>$ can be realized as any of the following absolutely convergent integrals:
\begin{align*}
&\frac{1}{4\pi^2}\int_{\R^4}\log\((\gamma_1(x_1)-\gamma_1(y_1))^2\)\log\((\gamma_2(x_2)-\gamma_2(y_2))^2\)\partial_{y_1}\partial_{y_2}f(y)g(x)b(x)dy\,dx,\\
&\frac{1}{4\pi^2}\int_{\R^4}\log\((\gamma_1(x_1)-\gamma_1(y_1))^2\)\log\((\gamma_2(x_2)-\gamma_2(y_2))^2\)f(y)\partial_{x_1}\partial_{x_2}g(x)b(y)dy\,dx,\\
&-\frac{1}{4\pi^2}\int_{\R^4}\log\((\gamma_1(x_1)-\gamma_1(y_1))^2\)\log\((\gamma_2(x_2)-\gamma_2(y_2))^2\)\partial_{y_1}f(y)\partial_{x_2}g(x)b(x_1,y_2)dy\,dx,\\
&-\frac{1}{4\pi^2}\int_{\R^4}\log\((\gamma_1(x_1)-\gamma_1(y_1))^2\)\log\((\gamma_2(x_2)-\gamma_2(y_2))^2\)\partial_{y_2}f(y)\partial_{x_1}g(x)b(y_1,x_2)dy\,dx.
\end{align*}
\end{proposition}

\begin{proof}
We first note that for $x_j,y_j\in\R$
\begin{align}
q_{ t_j}(\gamma_j(x_j)-\gamma_j(y_j))\gamma_j'(y_j)&=\frac{1}{\pi}\frac{\gamma_j(x_j)-\gamma_j(y_j)}{(\gamma_j(x_j)-\gamma_j(y_j))^2+ t_j^2}\gamma_j'(y_j)\notag\\
&=-\frac{1}{2\pi}\partial_{y_j}\log\((\gamma_j(x_j)-\gamma_j(y_j))^2+ t_j^2\).\label{potentialformula}
\end{align}
The derivative of $\log$ is well defined here since we defined it with the negative real branch cut, and for all $x_j,y_j\in\R$, we have $Re\((\gamma_j(x_j)-\gamma_j(y_j))^2+ t_j^2\)\geq t_j^2>0$.  Now for $f\in C_0^\infty(\R^2)$ and $x\in\R^2$, we compute the following pointwise limit
\begin{align*}
\widetilde C_\Gamma(bf)(x)&=\lim_{ t_1, t_2\rightarrow0^+}\int_{\R^2}q_{ t_1}(\gamma_1(x_1)-\gamma_1(y_1))q_{ t_2}(\gamma_2(x_2)-\gamma_2(y_2))f(y)\gamma_1'(y_1)\gamma_2'(y_2)dy\\
&=\lim_{ t_1, t_2\rightarrow0^+}\int_{\R^2}\[-\frac{1}{2\pi}\partial_{y_1}\log\((\gamma_1(x_1)-\gamma_1(y_1))^2+ t_1^2\)\]\[-\frac{1}{2\pi}\partial_{y_2}\log\((\gamma_2(x_2)-\gamma_2(y_2))^2+ t_2^2\)\]f(y)dy\\
&=\lim_{ t_1, t_2\rightarrow0^+}\frac{1}{4\pi^2}\int_{\R^2}\[\log\((\gamma_1(x_1)-\gamma_1(y_1))^2+ t_1^2\)\]\[\log\((\gamma_2(x_2)-\gamma_2(y_2))^2+ t_2^2\)\]\partial_{y_1}\partial_{y_2}f(y)dy\\
&=\frac{1}{4\pi^2}\int_{\R^2}\log\((\gamma_1(x_1)-\gamma_1(y_1))^2\)\log\((\gamma_2(x_2)-\gamma_2(y_2))^2\)\partial_{y_1}\partial_{y_2}f(y)dy.
\end{align*}
We integrate by parts in $y_1$ and $y_2$ above, and the boundary terms vanish since $f$ is compactly supported.  Also to justify the last inequality, note the following holds for all $x_j\neq y_j$, so that we can apply dominated convergence:  the following pointwise limit exists
\begin{align*}
\lim_{t_1,t_2\rightarrow0^+}\log\(\gamma_j(x_j)-\gamma_j(y_j))^2+ t_j^2\)\partial_{y_1}\partial_{y_2}f(y)=\log\(\gamma_j(x_j)-\gamma_j(y_j))^2\)\partial_{y_1}\partial_{y_2}f(y),
\end{align*}
and the integrand is dominated by an integrable function function independent of $t_1,t_2<1/4$
\begin{align*}
|\log\((\gamma_j(x_j)-\gamma_j(y_j))^2+ t_j^2\)|&\leq|\ln\(|(\gamma_j(x_j)-\gamma_j(y_j))^2+ t_j^2|\)|+\pi\less|\ln\((x_j-y_j)^2\)|+1.
\end{align*}
Since $\ln(|\cdot|)$ is locally integrable and $f\in C_0^\infty(\R^2)$, we may apply dominated convergence in the last line above.  Now take $f,g\in C_0^\infty(\R^2)$, and it immediately follows that
\begin{align*}
\<M_b\widetilde C_\Gamma M_bf,g\>&=\frac{1}{4\pi^2}\int_{\R^4}\log\((\gamma_1(x_1)-\gamma_1(y_1))^2\)\log\((\gamma_2(x_2)-\gamma_2(y_2))^2\)\\
&\hspace{5cm}\times\partial_{y_1}\partial_{y_2}f(y)g(x)\gamma_1'(x_1)\gamma_2'(x_2)dy\,dx.
\end{align*}
We also have that
\begin{align*}
\<M_b\widetilde C_\Gamma M_bf,g\>&=\lim_{ t_1, t_2\rightarrow0^+}\frac{1}{4\pi^2}\int_{\R^4}\log\((\gamma_1(x_1)-\gamma_1(y_1))^2+ t_1^2\)\log\((\gamma_2(x_2)-\gamma_2(y_2))^2+ t_2^2\)\\
&\hspace{3.2cm}\times\partial_{y_1}\partial_{y_2}f(y)g(x)\gamma_1'(x_1)\gamma_2'(x_2)dy\,dx\\
&\hspace{0cm}=\lim_{ t_1, t_2\rightarrow0^+}\int_{\R^4}q_{ t_1}(\gamma_1(x_1)-\gamma_1(y_1))q_{ t_2}(\gamma_2(x_2)-\gamma_2(y_2)) f(y)g(x)\gamma_1'(y_1)\gamma_2'(y_2)\gamma_1'(x_1)\gamma_2'(x_2)dy\,dx\\
&\hspace{0cm}=\lim_{ t_1, t_2\rightarrow0^+}\frac{1}{4\pi^2}\int_{\R^4}\[\partial_{x_1}\log\((\gamma_1(x_1)-\gamma_1(y_1))^2+ t_1^2\)\]\\
&\hspace{3.2cm}\times\[-\partial_{y_2}\log\((\gamma_2(x_2)-\gamma_2(y_2))^2+ t_2^2\)\] f(y)g(x)\gamma_1'(y_1)\gamma_2'(x_2)dy\,dx\\
&\hspace{0cm}=\lim_{ t_1, t_2\rightarrow0^+}-\frac{1}{4\pi^2}\int_{\R^4}\log\((\gamma_1(x_1)-\gamma_1(y_1))^2+ t_1^2\)\\
&\hspace{3.2cm}\times\log\((\gamma_2(x_2)-\gamma_2(y_2))^2+ t_2^2\)\partial_{y_2} f(y)\partial_{x_1}g(x)\gamma_1'(y_1)\gamma_2'(x_2)dy\,dx\\
&\hspace{0cm}=-\frac{1}{4\pi^2}\int_{\R^4}\log\((\gamma_1(x_1)-\gamma_1(y_1))^2\)\log\((\gamma_2(x_2)-\gamma_2(y_2))^2\)\partial_{y_2} f(y)\partial_{x_1}g(x)\gamma_1'(y_1)\gamma_2'(x_2)dy\,dx.
\end{align*}
Here we integrate by parts in $x_1$ and $y_2$ and use dominated convergence in essentially the same way as above.  A similar argument verifies the other formulas for $\<\widetilde C_\Gamma(bf),bg\>$.
\end{proof}

Note that we cannot use properties of logs to replace the integrand above by
\begin{align*}
4\,\log\(\gamma_1(x_1)-\gamma_1(y_1)\)\,\log\(\gamma_2(x_2)-\gamma_2(y_2)\).
\end{align*}
This is because $Re\[(\gamma_j(x_j)-\gamma(y_j))^2\]>0$ for $x_j\neq y_j$, and furthermore recall that we showed that $Re\[(\gamma_j(x_j)-\gamma(y_j))^2\]\geq (1-\lambda_j^2)(x_j-y_j)^2$.  So this term avoids the branch cut of $\log$, but $Re\[\gamma_j(x_j)-\gamma(y_j)\]$ may change sign, which causes problems with the complex $\log$ function.

\begin{lemma}\label{l:BMOest}
Suppose $L_j:\R\rightarrow\R$ is a Lipschitz function with small Lipschitz constant $\lambda_j<1$ for $j=1,2$, and define $\gamma(x)=(\gamma_1(x_1),\gamma_2(x_2))=(x_1+iL_1(x_1),x_2+iL_2(x_2))$.  If $\psi\in C_0^\infty(\R)$ is a normalized bump of any order with mean zero, then
\begin{align*}
\sup_{u_j\in\R,R_j>0}\left|\int_\R\log\((\gamma_j(x_j)-\gamma_j(y_j))^2\)R_j^{-1}\psi\(\frac{u_j-y_j}{R_j}\)dy_j\right|\less1,
\end{align*}
where the suppressed constant does not depend on $\psi$, $x_j$, or $\gamma$.  In other words, $\log((\gamma_j(x_j)-\gamma_j(\cdot))^2)\in BMO(\R)$ with norm independent of $x_j$, and $\gamma$.  In particular this holds when $\psi(u_j)=\varphi'(u_j)$ for some some normalized bump $\varphi\in C_0^\infty(\R)$ of order at least $1$.
\end{lemma}

\begin{proof}
Let $\psi\in C_0^\infty(\R)$ be a normalized bump with integral zero.  For $|u_j-x_j|\leq 2R_j$
\begin{align*}
&\left|\int_\R\log\((\gamma_j(x_j)-\gamma_j(y_j))^2\)R_j^{-1}\psi\(\frac{u_j-y_j}{R_j}\)dy_j\right|\\
&\hspace{1cm}\leq\frac{||\psi||_{L^\infty}}{R_j}\int_{u_j-x_j-R_j}^{u_j-x_j+R_j}\left|\log\((\gamma_j(x_j)-\gamma_j(x_j+y_j))^2\)-\log(R_j^2)\right|dy_j\\
&\hspace{1cm}\leq \int_{-3}^{3}\(\ln\(\frac{|(\gamma_j(x_j)-\gamma_j(x_j+R_jy_j))^2|}{R_j^2}\)+\pi\)dy_j\\
&\hspace{1cm}\less \int_{-3}^{3}(1+|\ln(|y_j|)|)dy_j\less1.
\end{align*}
Here we use that for $|y_j|\leq 3$
\begin{align*}
(1-\lambda_j^2)|y_j|^2\leq\frac{|(\gamma_j(x_j)-\gamma_j(x_j+R_jy_j))^2|}{R_j^2}\leq(1+\lambda_j)^2|y_j|^2\leq4|y_j|^2\leq36.
\end{align*}
Now for $|u_j-x_j|>2R_j$, we estimate as follows
\begin{align*}
&\left|\int_\R\log\((\gamma_j(x_j)-\gamma_j(y_j))^2\)R_j^{-1}\psi\(\frac{u_j-y_j}{R_j}\)dy_j\right|\\
&\hspace{.5cm}\leq\frac{||\psi||_{L^\infty}}{R_j}\int_{u_j-x_j-R_j}^{u_j-x_j+R_j}\left|\log\((\gamma_j(x_j)-\gamma_j(x_j+y_j))^2\)-\log\((\gamma_j(x_j)-\gamma_j(u_j))^2\)\right|dy_j\\
&\hspace{.5cm}\less1+\frac{1}{R_j}\int_{u_j-x_j-R_j}^{u_j-x_j+R_j}\left|\ln\(\frac{|\gamma_j(x_j)-\gamma_j(x_j+y_j)|^2}{|\gamma_j(x_j)-\gamma_j(u_j)|^2}\)\right|dy_j\\
&\hspace{.5cm}\less1+\frac{1}{R_j}\int_{|y_j-(u_j-x_j)|<R_j}\left|\ln\(\frac{|y_j|}{|u_j-x_j|}\)\right|dy_j\\
&\hspace{.5cm}\leq1+\frac{1}{R_j}\int_{|y_j-(u_j-x_j)|<R_j}\left|\ln\(\frac{|u_j-x_j|+|y_j-(u_j-x_j)|}{|u_j-x_j|}\)\right|dy_j\\
&\hspace{4cm}+\frac{1}{R_j}\int_{|y_j-(u_j-x_j)|<R_j}\left|\ln\(\frac{|u_j-x_j|}{|u_j-x_j|-|y_j-(u_j-x_j)|}\)\right|dy_j\\
&\hspace{.5cm}\leq1+\frac{1}{R_j}\int_{|y_j-(u_j-x_j)|<R_j}(\ln(3/2)+\ln(2))dy_j\less1.
\end{align*}
This completes the proof.
\end{proof}

Now we prove that $\widetilde{\mathcal C}_\Gamma$ satisfies the hypotheses of Theorem \ref{t:Tbtheorem}.

\begin{proposition}\label{p:WBP}
Assume $\Gamma$ satisfies the hypotheses of Theorem \ref{t:ext}.  The operator $M_b\widetilde{\mathcal C}_\Gamma M_b$ satisfies the weak boundedness and mixed weak boundedness properties, where $b(x)=\gamma_1'(x_1)\gamma_2'(x_2)$ for $x=(x_1,x_2)\in\R^2$.
\end{proposition}

\begin{proof}
Let $\varphi_j,\psi_j\in C_0^\infty$ be normalized bumps, $x\in\R^2$, and $R_1,R_2>0$.  Then
\begin{align*}
&\left|\<M_b\widetilde{\mathcal C}_\Gamma M_b(\varphi_1^{x_1,R_1}\otimes\varphi_2^{x_2,R_2}),\psi_1^{x_1,R_1}\otimes\psi_2^{x_2,R_2}\>\right|\\
&\hspace{1cm}=\frac{1}{4\pi^2}\left|\int_{\R^4}\log\((\gamma_1(u_1)-\gamma_1(v_1))^2\)\log\((\gamma_2(u_2)-\gamma_2(v_2))^2\)\right.\\
&\left.\phantom{\int}\hspace{4cm}\times(\varphi_1^{x_1,R_1})'(v_1)(\varphi_2^{x_2,R_2})'(v_2)\psi_1^{x_1,R_1}(u_1)\psi_2^{x_2,R_2}(u_2)du\,dv\right|\\
&\hspace{1cm}\leq\frac{1}{4\pi^2}\int_{x_1-R_1}^{x_1+R_1}\int_{x_2-R_2}^{x_2+R_2}\left|\int_{\R^2}\log\((\gamma_1(u_1)-\gamma_1(v_1))^2\)\log\((\gamma_2(u_2)-\gamma_2(v_2))^2\)\right.\\
&\left.\phantom{\int}\hspace{4cm}\times R_1^{-1}(\varphi_1')^{x_1,R_1}(v_1)R_2^{-1}(\varphi_2')^{x_2,R_2}(v_2)dv\right|du\less R_1\, R_2.
\end{align*}
The last inequality holds due to Lemma \ref{l:BMOest}.  Then $\widetilde{\mathcal C}_\Gamma$ satisfies the weak boundedness property.  Now we verify the mixed weak boundedness properties for $\widetilde{\mathcal C}_\Gamma$: we first verify \eqref{WBP2}.  Let $x_1\in\R$, $R_1>0$, and $\varphi_j,\psi_j\in C_0^\infty(\R)$ be normalized bumps.  Then for $x_1,x_2,y_2\in\R$ and $R_1,R_2>0$ such that $|x_1-y_1|>4R_1$
\begin{align*}
&\left|\<M_b\widetilde{\mathcal C}_\Gamma M_b(\varphi_1^{y_1,R_1}\otimes \varphi_2^{x_2,R_2}),\psi_1^{x_1,R_1}\otimes \psi_2^{x_2,R_2}\>\right|\\
&\hspace{1cm}=\lim_{ t_1, t_2\rightarrow0^+}\left|\int_{\R^2}q_{ t_1}(\gamma_1(u_1)-\gamma_1(v_1))\varphi_1^{y_1,R_1}(v_1)\psi_1^{x_1,R_1}(u_1)\gamma_1'(v_1)\gamma_1'(u_1)dv_1\,du_1\right|\\
&\hspace{3.5cm}\times\left|\int_{\R^2}q_{ t_2}(\gamma_2(u_2)-\gamma_2(v_2))\varphi_2^{y_2,R_2}(v_2)\psi_2^{x_2,R_2}(u_2)\gamma_2'(v_2)\gamma_2'(u_2)dv_2\,du_2\right|\\
&\hspace{1cm}\leq\lim_{ t_1, t_2\rightarrow0^+}\int_{\R^2}|q_{ t_1}(\gamma_1(u_1)-\gamma_1(v_1))|\,|\varphi_1^{y_1,R_1}(v_1)\psi_1^{x_1,R_1}(u_1)\gamma_1'(v_1)\gamma_1'(u_1)|dv_1\,du_1\\
&\hspace{3.5cm}\times\left|\int_{\R^2}\log\((\gamma_2(u_2)-\gamma_2(v_2))^2\)(\varphi_2^{
y_2,R_2})'(v_2)\psi_2^{x_2,R_2}(u_2)\gamma_2'(u_2)dv_2\,du_2\right|\\
&\hspace{1cm}=\lim_{ t_1, t_2\rightarrow0^+}A_{ t_1}\times B_{ t_2}.
\end{align*}
To estimate $A_{t_1}$, we use the kernel estimate for $q_{t_1}$ to conclude the following bound.
\begin{align*}
\int_{\R^2}|q_{ t_1}(\gamma_1(u_1)-\gamma_1(v_1))|\,|\varphi_1^{y_1,R_1}(v_1)\psi_1^{x_1,R_1}(u_1)\gamma_1'(v_1)\gamma_1'(u_1)|dv_1\,du_1&\less\int_{\R^2}\frac{1}{|u_1-v_1|}|\varphi_1^{y_1,R_1}(v_1)\psi_1^{x_1,R_1}(u_1)|dv_1\,du_1\\
&\less\frac{R_1^2}{|x_1-y_1|}=\frac{R_1}{(R_1^{-1}|x_1-y_1|)}.
\end{align*}
For the second term, we argue exactly as in the full weak boundedness case using Lemma \ref{l:BMOest}:
\begin{align*}
B_{ t_2}&\less\int_{\R}\left|\int_{\R}\log\((\gamma_2(u_2)-\gamma_2(v_2))^2\)R_2^{-1}(\varphi_2')^{y_2,R_2}(v_2)dv_2\right||\psi_2^{x_2,R_2}(u_2)|du_2\less \int_{\R}|\psi_2^{x_2,R_2}(u_2)|du_2\less R_2.
\end{align*}
Therefore $\widetilde C_\Gamma$ satisfies \eqref{WBP2}.  To prove \eqref{WBP3}, fix $x_1,x_2,y_2\in\R$, $R_1,R_2>0$, and $\varphi_j,\psi_j$ for $j=1,2$ as above, but furthermore assume (without loss of generality) that $\gamma_1'\psi^{x_1, R_1}_1$ has mean zero.  Since $|x_1-y_1|>4R_1$
\begin{align*}
&\left|\<M_b\widetilde{\mathcal C}_\Gamma M_b(\varphi_1^{y_1,R_1}\otimes \varphi_2^{x_2,R_2}),\psi_1^{x_1,R_1}\otimes \psi_2^{x_2,R_2}\>\right|\\
&\hspace{1cm}\leq\lim_{ t_1, t_2\rightarrow0^+}\int_{\R^2}|q_{ t_1}(\gamma_1(u_1)-\gamma_1(v_1))-q_{ t_1}(\gamma_1(x_1)-\gamma_1(v_1))|\,|\varphi_1^{y_1,R_1}(v_1)\psi_1^{x_1,R_1}(u_1)\gamma_1'(v_1)\gamma_1'(u_1)|dv_1\,du_1\\
&\hspace{3.5cm}\times\left|\int_{\R^2}\log\((\gamma_2(u_2)-\gamma_2(v_2))^2\)(\varphi_2^{
y_2,R_2})'(v_2)\psi_2^{x_2,R_2}(u_2)\gamma_2'(u_2)dv_2\,du_2\right|\\
&\hspace{1cm}=\lim_{ t_1, t_2\rightarrow0^+}\widetilde A_{ t_1}\times B_{ t_2}.
\end{align*}
By the support properties of $\varphi_1$ and $\psi_1$, we may assume that $|y_1-v_1|\leq R_1$ and $|x_1-u_1|\leq R_1$ to estimate the following part of the integrand from $\widetilde A_{ t_1}$:
\begin{align*}
&|q_{ t_1}(\gamma_1(u_1)-\gamma_1(v_1))-q_{ t_1}(\gamma_1(x_1)-\gamma_1(v_1))|\\
&\hspace{.5cm}=\left|\frac{(\gamma_1(u_1)-\gamma_1(v_1))(\gamma_1(x_1)-\gamma_1(v_1))^2-(\gamma_1(x_1)-\gamma_1(v_1))(\gamma_1(u_1)-\gamma_1(v_1))^2}{[(\gamma_1(x_1)-\gamma_1(v_1))^2+ t_1^2][(\gamma_1(u_1)-\gamma_1(v_1))^2+ t_1^2]}\right.\\
&\hspace{3.5cm}\left.+\frac{(\gamma_1(u_1)-\gamma_1(v_1)) t_1^2-(\gamma_1(x_1)-\gamma_1(v_1)) t_1^2}{[(\gamma_1(x_1)-\gamma_1(v_1))^2+ t_1^2][(\gamma_1(u_1)-\gamma_1(v_1))^2+ t_1^2]}\right|\\
&\hspace{.5cm}\leq\frac{|\gamma_1(u_1)-\gamma_1(v_1)|\,|\gamma_1(x_1)-\gamma_1(v_1)|\,|\gamma_1(x_1)-\gamma_1(u_1)|}{[(\gamma_1(u_1)-\gamma_1(v_1))^2+ t_1^2][(\gamma_1(x_1)-\gamma_1(v_1))^2+ t_1^2]}\\
&\hspace{3.5cm}+ t_1^2\frac{|\gamma_1(u_1)-\gamma_1(x_1)|}{|(\gamma_1(u_1)-\gamma_1(v_1))^2+ t_1^2|\,|(\gamma_1(x_1)-\gamma_1(v_1))^2+ t_1^2|}\\
&\hspace{.5cm}\less\frac{|u_1-v_1|\,|x_1-v_1|\,|x_1-u_1|}{|u_1-v_1|^2|x_1-v_1|^2}+\frac{|x_1-u_1|}{|x_1-v_1|^2}\less\frac{R_1}{|x_1-y_1|^2}.
\end{align*}
In the last line, we use that $|x_1-y_1|>R_1/4$, $|x_1-u_1|\leq R_1$, $|y_1-v_1|\leq R_1$,
\begin{align*}
&|u_1-v_1|\geq|x_1-y_1|/2,\hspace{.25cm}\text{and}\hspace{.25cm}|x_1-v_1|\geq|x_1-y_1|/2.
\end{align*}
It easily follows that
\begin{align*}
\widetilde A_{t_1}\less\frac{R_1}{|x_1-y_1|^2}\int_{\R^2}|\varphi_1^{y_1,R_1}(v_1)\psi_1^{x_1,R_1}(u_1)|dv_1\,du_1\less\frac{R_1^3}{|x_1-y_1|^2}=\frac{R_1}{(R_1^{-1}|x_1-y_1|)^2},
\end{align*}
as required in \eqref{WBP3} with $n_1=\gamma=1$.

This verifies the first mixed weak boundedness properties \eqref{WBP2} and \eqref{WBP3} for $\mathcal C_\Gamma$, and the other two conditions follow by symmetry.
\end{proof}

\begin{proposition}\label{p:Tb}
Assume $\Gamma$ satisfies the hypotheses of Theorem \ref{t:ext}.  The operator $\widetilde{\mathcal C}_\Gamma$ satisfies the $Tb=T^*\tilde b=0$ conditions with $b(x)=\tilde b(x)=\gamma_1'(x_1)\gamma_2'(x_2)$ for $x=(x_1,x_2)\in\R^2$.
\end{proposition}

\begin{proof}
Let $\eta_R\in C_0^\infty(\R^{n_1})$ be as above, $\varphi_1,\psi_1\in C_0^\infty(\R^{n_1})$, and $\psi_2\in C_0^\infty(\R^{n_2})$ such that $\gamma_1'\psi_1$ and $\gamma_2'\psi_2$ have mean zero.  We use Proposition \ref{p:integralrep} to compute
\begin{align*}
&\<\widetilde{\mathcal C}_\Gamma(\gamma_1'\eta_R\otimes\gamma_2'\varphi_2),\gamma_1'\psi_1\otimes\gamma_2'\psi_2\>=\frac{1}{4\pi^2}\int_{\R^4}\log\((\gamma_1(x_1)-\gamma_1(y_1))^2\)\log\((\gamma_2(x_2)-\gamma_2(y_2))^2\)\\
&\hspace{4cm}\times(\eta_R)'(y_1)\varphi_2'(y_2)\psi_1(x_1)\psi_2(x_2)\gamma_1'(x_1)\gamma_2'(x_2)dy\,dx\\
&\hspace{.5cm}=\frac{1}{4\pi^2}\int_{\R^4}\log\((\gamma_1(x_1)-\gamma_1(Ry_1))^2\)\log\((\gamma_2(x_2)-\gamma_2(y_2))^2\)
\eta'(y_1)\varphi_2'(y_2)\psi_1(x_1)\psi_2(x_2)
\gamma_1'(x_1)\gamma_2'(x_2)dy\,dx\\
&\hspace{.5cm}=\int_{\R^2}F_R(x_1)\(\int_{\R}\log\((\gamma_2(x_2)-\gamma_2(y_2))^2\) \varphi_2'(y_2)dy_2\)\psi_1(x_1)\psi_2(x_2)\gamma_1'(x_1)\gamma_2'(x_2)dx,\\
&\hspace{2cm}\text{ where }\hspace{.25cm}F_R(x_1)=\int_{\R}\log\((\gamma_1(x_1)-\gamma_1(Ry_1))^2\) \eta'(y_1)dy_1.
\end{align*} 
Since $\eta\in C_0^\infty(\R)$, it follows that $\eta'$ has mean zero.  Note also that $Re(c_1)=1$ since $\gamma_1(x_1)=x_1+iL_1(x_1)$ and $L_1$ is real-valued, so $\log(y_1^2c_1^2)$ is well defined for $y_1\neq0$.  Recall the definition of $c_1$ in the hypotheses of Theorem \ref{t:ext}.  Hence we can also write $F_R(x_1)$ in the following way.
\begin{align*}
F_R(x_1)&=\int_{\R}\[\log\((\gamma_1(x_1)-\gamma_1(Ry_1))^2\)-\log\(R^2\) \]\eta'(y_1)dy_1=\int_{\R}\log\(\frac{(\gamma_1(x_1)-\gamma_1(Ry_1))^2}{R^2}\)\eta'(y_1)dy_1.
\end{align*}
Now we note that for all $x_1\in\R$ and $y_1\neq0$
\begin{align*}
\lim_{R\rightarrow\infty}\log\(\frac{(\gamma_1(x_1)-\gamma_1(Ry_1))^2}{R^2}\)=\lim_{R\rightarrow\infty}\log\(y_1^2\frac{(\gamma_1(x_1)-\gamma_1(Ry_1))^2}{y_1^2R^2}\)=\log(y_1^2c_1^2).
\end{align*}
Recall that we have assumed $\gamma_1(u_1)/u_1\rightarrow c_1$ as $|u_1|\rightarrow\infty$.  For $R$ large enough so that $\supp(\psi_1)\subset B(0,R/2)$, it follows that for $x_1\in\supp(\psi_1)$ and $y_1\in\supp(\eta')\subset B(0,2)\backslash B(0,1)$
\begin{align*}
\frac{|\gamma_1(x_1)-\gamma_1(Ry_1)|^2}{R^2}&\geq(1-\lambda_1^2)\frac{|x_1-Ry_1|^2}{R^2}\geq(1-\lambda_1^2)\frac{R^2-|x_1|^2}{R^2}\geq1-\lambda_1^2.
\end{align*}
We also have
\begin{align*}
\frac{|\gamma_1(x_1)-\gamma_1(Ry_1)|^2}{R^2}&\leq\frac{4|x_1-Ry_1|^2}{R^2}\leq\frac{4|x_1|^2}{R^2}+4|y_1|^2\leq20
\end{align*}
Therefore
\begin{align*}
\left|\log\(\frac{(\gamma_1(x_1)-\gamma_1(Ry_1))^2}{R^2}\)\eta'(y_1)\right|\less\eta'(y_1).
\end{align*}
Then by dominated convergence,
\begin{align*}
\lim_{R\rightarrow\infty}F_R(x_1)=\int_{\R}\log(y_1^2c_1^2) \eta'(y_1)dy_1=c.
\end{align*}
Now $F_R(x_1)\rightarrow c$ for some constant $c\in\C$, which does not depend on $x_1$.  Since $F_R(x_1)$ is bounded independent of $x_1$, we apply dominated convergence again to conclude
\begin{align*}
\lim_{R\rightarrow\infty}\<\widetilde{\mathcal C}_\Gamma(\gamma_1'\eta_R\otimes\gamma_2'\varphi_2),\gamma_1'\psi_1\otimes\gamma_2'\psi_2\>&=\int_{\R^2}c\(\int_{\R}\log\((\gamma_2(x_2)-\gamma_2(y_2))^2\) \varphi_2'(y_2)dy_2\)\psi_1(x_1)\psi_2(x_2)\gamma_1'(x_1)\gamma_2'(x_2)dx\\
&\hspace{-2cm}=c\(\int_\R \psi_1(x_1)\gamma_1'(x_1)dx_1\)\(\int_{\R^2}\log\((\gamma_2(x_2)-\gamma_2(y_2))^2\) \varphi_2'(y_2)\psi_2(x_2)\gamma_2'(x_2)dy_2\,dx_2\)=0.
\end{align*}
Here we use that $\gamma_1'\psi_1$ has mean zero.  By symmetry, this holds when $\gamma'_1\varphi_1$ has mean zero in place of $\gamma_1'\psi_1$.  Hence the $\widetilde{\mathcal C}_\Gamma(b)=0$ condition is satisfied, and the adjoint condition follows by symmetry.
\end{proof}

By Theorem \ref{t:Tbtheorem}, it follows that $\widetilde{\mathcal C}_\Gamma$ can be extended to a bounded linear operator on $L^p(\R^2)$ for $1<p<\infty$.  Hence $\mathcal C_\Gamma$ can be defined for $g\in L^p(\Gamma)$ for $1<p<\infty$, and for $g\in L^p(\Gamma)$, it follows that
\begin{align*}
||\mathcal C_\Gamma g||_{L^p(\Gamma)}^p&=\int_{\R^2}|\widetilde {\mathcal C}_\Gamma M_b(g\circ\gamma)(x)|^p|\gamma_1'(x_1)\gamma_2'(x_2)|dx\\
&\hspace{0cm}\leq||\gamma_1'||_{L^\infty}||\gamma_2'||_{L^\infty}||\widetilde{\mathcal C}_\Gamma||_{L^p,L^p}^p\int_{\R^2}|(g\circ\gamma)(x)|^pdx\\
&\hspace{0cm}\leq4||(\gamma_1')^{-1}||_{L^\infty}||(\gamma_2')^{-1}||_{L^\infty}||\widetilde{\mathcal C}_\Gamma||_{L^p,L^p}^p\int_{\R^2}|g(x)|^p|\gamma_1'(x_1)\gamma_2'(x_2)|dx\leq4||\widetilde{\mathcal C}_\Gamma||_{L^p,L^p}^p||g||_{L^p(\Gamma)}^p.
\end{align*}
Furthermore for $f\in C_0^\infty(\R^2)$, there exists a constant $C_{f,p}>0$ such that 
\begin{align*}
|\widetilde{\mathcal C}_tM_bf(x)|^p&\leq C_{f,p}\(\chi_{|x_1|\leq2R_0}+\frac{1}{|x_1|^p}\chi_{|x_1|>2R_0}\)\(\chi_{|x_2|\leq2R_0}+\frac{1}{|x_2|^p}\chi_{|x_2|>2R_0}\),
\end{align*}
where $R_0$ is large enough so that $\supp(f)\subset B(0,R_0/2)$.  Then by dominated convergence, it follows that
\begin{align*}
\lim_{t_1,t_2\rightarrow0^+}\widetilde{\mathcal C}_tM_bf&=\widetilde{\mathcal C}_\Gamma M_bf\;\;\text{ in }\;\;L^p(\R^2).
\end{align*}
One can argue by density to verify that $\widetilde{\mathcal C}_\Gamma$ extends to all of $L^p(\R^2)$  and that $\widetilde{\mathcal C}_tf\rightarrow\widetilde{\mathcal C}_\Gamma f$ in $L^p(\R^2)$ for $f\in L^p(\R^2)$ as $t_1,t_2\rightarrow0^+$ for all $1<p<\infty$.

It easily follows that for $g\in L^p(\Gamma)$ where $1<p<\infty$
\begin{align*}
\lim_{t_1,t_2\rightarrow0^+}\mathcal C_tg=\mathcal C_\Gamma g
\end{align*}
in $L^p(\Gamma)$.  We prove now Theorem \ref{t:MaxCgammabounds}, in order to conclude about the almost everywhere convergence $\mathcal{C}_{t}g(z)=\mathcal{C}g(z)$ as $t_1,t_2\to 0^+$. We need the following lemma.

\begin{lemma}\label{semigroup}
Let $z_j, \zeta_j\in\Gamma_j$,$j=1,2$. Then, the following relationship holds
\begin{align*}
  \int_{\Gamma_j}q_{t_j}(z-\xi)p_{s_j}(\xi-\zeta) d\xi=q_{t_j+s_j}(z-\zeta)
\end{align*} 
for every $t_j,s_j\neq 0$.
\end{lemma}
\begin{proof}
It is not hard to prove that the conclusion follows using residue theorems similar to the proof of Lemma \ref{l:Poissonconverge}.
\end{proof}

Finally, we prove Theorem \ref{t:MaxCgammabounds}.
\begin{proof}
Fix $p\in(1,\infty)$ and suppose for the moment that 
$
\mathcal{C}_{t}g=\mathcal{C}_{t_1,t_2}g=P_{t_1,t_2}\mathcal{C}_\Gamma g
$
holds whenever $g$ is a function in $L^p(\Gamma)$ and $t_1,t_2\neq 0$ Then, using some of the estimates we used in the proof of Lemma \ref{l:Poissonconverge}, it is not hard to prove that 
\begin{align*}
\mathcal{C}^*_\Gamma g(\gamma(x))&=\sup_{t_1,t_2>0}|P_t\mathcal{C}_\Gamma g\circ \gamma(x)|\leq \mathcal{M_S}(\mathcal{C}_\Gamma g\circ\gamma)(x),
\end{align*}
where $\mathcal{M_S}$ is the biparameter strong maximal operator. Thus, the $L^p$ boundedness of the maximal operator $\mathcal{C}^{\ast}_\Gamma$ follows from the boundedness of the operators $\mathcal{M_S}$ and $\mathcal{C}_\Gamma$.

It remains to prove the equality $\mathcal{C}_tg=P_t\mathcal{C}_\Gamma g$. Without losing generality, we can suppose that $g\circ\gamma$ is in $C^{\infty}_0(\R^2)$, so that the existence of the pointwise limit $\mathcal{C}_tg(z)$ is guaranteed by Proposition \ref{p:integralrep}.  Thus, using Lemma \ref{semigroup}, we obtain

\begin{align*}
C_tg(z)&=\frac{1}{\pi^2}\int_{\Gamma_1\times\Gamma_2}\frac{z_1-\xi_1}{(z_1-\xi_1)^2+t_1^2}\frac{z_2-\xi_2}{(z_2-\xi_2)^2+t_2^2}g(\xi)d\xi\\
&=\lim_{s_1,s_2\rightarrow0^+}\frac{1}{\pi^2}\int_{\Gamma_1\times\Gamma_2}\frac{z_1-\xi_1}{(z_1-\xi_1)^2+(s_1+t_1)^2}\frac{z_2-\xi_2}{(z_2-\xi_2)^2+(s_2+t_2)^2}g(\xi)d\xi\\
&=\lim_{s_1,s_2\rightarrow0^+}\frac{1}{\pi^4}\int_{(\Gamma_1\times\Gamma_2)^2}\frac{t_1}{(z_1-\zeta_1)^2+t_1^2}\frac{\zeta_1-\xi_1}{(\zeta_1-\xi_1)^2+s_1^2}\frac{t_2}{(z_2-\zeta_2)^2+t_2^2}\frac{\zeta_2-\xi_2}{(\zeta_2-\xi_2)^2+s_2^2}d\zeta g(\xi)d\xi\\
&=\lim_{s_1,s_2\rightarrow0^+}\frac{1}{\pi^2}\int_{\Gamma_1\times\Gamma_2}\frac{t_1}{(z_1-\zeta_1)^2+t_1^2}\frac{t_2}{(z_2-\zeta_2)^2+t_2^2}\(\frac{1}{\pi^2}\int_{\Gamma_1\times\Gamma_2}\frac{\zeta_1-\xi_1}{(\zeta_1-\xi_1)^2+s_1^2}\frac{\zeta_2-\xi_2}{(\zeta_2-\xi_2)^2+s_2^2}g(\xi)d\xi\)d\zeta\\
&=\frac{1}{\pi^2}\int_{\Gamma_1\times\Gamma_2}\frac{t_1}{(z_1-\zeta_1)^2+t_1^2}\frac{t_2}{(z_2-\zeta_2)^2+t_2^2} \mathcal{C}_\Gamma g(\zeta)d\zeta=P_{t_1,t_2}C_\Gamma g(z),
\end{align*}
Therefore, the $L^p$ boundedness of the maximal operator $\mathcal{C}^*_\Gamma$ is proved. The almost everywhere pointwise convergence for a general function g in $L^p(\Gamma)$ can be now obtained using the existence of the pointwise limit for smooth functions, the boundedness of $\mathcal{C^\ast}_{\Gamma}$, and a standard argument. See, for example, \cite{G}.
\end{proof}
This completes the proof of the first part of Theorem \ref{t:Cgammabounds}, pertaining to $\mathcal C_\Gamma$.

\section{Bounds for $\mathcal C_\Gamma^{p1}$, $\mathcal C_\Gamma^{p2}$, $\widetilde{\mathcal C}_\Gamma^{p1}$, and $\widetilde{\mathcal C}_\Gamma^{p2}$}

Like in the last section, we define the parameterized versions of $\mathcal C_\Gamma^{p1}$ and $\mathcal C_\Gamma^{p2}$, for $f\in C_0^\infty(\R^2)$ and $x\in\R^2$
\begin{align*}
\widetilde{\mathcal C}_\Gamma^{p1}M_bf(x)&=\lim_{ t_1, t_2\rightarrow0^+}\widetilde{\mathcal C}_\Gamma^{p1}M_bf(x),\hspace{.25cm}\text{where}&	
&\widetilde{\mathcal 
C}_t^{p1}M_bf(x)=\int_{\R^2}q_{ t_1}(\gamma_1(x_1)-\gamma_1(y_1))p_{ t_2}(\gamma_2(x_2)-\gamma_2(y_2))f(y)b(y)dy,\\
\widetilde{\mathcal C}_\Gamma^{p2}M_bf(x)&=\lim_{ t_1, t_2\rightarrow0^+}\widetilde{\mathcal C}_t^{p2}M_bf(x),\hspace{.25cm}\text{where}&
&\widetilde{\mathcal C}_t^{p2}M_bf(x)=\int_{\R^2}p_{ t_1}(\gamma_1(x_1)-\gamma_1(y_1))q_{ t_2}(\gamma_2(x_2)-\gamma_2(y_2))f(y)b(y)dy.
\end{align*}
We prove these bounds by applying the single parameter $Tb$ theorem from \cite{DJS}.  We outline the proof that $\widetilde{\mathcal C}_\Gamma^{p1}$ and $\widetilde{\mathcal C}_\Gamma^{p2}$ are bounded on $L^p(\Gamma)$.  The details can be deciphered from the previous more complicated biparameter versions.  Define for $f_1,f_2:\R\rightarrow\C$ and $x_1,x_2\in\R$
\begin{align*}
&\widetilde{\mathcal C}_{\Gamma_1}M_{\gamma_1'}f_1(x_1)=\lim_{t_1\rightarrow0^+}\int_{\R}q_{t_1}(\gamma_1(x_1)-\gamma_1(y_1))f_1(y_1)\gamma_1'(y_1)dy_1,\\
&\widetilde{\mathcal C}_{\Gamma_2}M_{\gamma_2'}f_2(x_2)=\lim_{t_2\rightarrow0^+}\int_{\R}q_{t_2}(\gamma_2(x_2)-\gamma_2(y_2))f_2(y_2)\gamma_2'(y_2)dy_2.
\end{align*}
The following propositions are routine given the proofs of Propositions \ref{p:integralrep}, \ref{p:WBP}, and \ref{p:Tb}.

\begin{proposition}
Assume $\Gamma$ satisfies the hypotheses of Theorem \ref{t:ext}.  For all $f\in C_0^\infty(\R^2)$ and $x\in\R^2$, 
\begin{align*}
\widetilde{\mathcal C}_\Gamma^{p1}(bf)(x)&=\frac{1}{2\pi}\int_{\R}\log\((\gamma_1(x_1)-\gamma_1(y_1))^2\)\partial_{y_1}f(y_1,x_2)dy_1,\\
\widetilde{\mathcal C}_\Gamma^{p2}(bf)(x)&=\frac{1}{2\pi}\int_{\R}\log\((\gamma_2(x_2)-\gamma_2(y_2))^2\)\partial_{y_2}f(x_1,y_2)dy_2.
\end{align*}
Also, for all $f,g\in C_0^\infty(\R^2)$, the pairings $\<\widetilde{\mathcal C}_\Gamma^{p1}(bf),bg\>$ and $\<\widetilde{\mathcal C}_\Gamma^{p2}(bf),bg\>$ can be realized as any of the following absolutely convergent integrals:
\begin{align*}
\<\widetilde{\mathcal C}_\Gamma^{p1}(bf),bg\>&=\frac{1}{2\pi}\int_{\R^3}\log\((\gamma_1(x_1)-\gamma_1(y_1))^2\)\partial_{y_1}f(y_1,x_2)g(x)b(x)dy_1\,dx,\\
\<\widetilde{\mathcal  C}_\Gamma^{p1}(bf),bg\>&=-\frac{1}{2\pi}\int_{\R^3}\log\((\gamma_1(x_1)-\gamma_1(y_1))^2\)f(y_1,x_2)\partial_{x_1}g(x)b(y_1,x_2)dy_1\,dx,
\end{align*}
\begin{align*}
\<\widetilde{\mathcal C}_\Gamma^{p2}(bf),bg\>&=\frac{1}{2\pi}\int_{\R^3}\log\((\gamma_2(x_2)-\gamma_2(y_2))^2\)\partial_{y_2}f(x_1,y_2)g(x)b(x)dy_2\,dx,\\
\<\widetilde{\mathcal C}_\Gamma^{p2}(bf),bg\>&=-\frac{1}{2\pi}\int_{\R^3}\log\((\gamma_2(x_2)-\gamma_2(y_2))^2\)f(x_1,y_2)\partial_{x_2}g(x)b(x_1,y_2)dy_2\,dx.
\end{align*}
\end{proposition}

\begin{proposition}
Assume $\Gamma$ satisfies the hypotheses of Theorem \ref{t:ext}.  The operator $\widetilde{\mathcal C}_{\Gamma_1}$ and $\widetilde{\mathcal C}_{\Gamma_2}$ satisfies the single parameter weak boundedness property.
\end{proposition}

\begin{proposition}
Assume $\Gamma$ satisfies the hypotheses of Theorem \ref{t:ext}.  The operator $\widetilde{\mathcal C}_{\Gamma_1}$ and $\widetilde{\mathcal C}_{\Gamma_2}$ satisfies the cancellation conditions $\widetilde{\mathcal C}_{\Gamma_1}(\gamma_1')=\widetilde{\mathcal C}_{\Gamma_1}^*(\gamma_1')=\widetilde{\mathcal C}_{\Gamma_2}(\gamma_2')=\widetilde{\mathcal C}_{\Gamma_2}^*(\gamma_2')=0$.
\end{proposition}

Then by the $Tb$ theorem of David-Journ\'e-Semmes \cite{DJS}, it follows that $\widetilde{\mathcal C}_{\Gamma_1}$ and $\widetilde{\mathcal C}_{\Gamma_2}$ are bounded on $L^p(\R)$.  It follows that for $f,g\in C_0^\infty(\R)$
\begin{align*}
\left|\<\widetilde{\mathcal C}_\Gamma^{p1}(bf),bg\>\right|&=\frac{1}{2\pi}\int_\R\left|\lim_{ t_1\rightarrow0^+}\int_{\R^2}\log\((\gamma_1(x_1)-\gamma_1(y_1))^2+ t_1^2\)\partial_{y_1}f(y_1,x_2)g(x)\gamma_1'(x_1)dy_1\,dx_1\right||\gamma_2'(x_2)|dx_2\\
&\hspace{0cm}=\frac{1}{2\pi}\int_\R\left|\lim_{ t_1\rightarrow0^+}\int_{\R^2}q_{ t_1}(\gamma_1(x_1)-\gamma_1(y_1))f(y_1,x_2)\gamma_1'(y_1)g(x)\gamma_1'(x_1)dy_1\,dx_1\right||\gamma_2'(x_2)|dx_2\\
&\hspace{0cm}=\frac{1}{2\pi}\int_\R\left|\<\widetilde{\mathcal C}_{\Gamma_1}(\gamma_1'\,f(\cdot,x_2)),\gamma_1'g(\cdot,x_2)\>\right||\gamma_2'(x_2)|dx_2\\
&\hspace{0cm}\less\int_\R||f(\cdot,x_2)||_{L^p(\R)}||g(\cdot,x_2)||_{L^{p'}(\R)}dx_2\leq||f||_{L^p(\R^2)}||g||_{L^{p'}(\R^2)}.
\end{align*}
Therefore $\widetilde{\mathcal C}_\Gamma^{p1}$ is bounded on $L^p(\R^2)$ for $1<p<\infty$, and by symmetry $\widetilde{\mathcal C}_\Gamma^{p2}$ is as well.  Again it follows that for $f\in L^p(\R^2)$
\begin{align*}
&\lim_{t_1,t_2\rightarrow0^+}\widetilde{\mathcal C}_t^{p1}M_{b}f=\widetilde{\mathcal C}_{\Gamma_1}M_{b}f& &\text{and}& &\lim_{t_1,t_2\rightarrow0^+}\widetilde{\mathcal C}_t^{p2}M_{b}f=\widetilde{\mathcal C}_{\Gamma_2}M_{b}f& &\text{in }L^p(\R^2),&
\end{align*}
and for $g\in L^p(\Gamma)$
\begin{align*}
&\lim_{t_1,t_2\rightarrow0^+}\mathcal C_t^{p1}g=\mathcal C_\Gamma^{p1}g& &\text{and}& &\lim_{t_1,t_2\rightarrow0^+}\mathcal C_t^{p2}g=\mathcal C_\Gamma^{p2}g& &\text{in }L^p(\Gamma).&
\end{align*}
This completes the proof for the $L^p$ convergence.  Similarly, the almost everywhere convergence can be derived with arguments analogous to the ones used in the biparameter situation. We prefer not to report detail again since it is clear by now how to proceed.

\end{spacing}
\nocite{*}
\bibliographystyle{alpha}
\bibliography{bibliography17}

\end{document}